\documentclass{article}
\usepackage{amsmath,amssymb,amscd}
\usepackage{amsthm}
\usepackage{mathptmx}
\newtheorem{thm}{\indent \sc Theorem}[section]
\newtheorem{prop}[thm]{\indent \sc Proposition}
\newtheorem{cor}[thm]{\indent \sc Corollary}
\newtheorem{lem}[thm]{\indent \sc Lemma}
\newtheorem{con}[thm]{\indent \sc Conjecture}

\newtheorem{cond}[thm]{\indent \sc Condition}
\newtheorem{defi}[thm]{\indent \sc Definition}

\newtheorem{rem}[thm]{\indent \sc Remark}
\newtheorem{Qu}[thm]{\indent \sc Question}

\newcommand{\spec}{\operatorname{Spec}}
\newcommand{\et}{\mathrm{\acute{e}t}}
\newcommand{\Nis}{\mathrm{Nis}}
\newcommand{\Zar}{\mathrm{Zar}}

\newcommand{\HO}{\operatorname{H}}

\usepackage[all]{xy}

\begin{document}
\author{Makoto Sakagaito}
\title{
Gersten-type conjecture for henselian local rings of normal crossing varieties}
\date{}
\maketitle
\begin{center}
Indian Institute of Technology Gandhinagar
\\
\textit{E-mail address}: sakagaito43@gmail.com
\end{center}
\begin{abstract}
Let $n\geq 0$ and $r>0$ be integers. 
Let $\mathcal{O}_{X, x}^{h}$ be the henselization of the local ring $\mathcal{O}_{X, x}$
of a scheme $X$ at a point $x\in X$.
For a normal crossing variety $Y$ over the spectrum of a field $k$ of positive characteristic $p>0$,
K.Sato defined an \'{e}tale logarithmic Hodge-Witt sheaf $\lambda^{n}_{Y, r}$ (cf. \cite[p.726, Definition 3.1.1 (1)]{SaL})
on the \'{e}tale site $Y_{\et}$
which agrees with $W_{r}\Omega^{n}_{Y, \log}$ (cf. \cite{I}) in the case where $Y$ is smooth over
$\spec(k)$. 
In this paper, we prove the Gersten-type conjecture for 
\'{e}tale sheaves which satisfy some properties 
over $\mathcal{O}_{Y, y}^{h}$. 
For example, $\lambda_{Y, r}^{n}$ and $\mu_{l}^{\otimes n}$
satisfy these properties 
where $\mu_{l}$ is the \'{e}tale sheaf of $l$-th roots of unity
for an integer $l$ which is prime to the characteristic of $Y$.
Let $B$ be a discrete valuation ring of mixed characteristic $(0, p)$
and $\mathfrak{X}$ a semistable family over $\spec(B)$.
Suppose that $B$ contains $p$-th roots of unity.
As an application of the Gersten-type conjecture for $\lambda^{n}_{r}$, 
we prove the relative version of the Gersten-type conjecture
for the $p$-adic \'{e}tale Tate twist $\mathfrak{T}_{1}(n)$ (cf. \cite[pp.537--538, Definition 4.2.4]{SaP}) over 
$\mathcal{O}_{\mathfrak{X}, x}^{h}$. 
Moreover, we prove a generalization of Artin's theorem 
(cf. \cite[p.98, Th\'{e}or\`{e}me (3.1)]{Gr}) about the Brauer groups.
\end{abstract}

\section{Introduction}

Let $X$ be an equidimensional scheme and
noetherian of dimension $d$,
$X^{(s)}$ the set of points $x\in X$ of codimension $s$
and $\mathcal{F}^{\bullet}$ a bounded below complex of sheaves of abelian groups
on 
the \'{e}tale site $X_{\et}$. 
Let us denote
\begin{equation*}
\HO^{s+t}_{x}(X_{\et},
\mathcal{F}^{\bullet}
)
:=
\displaystyle\lim_{\substack{
\to \\
x\in U}
}
\HO^{s+t}_{\bar{\{x\}}\cap U}(
U_{\et},
\mathcal{F}^{\bullet}
)
\end{equation*}
where $U$ runs through open neighborhoods of $x$ in $X$.
Then we have a coniveau spectral sequence
\begin{align}\label{Inconisp}
E^{s, t}_{1}
=
\displaystyle
\bigoplus_{x\in X^{(s)}}
\HO^{s+t}_{x}(
X_{\et},
\mathcal{F}^{\bullet}
)
\Rightarrow
E^{s+t}
=
\HO^{s+t}_{\et}(
X,
\mathcal{F}^{\bullet}
)
\end{align}
and it induces the Cousin complex
\begin{align}\label{Gercom}
0
\to 
\HO^{n}_{\et}(
X, \mathcal{F}^{\bullet}
)
\to
\displaystyle\bigoplus_{
x\in X^{(0)}
}
\HO^{n}_{x}(
X_{\et},
\mathcal{F}^{\bullet}
)
\to
\displaystyle\bigoplus_{
x\in X^{(1)}
}
\HO^{n+1}_{x}(
X_{\et},
\mathcal{F}^{\bullet}
)
\to
\cdots
\end{align}
(cf. \cite[Part 1, \S 1]{C-H-K}, see also \cite[Theorem 2.1]{Sak5}).
We call the exactness of the Cousin complex (\ref{Gercom}) 
the \textbf{Gersten-type conjecture}
for \'{e}tale (hyper)cohomology with values in $\mathcal{F}^{\bullet}$.

Let $Y$ be a normal crossing variety over the spectrum of a field $k$
of positive characteristic $p>0$, 
that is,
$Y$ is everywhere \'{e}tale locally isomorphic to
\begin{align*}
\spec(k[T_{0}, T_{1}, \cdots, T_{N}]/(T_{0}T_{1}\cdots T_{a}))    
\end{align*}
for some integer $a$ with $0\leq a\leq N=\operatorname{dim}(Y)$
(cf. \cite[pp.180--181, Definition 2.1]{SaR}).
Then K.Sato defined the logarithmic Hodge-Witt sheaves
\begin{align*}
\lambda^{n}_{Y, r}
:=
\operatorname{Im}\left(
d\log:
(\mathbb{G}_{m, Y})^{\otimes n}
\to
\displaystyle 
\bigoplus_{x\in Y^{(0)}}
i_{y *}W_{r}\Omega^{n}_{y, \log}
\right)
\end{align*}
(cf. \cite[p.726, Definition 3.1.1 (1)]{SaL}) and
\begin{equation*}
\nu^{n}_{Y, r} 
:=
\operatorname{Ker}\left(
\delta:
\displaystyle\bigoplus_{y\in Y^{(0)}}
i_{y *}W_{r}\Omega^{n}_{y, \log}
\to
\displaystyle\bigoplus_{y\in Y^{(1)}}
i_{y *}W_{r}\Omega^{n-1}_{y, \log}
\right)
\end{equation*}
(cf. \cite[p.715, Definition 2.1.1]{SaL}) on $Y_{\et}$ for any integers $n\geq 0$ and $r>0$ which agree with the logarithmic Hodge-Witt sheaves 
$W_{r}\Omega^{n}_{Y, \log}$ (cf. \cite{I}) 
in the case where $Y$
is a smooth scheme over $\spec(k)$.
In general, 
\begin{math}
\lambda^{n}_{Y, r}
\neq
\nu^{n}_{Y, r}
\end{math}
(cf. \cite[p.737, Remark 4.2.3]{SaL}).

In the case where $X$ is the spectrum of the local ring $\mathcal{O}_{Y, y}$ of $Y$ at a point $y\in Y$ and
$\mathcal{F}=\nu^{n}_{X, r}$, 
the Cousin complex (\ref{Gercom}) is exact by \cite[Theorem 1.1]{Sak4}. 
This is an extension of (\cite[Th\'{e}or\`{e}me 1.4]{G-S}, \cite[p.600, Theorem 4.1]{Sh}).

In \S \ref{TGer}, we prove the Gersten-type conjecture for $\lambda_{X, r}^{n}$. Precisely, we prove the following:

\begin{thm}\upshape\label{MaiGCL}(cf. Theorem \ref{Gtcl})
Let $A$ be the henselization of the local ring
$\mathcal{O}_{Y, y}$ of a normal crossing variety $Y$ at a point $y\in Y$. 
Then the sequence
\begin{align*}
0\to \HO^{u}_{\et}(A, \lambda^{n}_{r})
\to
\bigoplus_{x\in \spec(A)^{(0)}}
\HO^{u}_{x}(A_{\et}, \lambda_{r}^{n})
\to
\bigoplus_{x\in \spec(A)^{(1)}}
\HO^{u+1}_{x}(A_{\et}, \lambda_{r}^{n}) 
\to
\cdots
\end{align*}
is exact for any integers $n\geq 0$, $u$ and $r>0$.
\end{thm}

The outline of the proof is as follows. If $u<0$, then we have
\begin{align*}
\HO^{u+s}_{x}(A_{\et}, \lambda_{r}^{n})
=0
\end{align*}
for $x\in \spec(A)^{(s)}$
by Proposition \ref{hvan}. So Theorem \ref{MaiGCL} holds
for $u<0$.
In the case where $u\geq 1$,
Theorem \ref{MaiGCL} follows from the following theorem:

\begin{thm}\upshape\label{Ex&Ger}(cf. Proposition \ref{cst}, Theorem \ref{1gerl})
Let $Y$ be a normal crossing variety over 
the spectrum of a field $k$.
Let 
\begin{math}
Y_{1}, \cdots, Y_{a}
\end{math}
be the irreducible components of $Y$.

Let $(\mathcal{F}_{X}, N)$ be a pair of an \'{e}tale sheaf of abelian groups
on an object $X$ in $\mathcal{V}$ and a non-negative integer
which satisfies 
Condition \ref{cdpb}. 
Here $\mathcal{V}$ is the category whose
objects are normal crossing varieties over the spectrum of
a field $k$
and
whose morphisms are morphisms of schemes.
Suppose that
$\operatorname{cd}_{q}(k)<\infty$ for a prime number $q$.
For an abelian group $M$,
$M_{T_{q}}$ denotes the $q$-power torsion subgroup of $M$.
Put 
\begin{align*}
T:=
\left\{
\begin{array}{ll}
0  &  \textrm{if $\mathcal{F}_{Y}$ is a sheaf of torsion groups}  \\ [5pt]
1   &  \textrm{otherwise}
\end{array}
\right.
\end{align*}
and $N_{1}=\operatorname{max}\{N, T\}$.

Then $\HO_{y}^{u}(Y_{\et}, \mathcal{F}_{Y})$ is a torsion group for $y\in Y^{(s)}$
and $u\geq s+T$.
Moreover, we have the followings:

\begin{itemize}
\item Property $\operatorname{P}_{1}(a)$:
Let $s\geq 0$ be an integer and
$i: Z\to Y$ a closed immersion with
\begin{equation*}
Z
=
\displaystyle\bigcup^{a-1}_{m=1} Y_{m}.
\end{equation*}
Then the homomorphism
\begin{equation*}
\HO^{N^{\prime}+s}_{y}(
Z_{\et},
i^{*}\mathcal{F}_{Y}
)_{T_{q}}    
\xrightarrow{\sim}
\HO^{N^{\prime}+s}_{y}(
Z_{\et},
\mathcal{F}_{Z}
)_{T_{q}}
\end{equation*}
is an isomorphism for
$N^{\prime}\geq N_{1}$ and
$y\in Y^{(s)}\cap Z$.
\item Property $\operatorname{P}_{2}(a)$:
Let the notations be the same as above
and $j: U\hookrightarrow Y$ the open complement
$Y\setminus Z$.
Then 
a distinguished triangle
\begin{align}\label{0Indisji}
\cdots\to  
j_{!}\mathcal{F}_{U}
\to
\mathcal{F}_{X}
\to
i_{*}i^{*}\mathcal{F}_{X}
\to\cdots
\end{align}
induces an exact sequence
\begin{equation}\label{0Inindij}
0\to
\HO^{N^{\prime}+s}_{y}(
Y_{\et}, j_{!}\mathcal{F}_{U}
)_{T_{q}}
\to
\HO^{N^{\prime}+s}_{y}(
Y_{\et}, 
\mathcal{F}_{Y}
)_{T_{q}}
\to 
\HO^{N^{\prime}+s}_{y}(Z_{\et},
\mathcal{F}_{Z})_{T_{q}}
\to 0
\end{equation}
for $y\in Y^{(s)}\cap Z$ and $N^{\prime}\geq N_{1}$. 

\item Property $\operatorname{P}_{3}(a)$: Let $A$ be the henselization of the local ring
$\mathcal{O}_{Y, y}$ of $Y$ at a point of $y\in Y$. Then
the sequence
\begin{align*}
0\to \HO^{N^{\prime}}_{\et}(A, \mathcal{F}_{\spec(A)})_{T_{q}}
\to
\bigoplus_{x\in \spec(A)^{(0)}}
\HO^{N^{\prime}}_{x}(A_{\et}, 
\mathcal{F}_{\spec(A)})_{T_{q}}
\\
\to
\bigoplus_{x\in \spec(A)^{(1)}}
\HO^{N^{\prime}+1}_{x}(A_{\et}, 
\mathcal{F}_{\spec(A)})_{T_{q}} 
\to
\cdots
\end{align*}
is exact for $N^{\prime}\geq N_{1}$.
\end{itemize}
\end{thm}

By Remark \ref{ExaCond}, $(\mathcal{F}_{X}, N)=(\lambda^{n}_{X, r}, 1)$
satisfies Condition \ref{cdpb}
and so Theorem \ref{MaiGCL} holds for $u\geq 1$ by Property $\operatorname{P}_{3}(a)$
of Theorem \ref{Ex&Ger}.
Then we are able to prove Theorem \ref{MaiGCL} for $u=0$
by using the coniveau spectral sequence (\ref{Inconisp}). 
This concludes the outline of the proof of Theorem \ref{MaiGCL}.

Let $l$ be an integer which is prime to
$\operatorname{char}(k)$ and
$\mu_{l}$ the \'{e}tale sheaf of $l$-th roots of unity.
Then $(\mathcal{F}_{X}, N)=(\mu_{l}^{\otimes n}, 0)$
satisfies Condition \ref{cdpb} (cf. Remark \ref{ExaCond}). As an application of Theorem \ref{Ex&Ger},
we also have the following:

\begin{cor}\upshape(cf. Corollary \ref{exBGO}) 
Let $A$ be the henselization of the local ring
$\mathcal{O}_{Y, y}$ of a normal crossing variety $Y$ 
over the spectrum of a field $k$
at a point $y\in Y$
and
$l$ an integer which is prime to
$\operatorname{char}(k)$. 
Then the sequence
\begin{align*}
0\to \HO^{u}_{\et}(A, \mu^{\otimes n}_{l})
\to
\bigoplus_{x\in \spec(A)^{(0)}}
\HO^{u}_{x}(A_{\et}, \mu_{l}^{\otimes n})
\to
\bigoplus_{x\in \spec(A)^{(1)}}
\HO^{u+1}_{x}(A_{\et}, \mu_{l}^{\otimes n}) 
\to
\cdots
\end{align*}
is exact for any integers $n\geq 0$ and $u$.    
\end{cor}

This is an extension of the Bloch-Gabber-Ogus theorem (cf. \cite[Corollary 2.2.2]{C-H-K}).

In the proof of Theorem \ref{Ex&Ger}, 
we use Lemma \ref{CAL}.
The proof of Theorem \ref{Ex&Ger} is similar to 
that of Lemma \ref{CAL}
and we prove them by 
downward
induction on $N^{\prime}$
and induction on the dimension of a scheme.
We use Propositions \ref{cdpt} and \ref{exind}
to apply downward induction on $N^{\prime}$ and
we are able to prove them
by 
using Lemma \ref{borTq}.
We use Lemma \ref{DGC} to apply induction on the dimension 
of a scheme.
See \S \ref{OutMai} below for more detailed outline of the proof.

\vspace{2.0mm}

As another application of Theorem \ref{Ex&Ger},
we obtain the relative version of the Gersten-type conjecture in mixed characteristic cases.

Let $B$ be a discrete valuation ring of mixed characteristic $(0, p)$ and 
$K$ the quotient field of $B$. Let $\mathfrak{X}$ be a semistable family
over $\spec(B)$, that is, a regular scheme of pure dimension which is flat
of finite type over $\spec(B)$,
the generic fiber 
$\mathfrak{X}\otimes_{B}K$ is smooth over $\spec(K)$
and the special fiber $Y$ of $\mathfrak{X}$ is a reduced divisor with normal
crossings on $\mathfrak{X}$.
Let $i: Y\to \mathfrak{X}$ be the inclusion of the closed fiber of $\mathfrak{X}$
and
$j: \mathfrak{X}\otimes_{B}K\to \mathfrak{X}$
the inclusion of the generic fiber of $\mathfrak{X}$.
Then $p$-adic \'{e}tale Tate twist is defined as follows:
\begin{defi}\upshape(\textbf{$p$-adic \'{e}tale Tate twist})\label{DefT}
(cf. \cite[pp.537--538, Definition 4.2.4]{SaP}, 
\cite[p.187, Remark 3.7]{SaR})
Let the notations be the same as above and $r$ any positive integer. For $n=0$,
\begin{equation*}
\mathfrak{T}_{r}(0)
:=
\mathbb{Z}/p^{r}.
\end{equation*}
For $n\geq 1$, 
$\mathfrak{T}_{r}(n)$ is defined as a complex which fits into
the distinguished triangle
\begin{align*}
i_{*}\nu_{Y, r}^{n-1}[-n-1]
\to
\mathfrak{T}_{r}(n)
\to
\tau_{\leq n}Rj_{*}\mu_{p^{r}}^{\otimes n}
\xrightarrow{\sigma^{n}_{\mathfrak{X}, r}}
i_{*}\nu_{Y, r}^{n-1}[-n]
\end{align*}
where $\mu_{p^{r}}$ is the sheaf of $p^{r}$-th roots of unity
and $\sigma^{n}_{\mathfrak{X}, r}$ is induced by
the homomorphism
$\sigma: i^{*}R^{n}j_{*}\mu_{p^{r}}^{\otimes n}\to \nu^{n-1}_{Y, r}$ 
which is defined in 
\cite[p.186, Theorem 3.4]{SaR}.
\end{defi}

Let 
$\mathbb{Z}(n)_{\et}$ and $\mathbb{Z}(n)$ be Bloch's cycle
complex for the \'{e}tale and Zariski topology, respectively 
(cf. \cite[p.779]{Ge}). For a positive integer $m$,
$\mathbb{Z}/m(n)_{\et}$ 
(resp. $\mathbb{Z}/m(n)$) denotes
$\mathbb{Z}(n)_{\et}\otimes \mathbb{Z}/m\mathbb{Z}$
(resp. $\mathbb{Z}(n)\otimes \mathbb{Z}/m\mathbb{Z}$).
By \cite[pp.209--210, Remark 7.2]{SaR} and \cite[Proposition 1.5]{Sak4}, 
we have an isomorphism
\begin{equation*}
\tau_{\leq n+1}\left(
\mathbb{Z}/p^{r}(n)_{\et}
\right)
\simeq 
\mathfrak{T}_{r}(n)
\end{equation*}
in
$D^{b}(\mathfrak{X}, \mathbb{Z}/p^{r}\mathbb{Z})$
which is the derived category of bounded complexes
of \'{e}tale $\mathbb{Z}/p^{r}\mathbb{Z}$-sheaves
on $\mathfrak{X}_{\et}$. If $\mathfrak{X}$ is
smooth over $\spec(B)$,
then trancation is unnecessary by \cite[p.786, Corollary 4.4]{Ge}.
By an application of Theorem \ref{MaiGCL}, we have the relative version of the
Gersten-type conjecture for $\mathfrak{T}_{r}(n)$ as follows:

\begin{thm}\upshape(cf. Theorem \ref{RelGer})\label{Intrelsupp}
Let $B$ be a discrete valuation ring of mixed characteristic $(0, p)$
and $\pi$ a prime element of $B$. Let $\mathfrak{X}$ be a semistable
family over $\spec(B)$ and $R$ the henselization of the local ring $\mathcal{O}_{\mathfrak{X}, x}$ of $\mathfrak{X}$ at a 
point $x$ of the closed fiber of $\mathfrak{X}$.
Put $Z=\spec(R/(\pi))$.
Suppose that $B$ contains $p^{r}$-th roots of unity.
Then the sequence
\begin{align*}
0\to 
\HO_{Z}^{q}(R_{\et}, \mathfrak{T}_{r}(n))
\to
\displaystyle\bigoplus_{z\in Z^{(0)}}
\HO_{z}^{q}(R_{\et}, \mathfrak{T}_{r}(n))
\to
\displaystyle\bigoplus_{z\in Z^{(1)}}
\HO_{z}^{q+1}(R_{\et}, \mathfrak{T}_{r}(n))
\to
\cdots
\end{align*}
is exact for $q\geq n+2$.
\end{thm}

In \S \ref{TrelGtM}, we observe a relation between the Gersten-type conjecture and
motivic cohomology groups. 
As an application of the Beilinson-Lichtenbaum conjecture (cf. \cite[p.774, Theorem 1.2.2]{Ge} and \cite{V2}), 
we prove the following:

\begin{prop}\upshape\label{InrelGerMot}(cf. Proposition \ref{RGCM})
Let $B$ be a discrete valuation ring 
of mixed characteristic $(0, p)$
and $R$ a local ring 
(resp. a henselian local ring)
of a regular scheme of finite type over $\spec(B)$.
Let $U$ be the generic fiber of $\spec(R)$ and $n$ a non-negative integer.
Suppose that $B$ contains $p$-th roots of unity. 

Then the sequence
\begin{align}\label{InUet}
0\to& 
\HO^{n+1}_{\et}(U, \mu_{p}^{\otimes n})
\to
\displaystyle\bigoplus_{x\in U^{(0)}}
\HO^{n+1}_{\et}(
\kappa(x),
\mu_{p}^{\otimes n}
)
\to
\displaystyle\bigoplus_{x\in U^{(1)}}
\HO^{n}_{\et}(
\kappa(x),
\mu_{p}^{\otimes (n-1)}
) \nonumber
\\
\to&\cdots
\end{align}
is exact 
if and only if
we have isomorphisms
\begin{equation*}
\HO^{n+1}_{\Zar}(U, \mathbb{Z}/p(n))
\simeq 
0
\end{equation*}
and
\begin{equation*}
\HO^{n+1+s}_{\Zar}(
U,
\mathbb{Z}/p(n)
)    
\simeq
\HO^{n+1+s}_{\Zar}(
U,
\mathbb{Z}/p(n+1)
)
\end{equation*}
for any integer $s\geq 1$. 
\end{prop}

By Proposition \ref{InrelGerMot}, the sequence (\ref{InUet}) is exact for any integer $n\geq 0$
if and only if
\begin{equation}\label{conGerU}
\HO_{\Zar}^{u}(
U,
\mathbb{Z}/p(n))=0    
\end{equation}
for any integers $n\geq 0$ and $u>n$.
In the case where $R$ is (the henselization of) the local ring $\mathcal{O}_{\mathfrak{X}, x}$
of a semistable family $\mathfrak{X}$ over $\spec(B)$,
the equation (\ref{conGerU}) holds
for any integers $n\geq 0$ and $u>n+1$
if and only if the equation
\begin{equation*}
\HO^{u}_{\Zar}(
R,
\mathbb{Z}/p(n)
)    
=0
\end{equation*}
holds for any integers $n\geq 0$ and $u>n+1$ by Proposition \ref{InrelGerMot} 
and \cite[Proposition 2.1]{Sak4} 
(cf. Remark \ref{revanimot}).
By an application of Theorem \ref{Intrelsupp} and Proposition \ref{InrelGerMot},
we have the following:

\begin{cor}\upshape\label{Int2GerT}(cf. Corollary \ref{2GerT})
Let $B$ be a discrete valuation ring of mixed characteristic $(0, p)$ 
and $R$ the henselization of a semistable family over $\spec(B)$. 
Suppose that $\operatorname{dim}(R)=2$ and $B$ contains $p$-th roots of unity.
Then the sequence
\begin{align*}
0
\to  
&\HO^{s}_{\et}(
R, \mathfrak{T}_{1}(n)
)
\to
\bigoplus_{x\in \spec(R)^{(0)}}
\HO^{s}_{x}(
R_{\et},
\mathfrak{T}_{1}(n)
)
\to
\\
&\bigoplus_{x\in \spec(R)^{(1)}}
\HO^{s+1}_{x}(
R_{\et},
\mathfrak{T}_{1}(n)
)
\to
\bigoplus_{x\in \spec(R)^{(2)}}
\HO^{s+2}_{x}(
R_{\et},
\mathfrak{T}_{1}(n)
)
\to 0
\end{align*}
is exact for any integer $s\geq 0$.
\end{cor}

In \S \ref{Tcommt}, we compute motivic cohomology groups in global cases.
We prove the following:

\begin{prop}\upshape\label{IntVaniMot} 
(cf. Proposition \ref{VaniNGL}, Proposition \ref{VSC} and Proposition \ref{VaniMtM1})
Let $B$ be a discrete valuation ring of mixed characteristic $(0, p)$ and $\pi$ a prime element of $B$. 
Let 
\begin{equation*}
C=B[T_{0}, T_{1}, \cdots, T_{N}]/(T_{0}\cdots T_{a}-\pi)    
\end{equation*}
for $0\leq a\leq N$. Then we have
\begin{equation}\label{IntvanC}
\HO^{q}_{\Zar}(
C,
\mathbb{Z}/m(n)
)    
=0
\end{equation}
for integers $m>0$ and $q\geq n+1$. Moreover, we have
\begin{equation*}
\HO^{q}_{\Zar}(
C/(\pi),
\mathbb{Z}(n)
)=0    
\end{equation*}
for $q\geq n+1$.
\end{prop}

By using the localization theorem (\cite[p.779, Theorem 3.2.a)]{Ge}),
we are able to prove Proposition \ref{IntVaniMot} by induction on $a$.
The outline of the proof of the isomorphism (\ref{IntvanC}) is as follows:
We have a homomorphism of polynomial rings over $B$
\begin{equation*}
B[T^{\prime}_{0}, \cdots, T^{\prime}_{N}]
\to
B[T_{0}, \cdots, T_{N}]
\end{equation*}
which sends $T^{\prime}_{a-1}$ to $T_{a-1}T_{a}$ and
sends $T^{\prime}_{r}$ to $T_{r}$ for $r\neq a-1$.
Then this homomorphism induces an isomorphism
\begin{align*}
&\left(
B[T^{\prime}_{0}, \cdots, T^{\prime}_{N}]/
(T^{\prime}_{0}\cdots T^{\prime}_{a-1}-\pi)
\right)_{T^{\prime}_{a}}
\\
\simeq
&\left(
B[T_{0}, \cdots, T_{N}]/
(T_{0}\cdots T_{a}-\pi)
\right)_{T_{a}}
\end{align*}
for $a>0$. So, if the homomorphism
\begin{equation}\label{IntredfC}
\HO^{n}_{\Zar}(
C_{T_{a}},
\mathbb{Z}/m(n)
)    
\to
\HO^{n-1}_{\Zar}(
C/(T_{a}),
\mathbb{Z}/m(n-1)
)
\end{equation}
is surjective, the isomorphism (\ref{IntvanC}) follows from the localization theorem and an inductive argument
on $a$.
Moreover, the surjectivity of the homomorphism (\ref{IntredfC}) is proved by using the following:

\begin{prop}\upshape\label{IntGenBL}
(cf. Proposition \ref{SBL} and Proposition \ref{cset})
Let $\mathfrak{X}$ be a regular scheme which is essentially of finite type
over the spectrum of a discrete valuation ring $B$ 
in mixed characteristic $(0, p)$.

\begin{enumerate}
\item Let $i: Z\to \mathfrak{X}$ be a closed immersion of codimension $1$ and 
$Z$ 
an irreducible closed scheme of $\mathfrak{X}$ with $\operatorname{char}(Z)=p$.
Then the canonical map induces a quasi-isomorphism
\begin{equation*}
\tau_{\leq n+2}\left(
\mathbb{Z}(n-1)_{\et}^{Z}[-2]
\right)
\xrightarrow{\simeq}
\tau_{\leq n+2}
Ri^{!}\mathbb{Z}(n)_{\et}^{\mathfrak{X}}.
\end{equation*}
\item
The canonical map induces a quasi-isomorphism
\begin{equation*}
\tau_{\leq n+1}\left(
\mathbb{Z}(n)^{\mathfrak{X}}
\right)
\xrightarrow{\simeq}
\tau_{\leq n+1}\left(
R\epsilon_{*}
\mathbb{Z}(n)_{\et}^{\mathfrak{X}}
\right)
\end{equation*}
where
$\epsilon: \mathfrak{X}_{\et}\to\mathfrak{X}_{\Zar}$
is the canonical map of sites and
$\epsilon_{*}$ is the forgetful functor.
\end{enumerate}
\end{prop}

In the case where $\mathfrak{X}$ is smooth over $\spec(B)$,
\begin{math}
\mathcal{H}^{q}(\mathbb{Z}(n))=0    
\end{math}
for $q\geq n+1$ by \cite[p.786, Corollary 4.4]{Ge}. 
So Proposition \ref{IntGenBL} is an extension of 
\cite[p.774, Theorem 1.2.1 and Theorem 1.2.2]{Ge}.
As an application of Proposition \ref{IntVaniMot}, we obtain the following: 

\begin{prop}\upshape(cf. Corollary \ref{injproEx} and Remark \ref{comia})
Let $B$ be a henselian discrete valuation ring of mixed characteristic $(0, p)$
and $\pi$ a prime element of $B$. Let
\begin{equation*}
\mathfrak{X}
=
\operatorname{Proj}\left(
B[T_{0}, \cdots, T_{N+1}]/
(T_{0}\cdots
T_{a}-\pi T_{N+1}^{a+1}
)
\right)
\end{equation*}
for $0\leq a\leq N$ and
$i: Y\to \mathfrak{X}$ the inclusion of 
the closed fiber of $\mathfrak{X}$. 
Then the homomorphism
\begin{equation*}
\Gamma\left(
\mathfrak{X},
R^{n+1}\alpha_{*}\mu_{l}^{\otimes n}
\right)
\to
\Gamma\left(
Y,
i^{*}R^{n+1}\alpha_{*}\mu_{l}^{\otimes n}
\right)
\end{equation*}
is injective for any integers $n\geq 0$ and $l>0$ with
$(l, p)=1$. Here 
$\alpha: \mathfrak{X}_{\et}\to \mathfrak{X}_{\Nis}$
is the canonical map of sites.    
\end{prop}

In \S \ref{TPrb}, we prove the following as an application of 
the proper base change theorem (\cite[p.224, VI, Corollary 2.7]{M})
and the local-global principle (\cite[Theorem 1.7]{Sak4}):

\begin{thm}\upshape\label{IntNisproper}(cf. Theorem \ref{Nisproper})
Let $B$ be a henselian excellent discrete valuation ring of mixed characteristic 
$(0, p)$
and $\mathfrak{X}$ a semistable family and proper over $\spec(B)$.
Let $i: Y\to \mathfrak{X}$ be the inclusion of 
the closed fiber of $\mathfrak{X}$.
Suppose that
$\operatorname{dim}(\mathfrak{X})=2$
and
$B$ contains $p$-th roots of unity.
Then the homomorphism
\begin{equation*}
\HO^{s}_{\Nis}(\mathfrak{X}, 
R^{t}\alpha_{*}\mathfrak{T}_{1}(n)
)
\xrightarrow{\simeq}
\HO^{s}_{\Nis}(Y, 
i^{*}R^{t}\alpha_{*}\mathfrak{T}_{1}(n)
)
\end{equation*}
is an isomorphism for integers $s\geq 0$ and $t\geq 2$. 
Moreover, we have an isomorphism
\begin{equation*}
\HO^{s}_{\Nis}(\mathfrak{X}, 
R^{n+1}\alpha_{*}\mathfrak{T}_{1}(n)
)
\xrightarrow{\simeq}
\HO^{s}_{\Nis}
(Y, R^{1}\alpha_{*}\lambda_{1}^{n}
)
\end{equation*}
for integers $s\geq 0$ and $n\geq 1$. Thus, the sequence
\begin{align*}
0\to  
\displaystyle\bigoplus_{x\in \mathfrak{X}^{(0)}}
\HO_{x}^{n+r}(\mathfrak{X}_{\et}, \mathfrak{T}_{1}(n))
\to
\displaystyle\bigoplus_{x\in \mathfrak{X}^{(1)}}
\HO_{x}^{n+r+1}(\mathfrak{X}_{\et}, \mathfrak{T}_{1}(n))
\nonumber
\\
\to
\displaystyle\bigoplus_{x\in \mathfrak{X}^{(2)}}
\HO_{x}^{n+r+2}(\mathfrak{X}_{\et}, \mathfrak{T}_{1}(n))
\to
0
\end{align*}
is exact for integers $n\geq 1$ and $r\geq 2$.
\end{thm}

Let the notation be the same as above. Then Theorem \ref{IntNisproper}
relates to a generalization of Artin's theorem that
the Brauer group of $\mathfrak{X}$
is isomorphic to that of $Y$ (cf. Remark \ref{Expgl}).

At the end of the paper, we raise questions 
(see Question \ref{Qpr} and Question \ref{QG})
which relates to Kato conjecture 
(cf. \cite{K}, \cite[p.125, Conjecture 0.2 and Conjecture 0.3]{K-S}).

\subsection*{Notations}
For a scheme $X$,
$X^{(s)}$ denotes the set of points
of codimension $s$ and
$k(X)$ denotes the ring of rational functions on $X$.
For a ring $A$, $k(A)$ denotes $k(\spec(A))$.
For a point $x\in X$, $\kappa(x)$ denotes the residue field
of $x$, $\mathcal{O}_{X, x}$ denotes the local ring of $X$ at $x\in X$
and
$\mathcal{O}^{h}_{X, x}$ denotes the henselization of $\mathcal{O}_{X, x}$.

$X_{\et}$, $X_{\Nis}$ and $X_{\Zar}$
denote the category of \'{e}tale schemes over a scheme $X$
equipped with
the \'{e}tale, Nisnevich and Zariski topology, respectively.
The category of sheaves on $X_{\et}$ is denoted by $\mathbb{S}_{X_{\et}}$.
Let $\epsilon: X_{\et}\to X_{\Zar}$ (resp. 
$\alpha: X_{\et}\to X_{\Nis}$)
be the canonical map of sites.
Then $\epsilon_{*}$ (resp. $\alpha_{*}$) denotes the forgetful functor.

$\mathbb{Z}(n)_{\et}$ (resp. $\mathbb{Z}(n)$, $\mathbb{Z}(n)_{\Nis}$)
denotes Bloch’s cycle complex for
\'{e}tale (resp. Zariski, Nisnevich) topology and 
$\mathbb{Z}/m(n)_{\et}=\mathbb{Z}(n)_{\et}\otimes\mathbb{Z}/m\mathbb{Z}$
(resp. $\mathbb{Z}/m(n)=\mathbb{Z}(n)\otimes\mathbb{Z}/m\mathbb{Z}$,
$\mathbb{Z}/m(n)_{\Nis}=\mathbb{Z}(n)_{\Nis}\otimes\mathbb{Z}/m\mathbb{Z}$)
for a positive integer $m$.
For an integer $m>0$, $\mu_{m}$ denotes the sheaf of $m$-th
roots of unity.
For a smooth scheme $Y$ over the spectrum of a field of positive characteristic
$p>0$, $\nu^{n}_{Y, r}=W_{r}\Omega_{Y, \log}^{n}$ denotes
the logarithmic de Rham-Witt sheaf
(cf. \cite{I}, \cite[p.575, Definition 2.6]{Sh}).
\section{Gersten-type conjecture}\label{TGer}
Let $Y$ be a normal crossing variety over the spectrum of a field $k$ of positive characteristic $p>0$,
that is,
$Y$ is everywhere \'{e}tale locally isomorphic to
\begin{align*}
\spec(k[T_{0}, T_{1}, \cdots, T_{M}]/(T_{0}T_{1}\cdots T_{a}))    
\end{align*}
for some integer $a$ with $0\leq a\leq M=\operatorname{dim}(Y)$
(cf. \cite[pp.180--181, Definition 2.1]{SaR}).
In this section, 
we prove the Gersten-type conjecture 
for a pair of a sheaf $\mathcal{F}$ on $Y_{\et}$ 
and a non-negative integer $N$
which satisfies
Condition \ref{cdpb}
over the henselization ring of the local ring $\mathcal{O}_{Y, y}$ of $Y$ at a point $y\in Y$
( Theorem \ref{Ex&Ger}, Theorem \ref{1gerl}).
Let
\begin{align*}
\lambda^{n}_{Y, r}
:=
\operatorname{Im}\left(
d\log:
(\mathbb{G}_{m, Y})^{\otimes n}
\to
\displaystyle 
\bigoplus_{y\in Y^{(0)}}
i_{x *}W_{r}\Omega^{n}_{y, \log}
\right)
\end{align*}
be the logarithmic Hodge-Witt sheaves (cf. \cite[p.726, Definition 3.1.1 (1)]{SaL}). 
For example, a pair $(\mathcal{F}, N)=(\lambda^{n}_{Y, r}, 1)$ satisfies Condition \ref{cdpb}
(see Remark \ref{ExaCond}.2. (ii)).
  
Let $\mathfrak{X}$ be a semistable family
over the spectrum of a discrete valuation ring $B$ of mixed characteristic $(0, p)$.
Suppose that $B$ contains $p$-th roots of unity.
As an application of Theorem \ref{Ex&Ger}, 
we prove the relative version of
Gersten-type conjecture for $\mathfrak{T}_{1}(n)$ over the henselization of
the local ring $\mathcal{O}_{\mathfrak{X}, x}$ 
at a point $x\in \mathfrak{X}$ ( Theorem \ref{Gtcl}).

\subsection{The outline of the proof of Theorem \ref{Ex&Ger}}\label{OutMai}

The outline of the proof of Theorem \ref{Ex&Ger}
is as follows:
By \cite[p.93, III, Corollary 1.28]{M},
it suffices to prove
$\operatorname{P}_{1}(a)$ and
$\operatorname{P}_{2}(a)$
in the case where
$Y$ is replaced by the spectrum of 
the henselization
$A=\mathcal{O}^{h}_{Y, y}$
of the local ring 
$\mathcal{O}_{Y, y}$
of $Y$ at a point $y\in Y$
and a point $y\in Y$ is replaced by the maximal ideal of $A$.
We put 
$\spec(C)=\spec(A)\times_{Y}Z$
and prove Theorem \ref{Ex&Ger} by induction on $a=\#(\spec(A)^{(0)})$.

In the case where $a=1$,
then $\spec(C)=\emptyset$ and so $\operatorname{P}_{1}(1)$
and $\operatorname{P}_{2}(1)$ hold. Since $Y$
is a smooth scheme over the spectrum of a field,
$\operatorname{P}_{3}(1)$ holds by Condition \ref{cdpb}.

Assume that $\operatorname{P}_{1}(u)$, $\operatorname{P}_{2}(u)$ and
$\operatorname{P}_{3}(u)$ hold for $u\leq a$, 
then we prove 
$\operatorname{P}_{1}(a+1)$, $\operatorname{P}_{2}(a+1)$
and $\operatorname{P}_{3}(a+1)$.

First, we prove $\operatorname{P}_{1}(a+1)$. Since $\#(\spec(C)^{(0)})=a$, $\operatorname{P}_{1}(a+1)$ follows from
$\operatorname{P}_{3}(a)$
and Lemma \ref{sredis}.

Next we prove $\operatorname{P}_{2}(a+1)$
and $\operatorname{P}_{3}(a+1)$ by downward induction on $N^{\prime}(\geq N_{1})$.
Let $\mathcal{G}$ be a sheaf of abelian groups on $\spec(A)_{\et}$. Then
$\HO^{N^{\prime}+s}_{x}(A_{\et}, \mathcal{G})_{T_{q}}$
vanishes for $x\in \spec(A)^{(s)}$ and sufficiently large $N^{\prime}$
by Proposition \ref{cdpt} and Remark \ref{cdsp}. 
Since $\operatorname{cd}_{q}(k)$ and $\operatorname{tr.deg}_{k}\kappa(y)$
are finite,
$\HO^{N^{\prime}}_{\et}(A, \mathcal{F}_{\spec(A)})_{T_{q}}$
vanishes for sufficiently large $N^{\prime}$ by the isomorphism (\ref{isoAI}). 
So $\operatorname{P}_{2}(a+1)$ and $\operatorname{P}_{3}(a+1)$ hold for
sufficiently large $N^{\prime}$.

Assume that $\operatorname{P}_{2}(a+1)$ and
$\operatorname{P}_{3}(a+1)$ hold for
$N^{\prime}>w (\geq N_{1})$. By the definition, 
$\spec(A)\times_{Y}Y_{a+1}$ is the spectrum of the henselization $A_{a+1}$
of the local ring $\mathcal{O}_{Y_{a+1}, y}$
of a smooth scheme $Y_{a+1}$ at a point $y\in Y_{a+1}$.
Moreover, 
$\spec(C)\times_{Y}Y_{a+1}$
is the spectrum of the henselization $C_{a+1}$
of the local ring
$\mathcal{O}_{Z\times_{Y}Y_{a+1}}, y$ of
a normal crossing variety $Z\times_{Y}Y_{a+1}$ at a point $y\in Z\times_{Y}Y_{a+1}$.
Since we have
\begin{equation*}
U^{\prime}
=
\spec(A)\setminus
\spec(C)
=
\spec(A_{a+1})\setminus
\spec(C_{a+1}),
\end{equation*}
there is an open immersion
$j^{\prime}: U^{\prime}\to \spec(A_{a+1})$.
Let $i_{a+1}: \spec(A_{a+1})\to \spec(A)$
be a closed immersion.
Then $j=i_{a+1}\circ j^{\prime}$ and so we have an
isomorphism
\begin{align*}
\HO^{u}_{x}(
A_{\et},
j_{!}\mathcal{F}_{U^{\prime}}
)
=
\left\{
\begin{array}{ll}
\HO^{u}_{x}(
(A_{a+1})_{\et},
(j^{\prime})_{!}\mathcal{F}_{U^{\prime}}
),  
&  x\in\spec(A)\cap\spec(A_{a+1})  
\\ [5pt]
0  &  x\in\spec(A)\setminus \spec(A_{a+1})
\end{array}
\right.
\end{align*}
for any $u\geq 0$ by Lemma \ref{ijdex} and \cite[Lemma 3.7]{Sak5}. 
Moreover, the sequence
\begin{align}\label{0AexC}
0
&\to
\HO^{N^{\prime}}_{\et}(A_{a+1}, (j^{\prime})_{!}\mathcal{F}_{U^{\prime}})_{T_{q}}
\to 
\displaystyle\bigoplus_{x\in\spec(A_{a+1})^{(0)}}
\HO^{N^{\prime}}_{x}(
(A_{a+1})_{\et},
(j^{\prime})_{!}\mathcal{F}_{U^{\prime}}
)_{T_{q}}  \nonumber
\\
&\to
\displaystyle\bigoplus_{x\in\spec(A_{a+1})^{(1)}}
\HO^{N^{\prime}+1}_{x}(
(A_{a+1})_{\et},
(j^{\prime})_{!}\mathcal{F}_{U^{\prime}}
)_{T_{q}}
\to
\cdots
\end{align}
is exact by Lemma \ref{CAL} and $\operatorname{P}_{3}(a)$.
Since we have the exact sequence (\ref{0AexC})
and $\operatorname{P}_{3}(a)$ holds
by the assumption,
we are able to prove that $\operatorname{P}_{2}(a+1)$
holds for $N^{\prime}=w$
by Proposition \ref{exind}, Lemma \ref{DGC}
and induction on $s=\operatorname{dim}(A)$.
Moreover, $\operatorname{P}_{3}(a+1)$
follows from the exact sequence (\ref{0AexC}),
$\operatorname{P}_{3}(a)$
and $\operatorname{P}_{2}(a+1)$.
This concludes the outline of the proof.
\subsection{Equi-characteristic cases}
\begin{lem}\upshape\label{Spbor}
Let $X$ be an equidimensional scheme and
noetherian of dimension $d\geq 1$.
Let $\mathcal{F}^{\bullet}$ be a complex of sheaves of abelian groups on $X_{\et}$.
Consider the coniveau spectral sequence
\begin{align}\label{cosp}
E^{u, v}_{1} 
=
\bigoplus_{x\in X^{(u)}}
\HO_{x}^{u+v}(
X_{\et},
\mathcal{F}^{\bullet}
)
\Rightarrow
E^{u+v}
=
\HO^{u+v}_{\et}(
X,
\mathcal{F}^{\bullet}
)
\end{align}
(cf. \cite[Part 1, \S 1]{C-H-K}, \cite[Theorem 2.1]{Sak5}). 
Let $n$ be an integer. Let  $e$ be a positive integer such that $d-1$ or $d$. 
Moreover, we assume the following two conditions:
\begin{itemize}
\item[(a)] The homomorphism
\begin{equation*}
E^{n+e}
\to
E_{2}^{0, n+e}
\end{equation*}
is injective.
\item[(b)] If $e\neq 1$, we have
\begin{align*}
E^{u, v}_{2}
=
\begin{cases}
E^{u+v}
& (\textrm{if}~~u=0) \\
\\
0
& (\textrm{if}~~e\geq 3~~\textrm{and}~~0<u\leq e-2)
\end{cases}
\end{align*}
for $u+v=n+e-1$.
\end{itemize}
Then we have
\begin{equation*}
E_{2}^{e, n}
=0.    
\end{equation*}
\end{lem}
\begin{proof}\upshape
We remark that there exists the spectral sequence (\ref{cosp})
because the proof of \cite[Theorem 2.1]{Sak5}
is valid without assuming that
$\mathcal{F}^{\bullet}$ is a bounded below complex.

By the definition of the spectral sequence, there is a filtration
\begin{align*}
E^{m}=E^{m}_{0}\supset 
E^{m}_{1}
\supset 
\cdots
\supset 
E^{m}_{m}
\supset 0
\end{align*}
such that 
\begin{align*}
E^{m}_{u}/E^{m}_{u+1}
=
E^{u, m-u}_{\infty}
\end{align*}
for $0\leq u\leq d$. By the assumption (a),
the homomorphism
\begin{equation*}
E^{n+e}
\to
E_{\infty}^{0, n+e}
\end{equation*}
is an isomorphism and so we have
\begin{equation*}
E_{\infty}^{e, n}
=0,
\end{equation*}
that is, we have 
\begin{equation*}
E_{r}^{e, n}
=0    
\end{equation*}
for sufficiently large $r$.
So it suffices to prove that
\begin{align}\label{spind}
E_{r+1}^{e, n}
=
\operatorname{Coker}\left(
E_{r}^{e-r, n+r-1}
\to
E_{r}^{e, n}
\right) 
=
E_{r}^{e, n}
\end{align}
for $r\geq 2$. If $r>e$,
then $e-r<0$ and so the isomorphism (\ref{spind})
holds for $r>e$. By the assumption (b),
we have isomorphisms
\begin{align*}
E^{n+e-1}
\xrightarrow{\sim}
E^{0, n+e-1}_{\infty}
\xrightarrow{\sim}
E^{0, n+e-1}_{r}
\xrightarrow{\sim}
E^{0, n+e+1}_{2}
\end{align*}
for $r\geq 2$. So we have
\begin{align*}
\operatorname{Ker}
\left(
E^{0, n+e-1}_{e}
\to 
E^{e, n}_{e}
\right)
=
E^{0, n+e-1}_{e+1}  
=
E^{0, n+e-1}_{e}
\end{align*}
and the isomorphism (\ref{spind}) holds for $r=e$.
If $2\leq r\leq e-1$, then
\begin{equation*}
E_{r}^{e-r, n+r-1}
=0    
\end{equation*}
by the assumption (b). 
So the isomorphism (\ref{spind}) holds for $2\leq r\leq e-1$.
Hence, the isomorphism (\ref{spind}) holds for $r\geq 2$. 
This completes the proof.
\end{proof}
\begin{lem}\upshape\label{borTq}
Let the notations be the same as in Lemma \ref{Spbor}. 
Let $q$ be a prime number.
For an abelian group $M$,
$M_{T_{q}}$
denotes the $q$-power torsion 
subgroup of $M$. 
We assume the following three conditions:
\begin{itemize}
\item[(a)] The homomorphism
\begin{equation*}
(E^{n+e})_{T_{q}}
\to
(E_{2}^{0, n+e})_{T_{q}}
\end{equation*}
is injective.
\item[(b)] If $u+v=n+e-1$ and $u\leq e-2$, then
$E_{2}^{u, v}$ is a torsion group. 
Moreover, $E^{n+e}$ is also a torsion group.
\item[(c)] If $e\neq 1$, we have
\begin{align*}
(E^{u, v}_{2})_{T_{q}}
=
\begin{cases}
(E^{u+v})_{T_{q}}
& (\textrm{if}~~u=0) \\
\\
0
& (\textrm{if}~~e\geq 3~~\textrm{and}~~0<u\leq e-2)
\end{cases}
\end{align*}
for $u+v=n+e-1$.
\end{itemize}
Then we have
\begin{equation*}
(E_{2}^{e, n})_{T_{q}}=0.    
\end{equation*}
\end{lem}
\begin{proof}\upshape
Let the notations be the same as in the proof of Lemma \ref{Spbor}. 
By the assumption (b), $E^{n+e}_{r}$ is a torsion group for any $r$ 
and
so $E^{n+e}_{r}/E^{n+e}_{r+1}=E_{\infty}^{r, n+e-r}$ is a torsion group for any $r$.
Especially,
$E_{\infty}^{e, n}$ is a torsion group.
Moreover, we have
\begin{equation*}
(E_{\infty}^{e, n})_{T_{q}}=0    
\end{equation*}
by the assumption (a).
By the assumption (b), 
$E_{r}^{e-r, n+r-1}$
is a torsion group for $r\geq 2$
and so we have
\begin{align*}
(E_{r+1}^{e, n})_{T_{q}}
=
\operatorname{Coker}\left(
(E_{r}^{e-r, n+r-1})_{T_{q}}
\to
(E_{r}^{e, n})_{T_{q}}
\right) 
\end{align*}
for $r\geq 2$ by (\ref{spind}).
In the same way as in the proof of Lemma \ref{Spbor},
we are able to prove that
\begin{equation*}
(E_{r+1}^{e, n})_{T_{q}}
=
(E_{r}^{e, n})_{T_{q}}
\end{equation*}
for $r\geq 2$ by the assumption (c). Hence the statement follows.
\end{proof}

In this section, we frequently assume that a ring $A$ is a formally
equidimensional noetherian local ring (cf. \cite[p.251, Definition]{Ma}).
For example, an equidimensional local ring which is essentially of finite
type over a regular local ring is formally equidimensional by \cite[p.251, Theorem 31.6 (iii)]{Ma}.

\begin{rem}\upshape\label{Rfe}
Let $A$ be a formally equidimensional noetherian local ring. 
By \cite[p.251, Theorem 31.6 (iv)]{Ma} and
\cite[p.250, Theorem 31.5]{Ma}, $A$ is equidimensional.
By \cite[p.251, Theorem 31.6 (i)]{Ma},
$A_{\mathfrak{p}}$
is also formally equidimensional for any $\mathfrak{p}\in \spec(A)$.
Moreover, the henselization
$A_{\mathfrak{p}}^{h}$
of $A_{\mathfrak{p}}$ is  formally equidimensional by the definition.
\end{rem}
\begin{prop}\upshape\label{cst}
Let $A$ be a formally equidimensional noetherian local ring,
$\mathcal{F}$ a sheaf of groups on $\spec(A)_{\et}$
and $x\in \spec(A)^{(s)}$. Put 
\begin{align*}
T:=
\left\{
\begin{array}{ll}
0  &  \textrm{if $\mathcal{F}$ is a sheaf of torsion groups}  \\ [5pt]
1   &  \textrm{otherwise}.
\end{array}
\right.
\end{align*}
Then
$\HO^{t}_{x}(A_{\et}, \mathcal{F})$
is a torsion group
for $t\geq s+T$.
\end{prop}
\begin{proof}\upshape
By Remark \ref{Rfe} and \cite[p.93. III, Corollary 1.28]{M},  
it suffices to prove the statement in the case where $A$
is a formally equidimensional henselian local ring and
$x$ is the maximal ideal of $\spec(A)$. 
We prove the statement
by induction on $\operatorname{dim}(A)$.

In the case where $\operatorname{dim}(A)=0$,
the statement follows from \cite[p.180, Theorem 6.5.8]{We}.

Suppose that the statement holds in the case where
$\operatorname{dim}(A)\leq s$.
Suppose that $\operatorname{dim}(A)=s+1$.
Consider the spectral sequence
\begin{equation}\label{spt}
E^{u, v}_{1}
=
\bigoplus_{x\in (\spec(A))^{(u)}}
\HO_{x}^{u+v}
(A_{\et},
\mathcal{F})
\Rightarrow
E^{u+v}=
\HO_{\et}^{u+v}(A,
\mathcal{F}
).
\end{equation}
By \cite[p.180, Theorem 6.5.8]{We},
$E^{u+v}$
is a torsion group for $u+v\geq T$ and so 
$E^{s+1, v}_{\infty}$ is also a torsion group for $v\geq 1$.
By the assumption,
$E^{u, v}_{r}$ is a torsion group for $u\leq s$,
$v\geq T$
and $r\geq 1$. By the definition of the spectral sequence, the sequence
\begin{align}\label{sqrt}
E_{r}^{s-r+1, v+r-1}
\to
E^{s+1, v}_{r}
\to
E^{s+1, v}_{r+1}
\to 0
\end{align}
is exact for $v\geq T$ and $r\geq 1$.
Hence $E^{s+1, v}_{1}$ is a torsion group for $v\geq T$.
This completes the proof.
\end{proof}
\begin{cor}\upshape\label{cdt}
Let the notations $A$, $\mathcal{F}$ and $T$  be the same as in
Proposition \ref{cst}. Suppose that
$\operatorname{dim}(A)=s$. Then
$\HO^{t}_{\et}(A, \mathcal{F})$
is a torsion group for $t\geq s+T$.
\end{cor}
\begin{proof}\upshape
By the spectral sequence (\ref{spt}),
the statement follows from Proposition \ref{cst}
and \cite[p.93, III, Corollary 1.28]{M}.
\end{proof}
\begin{rem}\upshape\label{HenTor}
Let $A$ be a henselian local ring,
$I$ an ideal of $A$
and $\mathcal{F}$ a sheaf of groups on
$\spec(A)_{\et}$. Then we have an isomorphism
\begin{equation}\label{isoAI}
\HO^{t}_{\et}(A, \mathcal{F})
\xrightarrow{\simeq}
\HO^{t}_{\et}(A/I, i^{*}\mathcal{F})
\end{equation}
for $t\geq 0$ where
$i: \spec(A/I)\to \spec(A)$
is the closed immersion
by 
\cite[p.777, The proof of Proposition 2.2.b)]{Ge}.
So 
$\HO^{t}_{\et}(A, \mathcal{F})$
is a torsion group for $t>0$.
Moreover, we have an isomorphism
\begin{equation*}
\HO^{t}_{\et}(
A,
j_{!}\mathcal{F}_{U}
)    
\simeq 
0
\end{equation*}
for any $t\geq 0$ by (\ref{isoAI})
where
$j: U\to\spec(A)$ is the immersion of the open
complement of $\spec(A/I)$ in
$\spec(A)$.
\end{rem}
\begin{prop}\upshape\label{cdpt}
Let $A$ be a formally  equidimentional noetherian local ring over a field $k$,
$l$ a prime number which is prime to $\operatorname{char}(k)$ and 
$\mathcal{F}$ a sheaf of groups on $\spec(A)_{\et}$.
Let the notation $T$ be the same as in Proposition \ref{cst}.
Suppose that $\operatorname{cd}_{l}(k)<\infty$
and
$\operatorname{tr.deg}_{k}A_{x}<\infty$ for any $x\in \spec(A)^{(0)}$.
Put 
\begin{equation*}
N:=\operatorname{cd}_{l}(k)
+
\max\{\operatorname{tr.deg}_{k}A_{x} | x\in\spec(A)^{(0)}\}.
\end{equation*}
For a group $M$, 
$M_{T_{l}}$ denotes the $l$-power torsion subgroup of $M$.
Then we have
\begin{equation*}
\HO^{m}_{x}(A,
\mathcal{F}
)_{T_{l}}
=0
\end{equation*}
for $x\in \spec(A)^{(s)}$ and $m\geq s+N+T+1$.
\end{prop}
\begin{proof}\upshape
It suffices to prove the statement in the case where $A$ is a formally
equidimensional henselian local ring and 
$x$ is the maximal ideal of $\spec(A)$. We prove the statement by induction on
$\operatorname{dim}(A)$.

In the case where $\operatorname{dim}(A)=0$,
the statement follows from
\cite[(3.3.3) Proposition]{N} and
\cite[(6.5.14) Theorem]{N}.

Suppose that the statement holds in the case where
$\operatorname{dim}(A)\leq s$. 
Suppose that 
$\operatorname{dim}(A)=s+1$. Consider the spectral sequence (\ref{spt}).
By the assumption, it suffices to prove that the homomorphism
\begin{equation*}
(E_{1}^{s, v})_{T_{l}}
\to
(E_{1}^{s+1, v})_{T_{l}}
\end{equation*}
is surjective for $v\geq N+T+1$, that is,
\begin{equation}\label{cdspl}
(E_{2}^{s+1, v})_{T_{l}}
=0
\end{equation}
for $v\geq N+T+1$.
In order to prove the statement, we use Lemma \ref{borTq}.
Since $A$ is henselian,
we have an isomorphism
\begin{equation}\label{locprp2}
\HO^{m}_{\et}(A, \mathcal{F})
\simeq 
\HO^{m}_{\et}
(\kappa(x),
i^{*}\mathcal{F})
\end{equation}
for $m\geq 0$
where 
$i: \spec(\kappa(x))\to\spec(A)$
is the closed immersion.
Since $A$ is universally catenary by \cite[p.251, Theorem 31.6 (iv)]{Ma},
we have
\begin{equation*}
\operatorname{tr.deg}_{k}\kappa(x)
=(N-\operatorname{cd}_{l}(k))-(s+1)
\end{equation*}
by \cite[p.119, Theorem 15.6]{Ma}. So we have
\begin{align}\label{VaniEl}
(E^{m})_{T_{l}}=
\HO^{m}_{\et}(A, \mathcal{F})_{T_{l}}
=0    
\end{align}
for $m\geq N+T-s$ by
\cite[(6.5.14) Theorem]{N} and \cite[(3.3.3) Proposition]{N}.
Especially, we have
\begin{equation*}
(E^{m})_{T_{l}}=0    
\end{equation*}
for $m\geq s+N+T+2$
and so the assumption (a) of Lemma \ref{borTq} is satisfied.
If $u+v=s+N+T+1$ and $u\leq s-1$,
then $v\geq N+T+2$ and so $E^{u, v}_{2}$
is a torsion group by Proposition \ref{cst}.
Moreover, $E^{m}$ is a torsion group for $m\geq s+N+T+2$
by (\ref{locprp2}) and \cite[p.180, Theorem 6.5.8]{We}. 
So the assumption (b) of Lemma \ref{borTq} is satisfied.
By the assumption, we have
\begin{align*}
(E^{u+v})_{T_{q}}
=
(E^{u, v}_{2})_{T_{q}}
=0
\end{align*}
for $u+v\geq s+N+T+1$ and
$u\leq s-1$. So the assumption (c) of Lemma \ref{borTq}
is satisfied. Hence, the isomorphism (\ref{cdspl}) holds.
This completes the proof.

\end{proof}
\begin{rem}\upshape\label{cdsp}
Let $A$ be a formally equidimensional local ring of positive characteristic $p>0$
and $\mathcal{F}$ a sheaf of groups on $\spec(A)_{\et}$.
Since $\operatorname{cd}_{p}(A)=1$,
we have
\begin{equation*}
\HO^{t}_{\et}(A,
\mathcal{F})_{T_{p}}
=0
\end{equation*}
for $t\geq T+2$ by a similar argument
as in the proof of \cite[Chapter III, \S 3, (3.3.3) Proposition]{N}.
Moreover, we have
\begin{align*}
\HO^{t}_{x}(
A_{\et},
\mathcal{F}
)_{T_{p}}
=0
\end{align*}
for $x\in \spec(A)^{(s)}$
and
$t\geq T+s+2$ by using a similar argument as in the proof 
of Proposition \ref{cdpt} (or \cite[Chapter III, \S 3, (3.3.3) Proposition]{N}).
\end{rem}
\begin{prop}\upshape\label{exind}
Let $A$ be a formally equidimensional local ring and
$\mathcal{F}$ a sheaf on the \'{e}tale site
$\spec(A)_{\et}$.
Let $N$ be a non-negative integer and $l$ a prime number.
Suppose that 
$\operatorname{dim}(A)=s>1$ and $N\geq T$ where
the notation $T$ is the same as 
in Proposition \ref{cst}.

If 
$\HO^{N+s-1}_{\et}(A, \mathcal{F})$
is a torsion group and 
the sequence
\begin{align*}
0\to& 
\HO^{N^{\prime}}_{\et}(
A, \mathcal{F}
)_{T_{l}}
\to
\bigoplus_{x\in \spec(A)^{(0)}}
\HO^{N^{\prime}}_{x}
(A_{\et}, \mathcal{F})_{T_{l}}
\to 
\cdots 
\to
\bigoplus_{x\in \spec(A)^{(s-2)}}
\HO^{N^{\prime}+s-2}_{x}
(A_{\et}, \mathcal{F})_{T_{l}}
\\
\to&
\bigoplus_{x\in \spec(A)^{(s-1)}}
\HO^{N^{\prime}+s-1}_{x}
(A_{\et}, \mathcal{F})_{T_{l}}
\end{align*}
is exact for 
$N^{\prime}\geq N+1$,
then the sequence
\begin{align*}
&\bigoplus_{x\in \spec(A)^{(s-2)}} 
\HO^{N+s-2}_{x}(A_{\et}, \mathcal{F})_{T_{l}}
\to
\bigoplus_{x\in \spec(A)^{(s-1)}} 
\HO^{N+s-1}_{x}(A_{\et}, \mathcal{F})_{T_{l}} 
\\
\to&
\bigoplus_{x\in \spec(A)^{(s)}} 
\HO^{N+s}_{x}(A_{\et}, \mathcal{F})_{T_{l}}
\to
0
\end{align*}
is exact.
\end{prop}
\begin{proof}\upshape
Consider the spectral sequence
\begin{align*}
E^{u, v}_{1}
=
\displaystyle
\bigoplus_{x\in \spec(A)^{(u)}}
\HO^{u+v}_{x}(A_{\et}, \mathcal{F})
\Rightarrow
E^{u+v}=
\HO^{u+v}_{\et}(A, \mathcal{F}).
\end{align*}
Then it suffices to prove that
\begin{equation}\label{indsptoex}
(E^{e, N}_{2})_{T_{l}}=0    
\end{equation}
for $s-1\leq e\leq s$.
We prove the isomorphism (\ref{indsptoex}) by using Lemma \ref{borTq}. 

By the assumption, 
the homomorphism
\begin{equation*}
(E^{e+N})_{T_{l}}
\to
(E^{0, e+N}_{2})_{T_{l}}
\end{equation*}
is injective.
If 
$u+v=N+e-1$ and $u\leq e-2$,
then $v\geq N+1$. If $e\neq 1$,
then $N+e-1\geq N+1$. So we have
\begin{align*}
(E^{u, v}_{2})_{T_{q}}
=
\begin{cases}
(E^{u+v})_{T_{q}}
& (\textrm{if}~~u=0) \\
\\
0
& (\textrm{if}~~e\geq 3~~\textrm{and}~~0<u\leq e-2)
\end{cases}
\end{align*}
for $e\neq 1$ and $u+v=N+e-1$ by the assumption.
By Proposition \ref{cst}, $E^{u, v}_{2}$ is a torsion group in the case where
$u+v=N+e-1$ and $u\geq e-2$. Moreover,
$E^{N+e}$ is a torsion group by
Corollary \ref{cdt} and the assumption.
Hence, the assumptions of Lemma \ref{borTq} are satisfied 
and the isomorphism (\ref{indsptoex}) holds.
This completes the proof.
\end{proof}
\begin{cor}\upshape
Let the notations $A$ and $\mathcal{F}$ be the same as 
in Proposition \ref{exind}. 
Let the notation $T$ be the same as in Proposition \ref{cst}.
Suppose that $\operatorname{char}(A)=p>0$. Then the sequence 
\begin{align*}
&\bigoplus_{x\in \spec(A)^{(s-2)}} 
\HO^{T+s-1}_{x}(A_{\et}, \mathcal{F})_{T_{p}}
\to
\bigoplus_{x\in \spec(A)^{(s-1)}} 
\HO^{T+s}_{x}(A_{\et}, \mathcal{F})_{T_{p}} 
\\
\to&
\bigoplus_{x\in \spec(A)^{(s)}} 
\HO^{T+s+1}_{x}(A_{\et}, \mathcal{F})_{T_{p}}
\to
0
\end{align*}
is exact.
\end{cor}
\begin{proof}\upshape
By Corollary \ref{cdt}, 
$\HO_{\et}^{m}(A, \mathcal{F})$ is a torsion group for
$m\geq 1+T+(s-1)=s+T$.
So the statement follows from Proposition \ref{exind}
and Remark \ref{cdsp}.
\end{proof}

\begin{cond}\upshape\label{cdpb}
Let $\mathcal{V}$ be the category whose
objects are normal crossing varieties over the spectrum of
a field $k$
and
whose morphisms are morphisms of schemes.
Let $X$ be an object in $\mathcal{V}$, 
$\mathcal{F}_{X}$ an \'{e}tale sheaf 
of abelian groups
on $X$
and
$f: X\to Y$
a morphism in $\mathcal{V}$. Let $N$ be a non-negative integer.
Then $(\mathcal{F}_{X}, N)\in \mathbb{S}_{X_{\et}}\times \mathbb{N}_{\geq 0}$
satisfies the followings:
\begin{enumerate}
\item We have a pull-back map
\begin{equation*}
f^{*}: \mathcal{F}_{Y}
\to
f_{*}\mathcal{F}_{X}
\end{equation*}
which satisfies the following properties:
\begin{enumerate}
\item[(i)] For two morphisms $f: X\to Y$
and $g: Y\to Z$ in $\mathcal{V}$,
the composite map
\begin{align*}
\mathcal{F}_{Z}
\xrightarrow[g^{*}]{}
g_{*}\mathcal{F}_{Y}
\xrightarrow[g_{*}(f^{*})]{}
(g\circ f)_{*}\mathcal{F}_{X}
\end{align*}
agrees with the pull-back map
$(g\circ f)^{*}$.
\item[(ii)] Let $X$ be an object in $\mathcal{V}$
and $j: U\to X$ an open immersion. Then the homomorphism
\begin{equation*}
j^{*}\mathcal{F}_{X}
\to 
\mathcal{F}_{U}
\end{equation*}
is an isomorphism where the homomorphism is induced by the pull-back map $j^{*}$.
\item[(iii)] Let 
$R=\mathcal{O}_{X, x}^{h}$
be the henselization of the local ring of a normal crossing variety $X$
at a point $x\in X$ and a closed immersion
$i: Z\to\spec(R)$ be a morphism in $\mathcal{V}$.
Then the homomorphism
\begin{equation*}
\HO^{N^{\prime}}_{\et}(Z,
i^{*}\mathcal{F}_{\spec(R)})
\to
\HO^{N^{\prime}}_{\et}(Z,
\mathcal{F}_{Z})
\end{equation*}
is an isomorphism for $N^{\prime}\geq N$ where
the homomorphism is induced by the pull-back map
$i^{*}$.
\end{enumerate}
\item
Let $X$ be a smooth scheme over $\spec(k)$ and 
$R$ the henselization of 
the local ring $\mathcal{O}_{X, x}$ at a point $x\in X$.
Then the sequence
\begin{align*}
0\to   
\HO^{N^{\prime}}_{\et}(R,
\mathcal{F}_{\spec(R)})
\to 
\bigoplus_{x\in\spec(R)^{(0)}}
\HO^{N^{\prime}}_{x}(
R_{\et},
\mathcal{F}_{\spec(R)}
)
\\
\to
\bigoplus_{x\in\spec(R)^{(1)}}
\HO^{N^{\prime}+1}_{x}(
R_{\et},
\mathcal{F}_{\spec(R)}
)
\to\cdots
\end{align*}
is exact for $N^{\prime}\geq N$.
\end{enumerate}
\end{cond}
\begin{rem}\upshape\label{ExaCond}
Let $\mathcal{V}$ be the same as in the above
and $X$ an object in $\mathcal{V}$.
\begin{enumerate}
\item 
Let $R$ be the henselization of the local ring of
$X$
at a point $x\in X$ and
$i: Z\to\spec(R)$
a closed immersion which is 
a morphism in $\mathcal{V}$.
If $(\mathcal{F}_{X}, N)\in\mathbb{S}_{X_{\et}}\times \mathbb{N}_{\geq 0}$
satisfies Condition \ref{cdpb}.1, 
then
we have an isomorphism
\begin{equation*}
\HO^{N^{\prime}}_{\et}(
R, \mathcal{F}_{\spec(R)}
)
\xrightarrow{\sim}
\HO^{N^{\prime}}_{\et}(
Z,
\mathcal{F}_{Z}
)
\end{equation*}
for $N^{\prime}\geq N$ by the isomorphism (\ref{isoAI}).
\item
In the following cases, Condition \ref{cdpb} is satisfied:
\begin{enumerate}
\item[(i)] 
Let $l$ be an integer which is prime to $\operatorname{char}(k)$
and $n$ a non-negative integer. 
If
$\mathcal{F}_{X}=\mu_{l}^{\otimes n}$ and  $N=0$, then $(\mathcal{F}_{X}, N)\in \mathbb{S}_{X_{\et}}\times \mathbb{N}_{\geq 0}$ 
satisfies Condition \ref{cdpb}
by \cite[Corollary 2.2.2]{C-H-K}.
\item[(ii)]
Suppose that $\operatorname{char}(k)=p>0$.
If $A^{\prime}$ is a commutative ring over $k$,
then $\operatorname{cd}_{p}(A^{\prime})=1$.

Let $r$, $n$ be non-negative integers. 
If $\mathcal{F}_{X}=\lambda_{X, r}^{n}$ and $N=1$, then
$(\mathcal{F}_{X}, N)\in \mathbb{S}_{X_{\et}}\times \mathbb{N}_{\geq 0}$ 
satisfies Condition \ref{cdpb}.1 
by \cite[p.732, Theorem 3.5.1]{SaL},
\cite[p.734, Corollary 3.5.3]{SaL} and \cite[p.30, Theorem 2.8]{Sak4}.
Moreover,
$(\mathcal{F}_{X}, N)\in \mathbb{S}_{X_{\et}}\times \mathbb{N}_{\geq 0}$ 
satisfies Condition \ref{cdpb}.2 
by \cite[p.600, Theorem 4.1]{Sh}.
\end{enumerate}
\end{enumerate}
\end{rem}

In order to give another example which satisfies
Condition \ref{cdpb}.2, 
we recall the followings:

\begin{lem}\upshape\label{Gyscom}
Let $X$ be essentially of finite type over the spectrum of a regular ring of
dimension at most $1$,
$n$ an integer
and
$\mathbb{Z}(n)$ Bloch's cycle complex for the Zariski topology.   
Let $i: Z\to X$ be a closed immersion of codimension $c$
with open complement $j: U\to X$,
$i_{1}: W\to X$
a closed immersion of codimension $c^{\prime}$
and
\begin{equation*}
\xymatrix{
Z_{1}\ar[r]^{i^{\prime}}\ar[d]_{i_{2}}
&
W\ar[d]_{i_{1}}
&
U_{1}\ar[l]_{j^{\prime}}\ar[d]^{i_{3}}
\\
Z\ar[r]_{i}
&
X
&
U \ar[l]^{j}
}    
\end{equation*}
Cartesian products.

Then the natural inclusion map
\begin{equation*}
i_{*}: i_{*}\mathbb{Z}(n-c)^{Z}[-2c]
\to 
\mathbb{Z}(n)^{X}
\end{equation*}
induces a quasi-isomorphism
\begin{equation*}
\operatorname{Gys}_{i}: 
\mathbb{Z}(n-c)^{Z}[-2c]
\to
Ri^{!}\mathbb{Z}(n)^{X}.
\end{equation*}
Moreover,
\begin{equation*}
\left(
(i^{\prime})_{*}(\operatorname{Gys}_{i\circ i_{2}}),
\operatorname{Gys}_{i_{1}},
R(j^{\prime})_{*}(\operatorname{Gys}_{i_{3}})
\right)
\end{equation*}
is a morphism of triangles,
that is,
we have a commutative diagram
\begin{equation*}
\xymatrix@C=6pt{
\cdots\ar[r]
&
(i^{\prime})_{*}\mathbb{Z}(n-c^{\prime\prime})^{Z_{1}}[-2c^{\prime\prime}]
\ar[r]^-{(i^{\prime})_{*}}
\ar[d]_{(i^{\prime})_{*}(\operatorname{Gys}_{i\circ i_{2}})}
&
\mathbb{Z}(n-c^{\prime})^{W}[-2c^{\prime}]
\ar[r]
\ar[d]_{\operatorname{Gys}_{i_{1}}}
&
R(j^{\prime})_{*}\mathbb{Z}(n-c^{\prime})^{U_{1}}[-2c^{\prime}]\ar[r]
\ar[d]^{R(j^{\prime})_{*}(\operatorname{Gys}_{i_{3}})}
& \cdots
\\
\cdots\ar[r]
&
(i^{\prime})_{*}R(i\circ i_{2})^{!}\mathbb{Z}(n)^{X}
\ar[r]
&
R(i_{1})^{!}\mathbb{Z}(n)^{X}
\ar[r]
&
R(j^{\prime})_{*}R(i_{3})^{!}\mathbb{Z}(n)^{U}
\ar[r]
&
\cdots
}    
\end{equation*}
where $c^{\prime\prime}=\operatorname{codim}(Z_{1}, X)$.
\end{lem}
\begin{proof}\upshape
By \cite[p.780, Corollary 3.3.a)]{Ge} and 
$i^{!}j_{*}=0$ (cf.\cite[p.76, II, Proposition 3.14 (c)]{M}),
$\operatorname{Gys}_{i}$ is an isomorphism.
Since we have a commutative diagram
\begin{equation*}
\xymatrix@C=40pt{
(i\circ i_{2})_{*}\mathbb{Z}(n-c^{\prime\prime})^{Z_{1}}[-2c^{\prime\prime}]
\ar[r]^-{(i_{1})_{*}((i^{\prime})_{*})}
\ar[d]_{i_{*}((i_{2})_{*})}
&
(i_{1})_{*}\mathbb{Z}(n-c^{\prime})^{W}[-2c^{\prime}]
\ar[d]^{(i_{1})_{*}}
\\
i_{*}\mathbb{Z}(n-c)^{Z}[-2c]
\ar[r]_{i_{*}}
&
\mathbb{Z}(n)^{X}
}    
\end{equation*}
by the functoriality of the natural inclusion map,
there exists a morphism
\begin{equation*}
h: Rj_{*}(i_{3})_{*}\mathbb{Z}(n-c^{\prime})^{U_{1}}
\to
Rj_{*}\mathbb{Z}(n)^{U}
\end{equation*}
so that
\begin{math}
\left(
i_{*}((i_{2})_{*}),
(i_{1})_{*},
h
\right)
\end{math}
is a morphism of triangles
by \cite[p.780, Corollary 3.3.a)]{Ge}. 
Moreover, we have
\begin{align*}
j^{*}(h)=j^{*}((i_{1})_{*})
&&
\textrm{and}
&&
h=Rj_{*}(i_{3})_{*}
\end{align*}
by $j^{*}i_{*}=0$ (cf.\cite[p.76, II, Proposition 3.14 (c)]{M}). 
So
\begin{math}
\left( 
i_{*}((i_{2})_{*}),
(i_{1})_{*},
Rj_{*}(i_{3})_{*}
\right)
\end{math}
is a morphism of triangles. Moreover, we have
\begin{align*}
R(i_{1})^{!}i_{*} 
=(i^{\prime})_{*}R(i_{2})^{!}
&&
\textrm{and}
&&
R(i_{1})^{!}Rj_{*}
=R(j^{\prime})_{*}R(i_{3})^{!}.
\end{align*}
So 
\begin{math}
\left( 
(i^{\prime})_{*}(\operatorname{Gys}_{i_{2}}),
\operatorname{Gys}_{i_{1}},
R(j^{\prime})_{*}(\operatorname{Gys}_{i_{3}})
\right)
\end{math}
is a morphism of triangles.
Since $i_{*}$ is a left adjoint functor
to $Ri^{!}$, we have a commutative diagram
\begin{equation*}
\xymatrix{
i_{*}\mathbb{Z}(n-c)^{Z}[-2c]
\ar[r]^-{i_{*}}
\ar[d]_{i_{*}(\operatorname{Gys}_{i})}
&
\mathbb{Z}(n)^{X} 
\ar@{=}[d]
\\
i_{*}Ri^{!}\mathbb{Z}(n)^{X}
\ar[r]
&
\mathbb{Z}(n)^{X}
}    
\end{equation*}
and so we are able to prove that
\begin{math}
\left( 
i_{*}(\operatorname{Gys}_{i}),
\operatorname{id}_{\mathbb{Z}(n)},
\operatorname{id}_{Rj_{*}\mathbb{Z}(n)}
\right)
\end{math}
is a morphism of triangles as above.
Moreover, we have
\begin{align*}
\operatorname{Gys}_{i\circ i_{2}} 
=
R(i_{2})^{!}(\operatorname{Gys}_{i})
\circ 
\operatorname{Gys}_{i_{2}}
\end{align*}
by the functoriality of the natural 
inclusion map.
Hence 
\begin{equation*}
\left( 
(i^{\prime})_{*}(\operatorname{Gys}_{i\circ i_{2}}),
\operatorname{Gys}_{i_{1}},
R(j^{\prime})_{*}(\operatorname{Gys}_{i_{3}})
\right)
\end{equation*}
is a morphism of triangles. This completes the proof.
\end{proof}
\begin{prop}\upshape\label{EtmotGer}
Let $X$ be a smooth scheme over the spectrum of a field $k$,
$A$ the local ring 
$\mathcal{O}_{X, y}$ of $X$ at a point $y\in X$
and $x\in \spec(A)^{(s)}$. 
Let $n$ be an integer.
Let $\mathbb{Z}(n)$
(resp. $\mathbb{Z}(n)_{\et}$)
be
Bloch's cycle complex for
the Zariski (resp. \'{e}tale) topology
and 
$\mathbb{Q}/\mathbb{Z}(n)_{\et}=\mathbb{Z}(n)_{\et}\otimes_{\mathbb{Z}}\mathbb{Q}/\mathbb{Z}$.
Then we have an isomorphism
\begin{align}\label{SupZ}
\HO^{i}_{x}(
A_{\et},
\mathbb{Z}(n)_{\et}
)
=
\left\{
\begin{array}{ll}
\HO^{i}_{x}(A_{\Zar},
\mathbb{Z}(n))
&  \textrm{if $i\leq s+n+1$}  \\ [5pt]
\HO^{i-1}_{x}(
A_{\et},
\mathbb{Q}/\mathbb{Z}(n)_{\et}
)  
&  \textrm{if $i> s+n+1$}.
\end{array}
\right.
\end{align}
Moreover, the sequence
\begin{align*}
0
\to 
\HO^{i}_{\et}(
A, \mathbb{Z}(n)_{\et}
)
\to 
\displaystyle\bigoplus_{x\in \spec(A)^{(0)}}
\HO^{i}_{x}(
A_{\et},
\mathbb{Z}(n)_{\et}
)
\\
\to
\displaystyle\bigoplus_{x\in \spec(A)^{(1)}}
\HO^{i+1}_{x}(
A_{\et},
\mathbb{Z}(n)_{\et}
)
\to
\cdots
\end{align*}
is exact for any integer $i$.
\end{prop}
\begin{proof}\upshape
By \cite[p.299, Theorem (10.1)]{B}, 
\cite[p.779, Theorem 3.2.b)]{Ge} and Lemma \ref{Gyscom},
the sequence
\begin{align}\label{GZarZ}
0
\to
\HO^{i}_{\Zar}(A, \mathbb{Z}(n))
\to
\displaystyle\bigoplus_{x\in \spec(A)^{(0)}}
\HO^{i}_{x}(A_{\Zar},
\mathbb{Z}(n)) \nonumber
\\
\to
\displaystyle\bigoplus_{x\in \spec(A)^{(1)}}
\HO^{i+1}_{x}(
A_{\Zar},
\mathbb{Z}(n)
) 
\to \cdots
\end{align}
is exact for any integer $i$.
If $\operatorname{char}(k)=p\geq 0$,
then we have a quasi-isomorphism
\begin{align*}
\mathbb{Z}/m\mathbb{Z}(n)_{\et}
\simeq
\left\{
\begin{array}{ll}
\mu_{m}^{\otimes n}
&  \textrm{if $(m, p)=1$}  \\ [5pt]
\nu_{r}^{n}
&  \textrm{if $m=p^{r}$, $r>0$}
\end{array}
\right.
\end{align*}
by \cite[p.774, Theorem 1.2.4]{Ge} and
\cite[p.491, Theorem 8.3]{Ge-L}. 
So the sequence
\begin{align}\label{GetQZ}
0
\to
\HO^{i}_{\et}(A, \mathbb{Q}/\mathbb{Z}(n)_{\et}
)
\to
\displaystyle\bigoplus_{x\in \spec(A)^{(0)}}
\HO^{i}_{x}(A_{\et},
\mathbb{Q}/\mathbb{Z}(n)_{\et}
)  \nonumber
\\
\to 
\displaystyle\bigoplus_{x\in \spec(A)^{(1)}}
\HO^{i+1}_{x}(
A_{\et},
\mathbb{Q}/\mathbb{Z}(n)_{\et}
)  
\to \cdots
\end{align}
is exact for any integer $i$
by
\cite[Corollary 2.2.2]{C-H-K} and
\cite[p.600, Theorem 4.1]{Sh}.
Since we have
a distinguished triangle
\begin{equation*}
\cdots 
\to
\mathbb{Z}(n)
\xrightarrow{\times m}
\mathbb{Z}(n) 
\to
\mathbb{Z}/m(n) 
\to \cdots
\end{equation*}
and an isomorphism
\begin{equation*}
\mathbb{Z}/m(n)
\simeq 
\frac{1}{m}\mathbb{Z}/\mathbb{Z}(n)
\end{equation*}
for any positive integer $m$, 
we have
\begin{align*}
\HO^{i}_{\et}(A, \mathbb{Q}(n)_{\et})
=
\HO^{i}_{\Zar}(A, \mathbb{Q}(n))
=
\HO^{i}_{\Zar}(A, \mathbb{Q}/\mathbb{Z}(n))
=0
\end{align*}
for $i\geq n+1$
by \cite[p.781, Proposition 3.6]{Ge} 
and \cite[p.786, Corollary 4.4]{Ge}. Hence 
we have an isomorphism
\begin{align}\label{emcvs}
\HO^{i}_{\et}(
A,
\mathbb{Z}(n)_{\et}
)
=
\left\{
\begin{array}{ll}
\HO^{i}_{\Zar}(A,
\mathbb{Z}(n))
&  \textrm{if $i\leq n+1$}  \\ [5pt]
\HO^{i-1}_{\et}(
A,
\mathbb{Q}/\mathbb{Z}(n)_{\et}
)  
&  \textrm{if $i> n+1$}
\end{array}
\right.
\end{align}
by \cite[p.774, Theorem 1.2.2]{Ge} and \cite{V2}. In order to prove the statement, it suffices to prove the isomorphism (\ref{SupZ})
in the case where $\operatorname{dim}(A)=s$.
We prove the isomorphism (\ref{SupZ}) by induction on $\operatorname{dim}(A)=s$.

In the case where 
$\operatorname{dim}(A)=s=0$,
the statement follows from (\ref{emcvs}).

Assume that $\operatorname{dim}(A)=s=1$.
Since the statement holds in the case where
$\operatorname{dim}(A)=s=0$, 
the homomorphism
\begin{equation*}
\HO^{i}_{\et}(
A, \mathbb{Z}(n)_{\et}
)
\to
\displaystyle\bigoplus_{x\in \spec(A)^{(0)}}
\HO^{i}_{x}(
A_{\et},
\mathbb{Z}(n)_{\et}
)   
\end{equation*}
is injective by (\ref{GZarZ}), 
(\ref{GetQZ}) and (\ref{emcvs}).
So the homomorphism
\begin{align*}
\displaystyle\bigoplus_{x\in \spec(A)^{(0)}}
\HO^{i}_{x}(
A_{\et},
\mathbb{Z}(n)_{\et}
)
\to
\displaystyle\bigoplus_{x\in \spec(A)^{(1)}}
\HO^{i+1}_{x}(
A_{\et},
\mathbb{Z}(n)_{\et}
)
\end{align*}
is surjective and the sequence
\begin{align*}
\HO^{i}_{\et}(
A, \mathbb{Z}(n)_{\et}
)
\to
\displaystyle\bigoplus_{x\in \spec(A)^{(0)}}
\HO^{i}_{x}(
A_{\et},
\mathbb{Z}(n)_{\et}
)
\to
\displaystyle\bigoplus_{x\in \spec(A)^{(1)}}
\HO^{i+1}_{x}(
A_{\et},
\mathbb{Z}(n)_{\et}
)
\to 0
\end{align*}
is exact. 
Hence the statement holds in the case where 
$\operatorname{dim}(A)=s=1$
by (\ref{GZarZ}), 
(\ref{GetQZ}) and (\ref{emcvs}).

Assume that the statement holds in the case where 
$\operatorname{dim}(A)=s\leq t$ and $t\geq 1$. Suppose that 
$\operatorname{dim}(A)=s=t+1$. Consider the spectral sequence
\begin{equation*}
E_{1}^{u, v}
=\displaystyle\bigoplus_{x\in \spec(A)^{(u)}}
\HO_{x}^{u+v}(
A_{\et},
\mathbb{Z}(n)_{\et}
)
\Rightarrow
E^{u+v}
=\HO^{u+v}_{\et}(A, \mathbb{Z}(n)_{\et}).
\end{equation*}
By (\ref{GZarZ}), (\ref{GetQZ})
and the assumption,
it suffices to prove that the sequence
\begin{equation*}
E_{1}^{t-1, i}
\to 
E_{1}^{t, i}
\to 
E_{1}^{t+1, i}
\to 0
\end{equation*}
is exact for any integer $i$,
that is,
\begin{equation}\label{E2eivan}
E^{e, i}_{2}=0    
\end{equation}
for any integer $i$ and $t\leq e\leq t+1$.
By (\ref{GZarZ}), (\ref{GetQZ}), (\ref{emcvs})
and the assumption,
we have
\begin{align*}
E^{u, i}_{2}
=
\begin{cases}
E^{i}
& (\textrm{if}~~u=0) \\
\\
0
& (\textrm{if}~~e\geq 3~~\textrm{and}~~0<u\leq e-2)
\end{cases}
\end{align*}
for any integer $i$. So the isomorphism (\ref{E2eivan}) follows from Lemma \ref{Spbor}.
This completes the proof.
\end{proof}
\begin{rem}\upshape\label{ExaCond2}
If $\mathcal{F}_{X}=\mathbb{G}_{m, X}$ and $N=1$, then
$(\mathcal{F}_{X}, N)\in \mathbb{S}_{X_{\et}}\times \mathbb{N}_{\geq 0}$ 
satisfies Condition \ref{cdpb} by \cite[p.116, III, Remark 3.11 (a)]{M}
and Proposition \ref{EtmotGer}.    
\end{rem}

Moreover, let us remark the following:
\begin{lem}\upshape\label{adfun}
Let $f: X\to Y$ and $g: Y\to Z$ be morphisms of schemes. Let
$\mathcal{H}\in \mathbb{S}(Z_{\et})$,
$\mathcal{G}\in \mathbb{S}(Y_{\et})$
and
$\mathcal{F}\in \mathbb{S}(X_{\et})$.
Let
\begin{align*}
\alpha\in \operatorname{Hom}_{\mathbb{S}(Y_{\et})}(g^{*}(\mathcal{H}), \mathcal{G}),
&&
\beta\in \operatorname{Hom}_{\mathbb{S}(X_{\et})}(f^{*}(\mathcal{G}), \mathcal{F})
\end{align*}
and
\begin{equation*}
\gamma\in \operatorname{Hom}_{\mathbb{S}(X_{\et})}((g\circ f)^{*}(\mathcal{H}), \mathcal{F}).
\end{equation*}

Let 
\begin{align*}
\alpha^{\prime}\in \operatorname{Hom}_{\mathbb{S}(Z_{\et})}(\mathcal{H}, g_{*}(\mathcal{G})),
&&
\beta^{\prime}\in \operatorname{Hom}_{\mathbb{S}(Y_{\et})}(\mathcal{G}, f_{*}\mathcal{F})
\end{align*}
and 
\begin{equation*}
\gamma^{\prime}\in\operatorname{Hom}_{\mathbb{S}(Z_{\et})}(\mathcal{H}, (g\circ f)_{*}\mathcal{F})
\end{equation*}
be the right adjoints of $\alpha$, $\beta$ and $\gamma$. Then the following are equivalent:
\begin{enumerate}
\item[(i)] 
$\beta\circ f^{*}(\alpha)$
agrees with $\gamma$.
\item[(ii)]
$g_{*}(\beta^{\prime})\circ \alpha^{\prime}$
agrees with $\gamma^{\prime}$.
\end{enumerate}
\end{lem}
\begin{proof}\upshape
Since the proof of the implication (ii) $\Rightarrow$ (i)
is similar to that of the implication (i) $\Rightarrow$ (ii), 
it suffices to show that (i) implies (ii).
We have the commutative diagram
\begin{equation*}
\xymatrix{
\mathcal{H}
\ar[r]
&
g_{*}g^{*}\mathcal{H}
\ar[r]\ar[d]
&
g_{*}\mathcal{G} 
\ar[d]
\\
&
(g\circ f)_{*}(g\circ f)^{*}
\mathcal{H}
\ar[r]
&
g_{*}(f_{*}f^{*})\mathcal{G}
\ar[r]
&
(g\circ f)_{*}\mathcal{F}
}    
\end{equation*}
and the composition of the lower sequence agrees with $(g\circ f)_{*}\gamma$ by (i).
So (ii) holds and the statement follows.
\end{proof}
\begin{lem}\upshape\label{heninj}
Let $Y$ be a normal crossing variety over the spectrum of a field $k$, 
$A$ the henselization of the local ring $\mathcal{O}_{Y, y}$
at a point $y\in Y$
and
$\mathcal{F}_{X}$
an \'{e}tale sheaf on an object $X$ in $\mathcal{V}$
which satisfies Condition \ref{cdpb}.
Then the homomorphism
\begin{equation*}
\HO^{N^{\prime}}_{\et}(
A, 
\mathcal{F}_{\spec(A)}
)    
\to
\HO^{N^{\prime}}_{\et}(
\kappa(x), 
\mathcal{F}_{\spec(\kappa(x))}
)   
\end{equation*}
is injective for
$x\in \spec(A)^{(0)}$
and
$N^{\prime}\geq N$.
\end{lem}
\begin{proof}\upshape
By Condition \ref{cdpb},
we have a commutative diagram
\begin{equation*}
\xymatrix{
\HO^{N^{\prime}}_{\et}(A,
\mathcal{F}_{\spec(A)})
\ar[r]
\ar[d]
&
\HO^{N^{\prime}}_{\et}(
\kappa(x),
\mathcal{F}_{\spec(\kappa(x))}
)
\ar@{=}[d]
\\
\HO^{N^{\prime}}_{\et}(
A/x,
\mathcal{F}_{\spec(A/x)}
)
\ar[r]
&
\HO^{N^{\prime}}_{\et}(
\kappa(x),
\mathcal{F}_{\spec(\kappa(x))}
).
}    
\end{equation*}
Then the left map is an isomorphism by Condition \ref{cdpb}. Moreover, 
the lower map is injective by Condition \ref{cdpb} and Lemma \ref{adfun}. So the statement follows.
\end{proof}
\begin{lem}\upshape\label{sredis}
Let $Y$ be a normal crossing variety over the spectrum of a field $k$,
$i: Z\to Y$  a closed immersion with
a normal crossing variety $Z$ over the spectrum of a field $k$
and $C$ the henselization of the local ring
$\mathcal{O}_{Z, z}$
at a point $z\in Z$.

Let $(\mathcal{F}_{X}, N)$ be a pair of an \'{e}tale sheaf on an object
in $\mathcal{V}$
and a non-negative integer which satisfies Condition \ref{cdpb}.
Let $q$ be a prime number. Put
\begin{align*}
T:=
\left\{
\begin{array}{ll}
0  &  \textrm{if $\mathcal{F}_{Y}$ is a sheaf of torsion groups}  \\ [5pt]
1   &  \textrm{otherwise}
\end{array}
\right.
\end{align*}
and $N_{1}=\operatorname{max}\{N, T\}$. Suppose that the sequence
\begin{align}\label{CTex}
0\to 
\HO^{N^{\prime}}_{\et}(
C,
\mathcal{F}_{\spec(C)}
)_{T_{q}}
\to
\bigoplus_{x\in\spec(C)^{(0)}}
\HO^{N^{\prime}}_{x}(
C_{\et},
\mathcal{F}_{\spec(C)}
)_{T_{q}} 
\nonumber
\\
\to
\bigoplus_{x\in \spec(C)^{(1)}}
\HO^{N^{\prime}+1}_{x}(
C_{\et}, \mathcal{F}_{\spec(C)}
)_{T_{q}}
\to \cdots
\end{align}
is exact for any point $z\in Z$ and
$N^{\prime}\geq N_{1}+1$. Then we have an isomorphism
\begin{equation}\label{sisofit}
\HO^{N^{\prime}+s}_{x}(
Z_{\et},
i^{*}\mathcal{F}_{Y}
)_{T_{q}}   
\xrightarrow{\simeq}
\HO^{N^{\prime}+s}_{x}(
Z_{\et},
\mathcal{F}_{Z}
)_{T_{q}}
\end{equation}
for $N^{\prime}\geq N_{1}$ and
$x\in Y^{(s)}\cap Z$.
\end{lem}
\begin{proof}\upshape
Let $B$ be the henselization of the local ring
$\mathcal{O}_{Y, z}$ at a point $z\in Y\cap Z$.
By \cite[p.93, III, Corollary 1.28]{M}, we may prove the statement in the case where $Z$ is replaced
by $\spec(C)$ with $\operatorname{dim}(C)=s$
and $Y$ is replaced by $\spec(B)$.
We prove the statement by induction on $s$. 

In the case where $s=0$, we have isomorphisms
\begin{equation*}
\HO^{N^{\prime}}_{x}(
C_{\et},
i^{*}\mathcal{F}_{\spec(B)}
)
\simeq 
\HO^{N^{\prime}}_{x}(
C_{\et},
\mathcal{F}_{\spec(C)}
)
\simeq 
\HO^{N^{\prime}}_{\et}(
\kappa(x),
\mathcal{F}_{\spec(\kappa(x))}
)
\end{equation*}
for $x\in \spec(C)^{(0)}$ and $N^{\prime}\geq N$. 
So the statement holds in this case.

In the case where $s=1$, 
if $\mathcal{F}^{\prime}$ is $i^{*}\mathcal{F}_{\spec(B)}$ or
$\mathcal{F}_{\spec(C)}$,
then the sequence
\begin{align*}
0\to 
\HO^{N^{\prime}}_{\et}(C, \mathcal{F}^{\prime})
\to 
\bigoplus_{x\in \spec(C)^{(0)}}
\HO^{N^{\prime}}_{x}(
C_{\et},
\mathcal{F}^{\prime}
)
\\
\to
\bigoplus_{x\in \spec(C)^{(1)}}
\HO^{N^{\prime}+1}_{x}(
C_{\et},
\mathcal{F}^{\prime}
)
\to 0
\end{align*}
is exact for $N^{\prime}\geq N$ 
by Condition \ref{cdpb} and Lemma \ref{heninj}. 
Since the statement holds in the case where 
$s=0$, 
the statement follows from Condition \ref{cdpb}. 

Assume that the statement holds in the case where $s\leq s^{\prime}$ and $s^{\prime}\geq 1$. 
Suppose that
$\operatorname{dim}(C)=s^{\prime}+1$.
Let $\mathcal{F}^{\prime}$ be
$i^{*}\mathcal{F}_{\spec(B)}$ or
$\mathcal{F}_{\spec(C)}$. Let $N^{\prime}\geq N_{1}$.
By the assumption, we have the isomorphism (\ref{sisofit})
for $x\in \spec(C)^{(s)}$ and $s\leq s^{\prime}$.
So it suffices to show that
the sequence
\begin{align*}
\bigoplus_{x\in\spec(C)^{(s^{\prime}-1)}}
\HO^{N^{\prime}+s^{\prime}-1}_{x}
(C_{\et}, \mathcal{F}^{\prime})_{T_{q}}
\to
\bigoplus_{x\in\spec(C)^{(s^{\prime})}}
\HO^{N^{\prime}+s^{\prime}}_{x}
(C_{\et}, \mathcal{F}^{\prime})_{T_{q}}
\\
\to
\bigoplus_{x\in\spec(C)^{(s^{\prime}+1)}}
\HO^{N^{\prime}+s^{\prime}+1}_{x}
(C_{\et}, \mathcal{F}^{\prime})_{T_{q}}
\to 0
\end{align*}
is exact.
Consider the spectral sequence
\begin{equation*}
E^{u, v}_{1}(\mathcal{G})
=
\bigoplus_{x\in \spec(C)^{(u)}}
\HO^{u+v}_{x}(C_{\et}, \mathcal{G})
\Rightarrow
E^{u+v}(\mathcal{G})
=
\HO^{u+v}_{\et}
(C, \mathcal{G})
\end{equation*}
for $\mathcal{G}\in\mathbb{S}(\spec(C)_{\et})$.
Then we have isomorphisms
\begin{align}\label{casevani}
E^{u, N^{\prime}+1}_{2}(i^{*}\mathcal{F}_{\spec(B)})_{T_{q}}
\simeq 
E^{u, N^{\prime}+1}_{2}(\mathcal{F}_{\spec(C)})_{T_{q}}
\simeq 
\left\{
\begin{array}{ll}
\HO^{N^{\prime}+1}_{\et}(
C,
\mathcal{F}^{\prime}
)_{T_{q}} 
& (u=0)
\\ [5pt]
0 & (u\leq  s^{\prime}-1)
\end{array}
\right.
\end{align}
by (\ref{CTex}) and the assumption. 
Moreover, we have an isomorphism
\begin{equation*}
\HO^{m}_{\et}(C, \mathcal{F}_{\spec(C)})
\simeq 
\HO^{m}_{\et}(\kappa(z),
i^{*}(\mathcal{F}_{\spec(C)})
)
\end{equation*}
for $m\geq 0$. 
So 
$\HO_{\et}^{N^{\prime}+s^{\prime}}(C, \mathcal{F}_{\spec(C)})$
is a torsion group
because
$N^{\prime}+s^{\prime}\geq 1$.
Hence the statement follows from Proposition \ref{exind}.
\end{proof}

Moreover,
let us remark the following facts:
\begin{lem}\upshape\label{DGC}
Let $\mathcal{A}$ be a small abelian category. Then we have the followings:
\begin{enumerate}
\item
Consider the following commutative diagram in $\mathcal{A}$:
\begin{equation*}
\xymatrix@=10pt{
&&0\ar[d]&& \\
& B_{1}\ar[r]\ar[d]& C_{1}\ar[r]\ar[d]
&D_{1}\ar[r]\ar[d]^{f}&0
\\
A_{2}\ar[r]\ar[d]& B_{2}\ar[r]\ar[d] & C_{2}\ar[r]\ar[d] & D_{2} \\
A_{3}\ar[r]\ar[d]& B_{3}\ar[r]
& C_{3}  \\
0& 
}
\end{equation*}
where all columns are exact and all rows are complex. Assume that the sequences
\begin{align*}
&C_{1}\to D_{1}\to 0   \\
&B_{2}\to C_{2}\to D_{2} \\
&A_{3}\to B_{3}\to C_{3}
\end{align*}
are exact. Then the morphism $f: D_{1}\to D_{2}$ is injective.
\item Consider the following diagram in $\mathcal{A}$:
\begin{equation*}
\xymatrix@=10pt{
&0 \\
&A_{2}\ar[u]\ar[r]^{g}
&B_{2}\ar[r]
&C_{2} \\
0\ar[r]
&A_{1}\ar[r]\ar[u]^{f}
&B_{1}\ar[r]\ar[u]
&C_{1}\ar[u]  \\
&
&B\ar[u]\ar@{=}[r]
&B\ar[u]_{h}
}    
\end{equation*}
where the sequences are exact. Then the morphism
\begin{equation*}
\delta: \operatorname{Ker}(h)
\to A_{1}
\end{equation*}
is defined by the commutative diagram.
Moreover, assume that 
\begin{equation*}
\operatorname{Im}(\delta)
\subset
\operatorname{Ker}(f).
\end{equation*}
Then the morphism $g: A_{2}\to B_{2}$ is injective.
\end{enumerate}
\end{lem}
\begin{proof}\upshape
By the Freyd-Mitchell Embedding Theorem, the statement follows by diagram chasing.
\end{proof}
\begin{lem}\upshape\label{ijdex}
Let
\begin{equation*}
\xymatrix{
Z\cap U 
\ar[r]^{i^{\prime}}
\ar[d]_{j^{\prime}}
&
U 
\ar[d]^{j} \\
Z \ar[r]_{i}
&
X
}
\end{equation*}
be Cartesian. 
If $j$ is an open immersion 
and $i$ is a closed immersion, 
then we have
\begin{align*}
i_{*}(j^{\prime})_{!}
&=
j_{!}(i^{\prime})_{*}.
\end{align*}
In particular, we have
\begin{equation*}
i_{*}(j^{\prime})_{!}
=
j_{!}
\end{equation*}
in the case where $U=Z\cap U\subset Z$.
\end{lem}
\begin{proof}\upshape
The statement follows from \cite[p.71, II, Corollary 3.5]{M} 
and the definition of the extension by zero (cf. \cite[p.78]{M}).
\end{proof}
\begin{lem}\upshape\label{CAL}
Let $Y$ be a smooth scheme over the spectrum of a field $k$
and $Z$ a normal crossing variety over $\spec(k)$
which is a closed scheme of $Y$ of codimension $1$.

Let $A$ be the henselization of the local ring 
$\mathcal{O}_{Y, y}$ of $Y$ at a point $y$ of $Y$,
$C$ the henselization of the local ring $\mathcal{O}_{Z, y}$
and
$i: \spec (C)\to \spec (A)$ the corresponding closed immersion.

Let $U=\spec(A)\setminus\spec(C)$ and 
$j: U\to\spec(A)$  the corresponding open immersion. 

Let $(\mathcal{F}_{X}, N)$ be a pair of
an \'{e}tale sheaf on an object
$X$ in $\mathcal{V}$ and a non-negative integer
which satisfies Condition \ref{cdpb}
and $q$ a prime number.
Suppose that 
$N^{\prime}\geq \operatorname{max}\{N, T\}=N_{1}$ and
$\operatorname{cd}_{q}(k)<\infty$.

If the sequence
\begin{align}\label{AexA}
0\to 
\HO^{t}_{\et}(C, \mathcal{F}_{\spec(C)})_{T_{q }}
\to
\bigoplus_{x\in \spec(C)^{(0)}}
\HO^{t}_{x}(C_{\et}, \mathcal{F}_{\spec(C)})_{T_{q}}
\to  \nonumber
\\
\bigoplus_{x\in \spec(C)^{(1)}}
\HO^{t+1}_{x}(C_{\et}, \mathcal{F}_{\spec(C)})_{T_{q}}
\to\cdots
\end{align}
is exact for any point $y\in Y$ and $t\geq N^{\prime}$, then a distinguished triangle
\begin{equation}\label{dtji}
\cdots
\to
j_{!}\mathcal{F}_{U}
\to
\mathcal{F}_{\spec(A)}
\to
i_{*}i^{*}\mathcal{F}_{\spec(A)}
\to
\cdots
\end{equation}
induces an exact sequence
\begin{equation}\label{CA}
0\to  
\HO_{x}^{t+s-1}(C_{\et}, 
\mathcal{F}_{\spec(C)}
)_{T_{q}}
\to
\HO_{x}^{t+s}(A_{\et}, 
j_{!}\mathcal{F}_{U}
)_{T_{q}}
\to
\HO_{x}^{t+s}(A_{\et}, 
\mathcal{F}_{\spec(A)}
)_{T_{q}}
\to
0
\end{equation}
for $x\in (\spec(A))^{(s)}\cap \spec(C)$ and the sequence
\begin{align}\label{AexC}
0\to
\HO^{t}_{\et}(A, j_{!}\mathcal{F}_{U})_{T_{q}}
\to
\bigoplus_{x\in\spec(A)^{(0)}}
\HO^{t}_{x}(
A_{\et},
j_{!}\mathcal{F}_{U}
)_{T_{q}}
\to  
\nonumber
\\
\bigoplus_{x\in\spec(A)^{(1)}}
\HO^{t+1}_{x}(A_{\et},
j_{!}\mathcal{F}_{U})_{T_{q}}
\to\cdots
\end{align}
is exact for $t\geq N^{\prime}$.
\end{lem}
\begin{proof}\upshape
We prove the exactness of the sequences (\ref{CA}) and (\ref{AexC}) by downward induction on $N^{\prime}$.
By \cite[Lemma 3.7]{Sak5} and Lemma \ref{sredis}, we have an isomorphism 
\begin{equation*}
\HO^{u}_{x}(
A_{\et},
i_{*}i^{*}
\mathcal{F}_{\spec(A)}
)_{T_{q}}
\simeq
\left\{
\begin{array}{ll}
\HO_{x}^{u}(C_{\et}, 
\mathcal{F}_{\spec(C)}
)_{T_{q}}, 
& x\in
(\spec(A))^{(s)}\cap \spec(C)
\\ [5pt]
0, 
& x\in
(\spec(A))^{(s)}\setminus \spec(C)
\end{array}
\right.
\end{equation*}
for $u\geq N^{\prime}+s$.
So the distinguished triangle (\ref{dtji}) induces
an exact sequence
\begin{align*}
&\HO^{N^{\prime}+s-1}_{x}(
C_{\et},
\mathcal{F}_{\spec(C)}
)_{T_{q}}
\to
\HO^{N^{\prime}+s}_{x}(
A_{\et},
j_{!}\mathcal{F}_{U}
)_{T_{q}}
\to
\HO^{N^{\prime}+s}_{x}(
A_{\et},
\mathcal{F}_{\spec(A)}
)_{T_{q}}
\\
\to
&\HO^{N^{\prime}+s}_{x}(
C_{\et},
\mathcal{F}_{\spec(C)}
)_{T_{q}}
\to\cdots
\end{align*}
for $x\in (\spec(A))^{(s)}$. 
Let $\mathcal{G}$ be a sheaf of abelian groups on 
$\spec(A)_{\et}$ and $x\in\spec(A)^{(s)}$.
Then
\begin{equation*}
\HO^{t+s}_{x}(A_{\et},
\mathcal{G})_{T_{q}}
=0
\end{equation*}
for sufficiently large $t$
by Proposition \ref{cdpt} and Remark \ref{cdsp}.
Moreover, we have
\begin{equation*}
\HO^{t}
(A_{\et}, j_{!}\mathcal{F}_{U})
=0
\end{equation*}
for any $t\geq 0$ by the isomorphism (\ref{isoAI}).
So we have the exact sequences (\ref{CA}) and
(\ref{AexC}) for sufficiently large $N^{\prime}$.

Assume that the sequences (\ref{CA}) and (\ref{AexC}) are exact for $t> N^{\prime} (\geq N_{1})$.
If the first map in the sequence (\ref{CA})
is injective for $t=N^{\prime}+1$,
the second map in the sequence (\ref{CA})
is surjective for $t=N^{\prime}$.
So, in order to prove that the sequence (\ref{CA}) is exact for $t=N^{\prime}$, 
it suffices to show that the first map in the sequence
(\ref{CA}), that is,
the homomorphism
\begin{equation}\label{ijinj}
\HO^{t+s-1}_{x}(C_{\et}, \mathcal{F}_{\spec(C)})_{T_{q}}
\to 
\HO^{t+s}_{x}(
A_{\et},
j_{!}\mathcal{F}_{U}
)_{T_{q}}
\end{equation}
is injective for $t=N^{\prime}$.

We prove the injectivity of the homomorphism (\ref{ijinj}) for $t=N^{\prime}$ by induction on $s$.
By \cite[p.93, III, Corollary 1.28]{M}, it suffices to prove the injectivity of
the homomorphism (\ref{ijinj}) in the case where $A$ is a henselian regular local ring 
with $\operatorname{dim}(A)=s$.

In the case where $s=0$, $(\spec(A))^{(0)}\cap \spec(C)=\emptyset$.
So the homomorphism (\ref{ijinj}) is injective.

In the case where $s=1$, we prove the injectivity of the homomorphism (\ref{ijinj}).
Since $A$ is a henselian valuation ring, 
we have an isomorphism
\begin{equation*}
\HO^{u}_{\et}\left(
A,
j_{!}\mathcal{F}_{\spec(k(A))}
\right)
\simeq 0
\end{equation*}
for any integer $u\geq 0$ by the isomorphism (\ref{isoAI})
and so we have an isomorphism
\begin{equation*}
\HO^{N^{\prime}}_{\et}(
k(A),
\mathcal{F}_{\spec(k(A))}
)    
\xrightarrow{\simeq}
\HO^{N^{\prime}+1}_{x}(
A, 
j_{!}\mathcal{F}_{\spec(k(A))}
)
\end{equation*}
for $x\in (\spec(A))^{(1)}\cap \spec(C)$.
Since the diagram
\begin{equation}\label{AntAC}
\xymatrix{
\HO^{N^{\prime}}_{\et}(
A,
\mathcal{F}_{\spec(A)}
)
\ar[r]\ar[d]_{\simeq}
&
\HO^{N^{\prime}}_{\et}(
k(A),
\mathcal{F}_{\spec(k(A))}
)
\ar[d]^{\simeq}
\\
\HO^{N^{\prime}}_{x}(
C_{\et},
\mathcal{F}_{\spec(C)}
)
\ar[r]
&
\HO^{N^{\prime}+1}_{x}(
A_{\et},
j_{!}
\mathcal{F}_{\spec(k(A))}
)
}    
\end{equation}
is anti-commutative by \cite[Lemma 3.11]{Sak5}, the homomorphism (\ref{ijinj}) is injective by
Condition \ref{cdpb} and Lemma \ref{heninj}.

Next, 
we prove the injectivity of the homomorphism (\ref{ijinj})
in the case where $s=2$. 
Put
\begin{equation*}
E^{u, v}_{1}(D, \mathcal{G})
=
\displaystyle\bigoplus_{x\in \spec(D)^{(u)}}
\HO^{u+v}_{x}(
D_{\et},
\mathcal{G}
)
\end{equation*}
for an equidimensional ring $D$ and 
a sheaf $\mathcal{G}$ on $\spec(D)_{\et}$.
Then we have a commutative diagram
\footnotesize
\begin{equation*}
\xymatrix{
&0 \\
&
E^{1, N^{\prime}}_{1}\left(
C,
\mathcal{F}_{\spec(C)}
\right)_{T_{q}}
\ar[u]\ar[r]
&
E^{2, N^{\prime}}_{1}\left(
A,
j_{!}\mathcal{F}_{U}
\right)_{T_{q}}
\ar[r]
&
E^{2, N^{\prime}}_{1}\left(
A,
\mathcal{F}_{\spec(A)}
\right)_{T_{q}}
\\
0\ar[r]
&
E^{0, N^{\prime}}_{1}\left(
C,
\mathcal{F}_{\spec(C)}
\right)_{T_{q}}
\ar[r]\ar[u]
&
E^{1, N^{\prime}}_{1}\left(
A,
j_{!}\mathcal{F}_{U}
\right)_{T_{q}}
\ar[r]\ar[u]
&
E^{1, N^{\prime}}_{1}\left(
A,
\mathcal{F}_{\spec(A)}
\right)_{T_{q}}
\ar[u]  \\
&
&
E^{0, N^{\prime}}_{1}\left(
A,
j_{!}\mathcal{F}_{U}
\right)_{T_{q}}
\ar[u]\ar@{=}[r]
&
E^{0, N^{\prime}}_{1}\left(
A,
\mathcal{F}_{\spec(A)}
\right)_{T_{q}}
\ar[u]
}    
\end{equation*}
\normalsize
where the columns are exact by 
Remark \ref{HenTor}, the assumption and Proposition \ref{exind}.
Moreover, the middle row is exact by the above.
So the homomorphism (\ref{ijinj}) is injective by the anti-commutative diagram (\ref{AntAC}) 
and Lemma 
\ref{DGC}.2.

Assume that the homomorphism (\ref{ijinj}) is injective for $s\leq u$ and $u\geq 2$.
Then we prove the injectivity of the homomorphism (\ref{ijinj}) for $s=u+1$. Consider
the following diagram:
%
\fontsize{7.7pt}{7.7pt}\selectfont
\begin{equation*}
\xymatrix@C=8pt{
&
& 0
\ar[d]
&
& 
\\
&
E^{u-2, N^{\prime}}_{1}(
C,
\mathcal{F}_{\spec(C)}
)_{T_{q}}
\ar[r]
\ar[d]
&
E^{u-1, N^{\prime}}_{1}(
C,
\mathcal{F}_{\spec(C)}
)_{T_{q}}
\ar[r]
\ar[d]
&
E_{1}^{u, N^{\prime}}(
C,
\mathcal{F}_{\spec(C)}
)_{T_{q}}
\ar[r]
\ar[d]
& 0 \\
E_{1}^{u-2, N^{\prime}}(
A,
j_{!}\mathcal{F}_{U}
)_{T_{q}}
\ar[r]
\ar[d]
&
E_{1}^{u-1, N^{\prime}}(
A,
j_{!}\mathcal{F}_{U}
)_{T_{q}}
\ar[r]
\ar[d]
&
E_{1}^{u, N^{\prime}}(
A,
j_{!}\mathcal{F}_{U}
)_{T_{q}}
\ar[r]
\ar[d]
& 
E_{1}^{u+1, N^{\prime}}(
A,
j_{!}\mathcal{F}_{U}
)_{T_{q}}
\\
E_{1}^{u-2, N^{\prime}}(
A,
\mathcal{F}_{\spec(A)}
)_{T_{q}}
\ar[r]
\ar[d]
&
E_{1}^{u-1, N^{\prime}}(
A,
\mathcal{F}_{\spec(A)}
)_{T_{q}}
\ar[r]
&
E_{1}^{u, N^{\prime}}(
A,
\mathcal{F}_{\spec(A)}
)_{T_{q}}
&
\\
0
}    
\end{equation*}
\normalsize
where 
the upper row and
the bottom row are exact by
Condition \ref{cdpb} and the assumption.
Moreover, the sequence
\begin{align*}
E^{u-1, N^{\prime}}(
A,
j_{!}\mathcal{F}_{U}
)_{T_{q}}    
\to
E^{u, N^{\prime}}(
A,
j_{!}\mathcal{F}_{U}
)_{T_{q}}    
\to
E^{u+1, N^{\prime}}(
A,
j_{!}\mathcal{F}_{U}
)_{T_{q}}    
\end{align*}
is exact by Remark \ref{HenTor} and Proposition \ref{exind}. 
By the assumption, the columns except for the right column are also exact.
So the homomorphism (\ref{ijinj}) is injective for $s=u+1$ by Lemma \ref{DGC}.1.
Hence the sequence (\ref{CA}) is exact for $x\in (\spec(A))^{(s)}\cap \spec(C)$.
Finally, we prove the exactness of the sequence (\ref{AexC}).
Since we have an anti-commutative diagram
\begin{equation*}
\xymatrix{
\displaystyle\bigoplus_{x\in\spec(C)^{(0)}}\HO_{x}^{N^{\prime}}(C_{\et}, \mathcal{F}_{\spec(C)})
\ar[r]
&
\displaystyle\bigoplus_{x\in\spec(A)^{(1)}}\HO_{x}^{N^{\prime}+1}
(A_{\et}, j_{!}\mathcal{F}_{U})
\\
\HO_{\et}^{N^{\prime}}(C, \mathcal{F}_{\spec(C)})
\ar[u]
&
\displaystyle\bigoplus_{x\in\spec(A)^{(0)}}\HO_{x}^{N^{\prime}}
(A_{\et}, j_{!}\mathcal{F}_{U})
\ar[u]
\\
&
\HO^{N^{\prime}}_{\et}(A, \mathcal{F}_{\spec(A)})
\ar[u]\ar[lu]
}
\end{equation*}
for $N^{\prime}\geq N$ by the anti-commutative diagram (\ref{AntAC}),
the exactness of the sequence (\ref{AexC}) follows from Condition \ref{cdpb}
and the exact sequences (\ref{AexA}) and (\ref{CA}). 
This completes the proof.
\end{proof}

We prove Theorem \ref{Ex&Ger}. The proof of the following is similar to the proof of
\cite[Theorem 1.3]{Sak5}.

\begin{thm}\upshape\label{1gerl}
Let $Y$ be a normal crossing variety over 
the spectrum of a field $k$.
Let 
\begin{math}
Y_{1}, \cdots, Y_{a}
\end{math}
be the irreducible components of $Y$.

Let $(\mathcal{F}_{X}, N)$ be a pair of
an \'{e}tale sheaf of abelian groups
on an object $X$ in $\mathcal{V}$ 
and a non-negative integer
which satisfies 
Condition \ref{cdpb}. 
Let the notation $T$ be the same as in Proposition \ref{cst}
and $N_{1}=\operatorname{max}\{N, T\}$.
Suppose that
$\operatorname{cd}_{q}(k)<\infty$ for a prime number $q$.

Then we have the followings:

\begin{itemize}
\item Property $\operatorname{P}_{1}(a)$:
Let $s\geq 0$ be an integer and
$i: Z\to Y$ a closed immersion with
\begin{equation*}
Z
=
\displaystyle\bigcup^{a-1}_{m=1} Y_{m}.
\end{equation*}
Then the homomorphism
\begin{equation*}
\HO^{N^{\prime}+s}_{y}(
Z_{\et},
i^{*}\mathcal{F}_{Y}
)_{T_{q}}    
\xrightarrow{\sim}
\HO^{N^{\prime}+s}_{y}(
Z_{\et},
\mathcal{F}_{Z}
)_{T_{q}}
\end{equation*}
is an isomorphism for
$N^{\prime}\geq N_{1}$ and
$y\in Y^{(s)}\cap Z$.
\item Property $\operatorname{P}_{2}(a)$:
Let the notations be the same as above
and $j: U\hookrightarrow Y$ the immersion
of the open complement
of $Z$ in $Y$.
Then 
a distinguished triangle
\begin{align}\label{Indisji}
\cdots\to  
j_{!}\mathcal{F}_{U}
\to
\mathcal{F}_{X}
\to
i_{*}i^{*}\mathcal{F}_{X}
\to\cdots.
\end{align}
induces an exact sequence
\begin{equation}\label{Inindij}
0\to
\HO^{N^{\prime}+s}_{y}(
Y_{\et}, j_{!}\mathcal{F}_{U}
)_{T_{q}}
\to
\HO^{N^{\prime}+s}_{y}(
Y_{\et}, 
\mathcal{F}_{Y}
)_{T_{q}}
\to 
\HO^{N^{\prime}+s}_{y}(Z_{\et},
\mathcal{F}_{Z})_{T_{q}}
\to 0
\end{equation}
for $y\in Y^{(s)}\cap Z$ and $N^{\prime}\geq N_{1}$. 

\item Property $\operatorname{P}_{3}(a)$: Let $A$ be the henselization of the local ring
$\mathcal{O}_{Y, y}$ of $Y$ at a point of $y\in Y$. Then
the sequence
\begin{align*}
0\to \HO^{N^{\prime}}_{\et}(A, \mathcal{F}_{\spec(A)})_{T_{q}}
\to
\bigoplus_{x\in \spec(A)^{(0)}}
\HO^{N^{\prime}}_{x}(A_{\et}, 
\mathcal{F}_{\spec(A)})_{T_{q}}
\\
\to
\bigoplus_{x\in \spec(A)^{(1)}}
\HO^{N^{\prime}+1}_{x}(A_{\et}, 
\mathcal{F}_{\spec(A)})_{T_{q}} 
\to
\cdots
\end{align*}
is exact for $N^{\prime}\geq N_{1}$.
\end{itemize}
\end{thm}
\begin{proof}\upshape
We prove the statement by induction on $a$.

In the case where $a=1$,
then $Z=\phi$ and so $\operatorname{P}_{1}(1)$ and $\operatorname{P}_{2}(1)$ hold. 
Since 
$Y$ is a smooth scheme over the spectrum of a field,
$\operatorname{P}_{3}(1)$ holds by Condition
\ref{cdpb}.

Assume that $\operatorname{P}_{u}(v)$ holds for
$u=1, 2, 3$ and
$v\leq a$. Then we prove
$\operatorname{P}_{u}(a+1)$ for $u=1, 2, 3$.

First, we prove $\operatorname{P}_{1}(a+1)$.  
By \cite[p.93, III, Corollary 1.28]{M},
it suffices to prove $\operatorname{P}_{1}(a+1)$
in the case where $Y$ is the spectrum of a henselian local
ring $A$ and $\operatorname{dim}(Y)=s$.
Since $\#(Z^{(0)})=a$, $\operatorname{P}_{1}(a+1)$ follows from
$\operatorname{P}_{3}(a)$
and Lemma \ref{sredis}.

Next we prove $\operatorname{P}_{2}(a+1)$
and $\operatorname{P}_{3}(a+1)$
by downward induction on $N^{\prime}$.
Since $\operatorname{P}_{1}(a+1)$ holds, 
the distinguished triangle (\ref{Indisji})
induces an exact sequence
\begin{align}\label{FN1ex}
&\HO_{y}^{N_{1}+s}(Y_{\et}, j_{!}\mathcal{F}_{U})_{T_{q}}
\to 
\HO_{y}^{N_{1}+s}(Y_{\et}, \mathcal{F}_{Y})_{T_{q}}
\to
\HO_{y}^{N_{1}+s}(Z_{\et}, \mathcal{F}_{Z})_{T_{q}}
\nonumber
\\
\to 
&
\HO_{y}^{N_{1}+s+1}(Y_{\et}, j_{!}\mathcal{F}_{U})_{T_{q}}
\to \cdots
\end{align}
for $y\in Y^{(s)}\cap Z$.

Let
$x\in Y^{(s)}$ and $\mathcal{G}$ be a sheaf of abelian groups 
on $Y_{\et}$.
Since $\operatorname{cd}_{q}(k)$
and $\operatorname{tr.deg}_{k}\mathcal{O}_{Y, y}$ are finite
for any $y\in Y^{(0)}$,
we have
\begin{equation*}
\HO^{N^{\prime}+s}_{x}(Y_{\et}, 
\mathcal{G})_{T_{q}}
=0
\end{equation*}
for sufficiently large $N^{\prime}$ 
by Proposition \ref{cdpt} and Remark \ref{cdsp}. 
Moreover, we have
\begin{equation*}
\HO^{N^{\prime}}_{\et}(A, \mathcal{F}_{\spec(A)})_{T_{q}}=0    
\end{equation*}
for sufficiently large $N^{\prime}$
by (\ref{VaniEl}) and Remark \ref{cdsp}. So
the sequence (\ref{Inindij}) is exact
and $\operatorname{P}_{3}(a+1)$ holds
for sufficiently large $N^{\prime}$.

Assume that the sequence (\ref{Inindij}) is exact 
and $\operatorname{P}_{3}(a+1)$ holds
for $N^{\prime}> N^{\prime\prime} ( \geq N_{1})$. 
Then the sequence
\begin{align*}
\HO^{N^{\prime\prime}+s}_{y}(
Y_{\et},
j_{!}\mathcal{F}_{U}
)_{T_{q}}    
\to 
\HO^{N^{\prime\prime}+s}_{y}(
Y_{\et},
\mathcal{F}_{Y}
)_{T_{q}} 
\to 
\HO^{N^{\prime\prime}+s}_{y}(
Z_{\et},
\mathcal{F}_{Z}
)_{T_{q}} 
\to 0
\end{align*}
is exact for $y\in Y^{(s)}\cap Z$
by the exact sequence (\ref{FN1ex}).
In order to prove $\operatorname{P}_{2}(a+1)$,
it suffices to show that the homomorphism
\begin{equation}\label{injsupj2}
\HO^{N^{\prime\prime}+s}_{y}(
A_{\et},
j_{!}\mathcal{F}_{U}
)_{T_{q}}    
\to
\HO^{N^{\prime\prime}+s}_{y}(
A_{\et},
\mathcal{F}_{\spec(A)}
)_{T_{q}}
\end{equation}
is injective in the case where $A$ is the henselization at a point
$y$ of $Y^{(s)}\cap Z$ by \cite[p.93, III, Corollary 1.28]{M}.
Since we have
\begin{equation*}
U=
Y\setminus Z
=
Y_{a+1}\setminus
Z\cap Y_{a+1},
\end{equation*}
there is an open immersion 
$j^{\prime}: Y\setminus Z\to Y_{a+1}$. 
Let $i_{a+1}: Y_{a+1}\to Y$ 
be a closed immersion.
Then we have
\begin{equation*}
j=i_{a+1}\circ j^{\prime}.    
\end{equation*}
Let $A_{a+1}$ be the henselization of $\mathcal{O}_{Y_{a+1}, y}$
and 
$U^{\prime}=\spec(A_{a+1})\times_{Y_{a+1}}U$. 
For notational simplicity, 
a base change of $j$ (resp. $j^{\prime}$, $i$)
is also denoted by
$j$ (resp. $j^{\prime}$, $i$),
unless confusion arises.
Then we have an isomorphism
\begin{align}\label{eqj12}
&
\HO_{x}^{u}(
A_{\et},
j_{!}\mathcal{F}_{U^{\prime}}
)
=
\left\{
\begin{array}{ll}
\HO^{u}_{x}((A_{a+1})_{\et}, (j^{\prime})_{!}\mathcal{F}_{U^{\prime}})
& \textrm{for}~~x\in (Y_{a+1})^{(s)} \\ [5pt]
0
& \textrm{for}~~x\in Y^{(s)}\setminus Y_{a+1}
\end{array}
\right.
\end{align}
for any $u\geq 0$
by Lemma \ref{ijdex} (or \cite[p.169, V, Proposition 1.13]{M}) and \cite[Lemma 3.7]{Sak5}. 
By Condition \ref{cdpb} and 
$\operatorname{P}_{3}(a)$,
the sequences
\begin{align*}
0\to 
&\HO_{\et}^{N^{\prime}}(A_{a+1}, \mathcal{F}_{\spec(A_{a+1})})_{T_{q}}
\to
\bigoplus_{x\in\spec(A_{a+1})^{(0)}}
\HO^{N^{\prime}}_{x}(
(A_{a+1})_{\et},
\mathcal{F}_{\spec(A_{a+1})}
)_{T_{q}}  \\
\to
&\bigoplus_{x\in\spec(A_{a+1})^{(1)}}
\HO^{N^{\prime}+1}_{x}((A_{a+1})_{\et}, \mathcal{F}_{\spec(A_{a+1})})_{T_{q}}
\to\cdots
\end{align*}
and
\begin{align*}
0\to 
&\HO_{\et}^{N^{\prime}}(C_{a+1}, \mathcal{F}_{\spec(C_{a+1})})_{T_{q}}
\to
\bigoplus_{x\in\spec(C_{a+1})^{(0)}}
\HO^{N^{\prime}}_{x}(
(C_{a+1})_{\et},
\mathcal{F}_{\spec(C_{a+1})}
)_{T_{q}}  \\
\to
&\bigoplus_{x\in\spec(C_{a+1})^{(1)}}
\HO^{N^{\prime}+1}_{x}((C_{a+1})_{\et}, \mathcal{F}_{\spec(C_{a+1})})_{T_{q}}
\to\cdots   
\end{align*}
are exact for $N^{\prime}\geq N_{1}$
where $\spec(C_{a+1})=\spec(A_{a+1})\times_{Y}Z$. So the sequence
\begin{align*}
0\to 
&
\HO_{\et}^{N^{\prime}}(A_{a+1}, (j^{\prime})_{!}\mathcal{F}_{U^{\prime}})_{T_{q}}
\to
\bigoplus_{x\in\spec(A_{a+1})^{(0)}}
\HO^{N^{\prime}}_{x}(
(A_{a+1})_{\et},
(j^{\prime})_{!}\mathcal{F}_{U^{\prime}}
)_{T_{q}}  
\\
\to
&
\bigoplus_{x\in\spec(A_{a+1})^{(1)}}
\HO^{N^{\prime}+1}_{x}(
(A_{a+1})_{\et}, 
(j^{\prime})_{!}\mathcal{F}_{U^{\prime}})_{T_{q}}
\to\cdots
\end{align*}
is exact for $y\in (Y_{a+1})^{(s)}$
and $N^{\prime}\geq N_{1}$
by Lemma \ref{CAL}.
Hence the sequence
\begin{align}\label{exj1}
0\to 
&
\HO_{\et}^{N^{\prime}}(A, j_{!}\mathcal{F}_{U^{\prime}})_{T_{q}}
\to
\bigoplus_{x\in\spec(A)^{(0)}}
\HO^{N^{\prime}}_{x}(
A_{\et},
j_{!}\mathcal{F}_{U^{\prime}}
)_{T_{q}}  \nonumber 
\\
\to
&
\bigoplus_{x\in\spec(A)^{(1)}}
\HO^{N^{\prime}+1}_{x}(
A_{\et}, 
j_{!}\mathcal{F}_{U^{\prime}})_{T_{q}}
\to\cdots
\end{align}
is exact for $N^{\prime}\geq N_{1}$ by the isomorphism (\ref{eqj12}).

Then we prove that the homomorphism (\ref{injsupj2}) is injective
by induction on $s$.
Put 
\begin{math}
\spec(C)=
\spec(A)\times_{Y}Z
\end{math}
and $U^{\prime}=\spec(A)\times_{Y}U$.
Remark that we have an isomorphism
\begin{equation*}
\HO^{u}_{x}(
A_{\et}, j_{!}\mathcal{F}_{U^{\prime}}
)
\xrightarrow{\sim}
\HO^{u}_{x}(
A_{\et}, 
\mathcal{F}_{\spec(A)}
)
\end{equation*}
for $x\in U=Y\setminus Z$ and any integer $u\geq 0$.

In the case where $s=0$, then we have an isomorphism
\begin{equation*}
\HO_{y}^{u}(A_{\et}, 
\mathcal{F}_{\spec(A)}
)
\xrightarrow{\simeq}   
\HO_{y}^{u}(C_{\et}, i^{*}\mathcal{F}_{\spec(A)}
)
\end{equation*}
for $y\in Y^{(0)}\cap Z$ and $u\geq 0$.
So the homomorphism (\ref{injsupj2}) is injective for $s=0$.

In the case where $s=1$, we have the commutative diagram
\fontsize{7.8pt}{7.8pt}\selectfont
\begin{equation*}
\xymatrix@C=10pt{
&
&
0\ar[d]
&
&
\\
&
\HO^{N^{\prime\prime}}_{\et}(A, j_{!}\mathcal{F}_{U^{\prime}})_{T_{q}}
\ar[r]\ar[d]
&
\displaystyle\bigoplus_{x\in \spec(A)^{(0)}}
\HO^{N^{\prime\prime}}_{x}(A_{\et}, 
j_{!}\mathcal{F}_{U^{\prime}})_{T_{q}}
\ar[r]\ar[d]
&
\displaystyle\bigoplus_{x\in \spec(A)^{(1)}}
\HO^{N^{\prime\prime}+1}_{x}(A_{\et}, 
j_{!}\mathcal{F}_{U^{\prime}})_{T_{q}}
\ar[r]\ar[d]
&
0 \\
&
\HO^{N^{\prime\prime}}_{\et}(A, \mathcal{F}_{\spec(A)})_{T_{q}}
\ar[r]\ar[d]
&
\displaystyle\bigoplus_{x\in \spec(A)^{(0)}}
\HO^{N^{\prime\prime}}_{x}(A_{\et}, \mathcal{F}_{\spec(A)})_{T_{q}}
\ar[r]\ar[d]
&
\displaystyle\bigoplus_{x\in \spec(A)^{(1)}}
\HO^{N^{\prime\prime}+1}_{x}(A_{\et}, \mathcal{F}_{\spec(A)})_{T_{q}}
\\
0\ar[r]
&
\HO^{N^{\prime\prime}}_{\et}(C, 
\mathcal{F}_{\spec(C)})_{T_{q}}
\ar[r]
&
\displaystyle\bigoplus_{x\in \spec(C)^{(0)}}
\HO^{N^{\prime\prime}}_{x}(C_{\et}, \mathcal{F}_{\spec(C)})_{T_{q}}
}    
\end{equation*}
\normalsize
where the upper and the middle rows are exact by
the exact
sequence (\ref{exj1}).
Since $Y_{1}, \cdots, Y_{a}$ are the irreducible components of $Z$,
the lower row is exact by
$\operatorname{P}_{3}(a)$.
Moreover, the left and the middle columns are exact
by the distinguished triangle (\ref{Indisji}) and the assumption.
Hence the morphism (\ref{injsupj2}) is injective for $s=1$ 
by Lemma \ref{DGC}.

Let $t$ be an integer and $t\geq 1$.
Assume that $\operatorname{P}_{2}(a+1)$ holds for $s\leq t$
and $\operatorname{dim}(A)=t+1$.
Put
\begin{equation*}
E^{u, v}_{1}(D, \mathcal{G})
=
\displaystyle\bigoplus_{x\in \spec(D)^{(u)}}
\HO^{u+v}_{x}(
D_{\et},
\mathcal{G}
)
\end{equation*}
for an equidimensional ring $D$ and 
a sheaf $\mathcal{G}$ on $\spec(D)_{\et}$.
Then we have the commutative
diagram
\begin{equation}
\label{comBD}
\fontsize{9.6pt}{9.6pt}\selectfont
\xymatrix@C=10pt{
&
&
0\ar[d]
&
&
\\
&
E_{1}^{t-1, N^{\prime\prime}}(A, j_{!}\mathcal{F}_{U^{\prime}})_{T_{q}}
\ar[r]\ar[d]
&
E_{1}^{t, N^{\prime\prime}}(A, j_{!}\mathcal{F}_{U^{\prime}})_{T_{q}}
\ar[r]\ar[d]
&
E_{1}^{t+1, N^{\prime\prime}}(A, j_{!}\mathcal{F}_{U^{\prime}})_{T_{q}}
\ar[r]\ar[d]
&
0 \\
B_{t}
\ar[r]\ar[d]
&
E_{1}^{t-1, N^{\prime\prime}}(A, \mathcal{F}_{\spec(A)})_{T_{q}}
\ar[r]\ar[d]
&
E_{1}^{t, N^{\prime\prime}}(A, \mathcal{F}_{\spec(A)})_{T_{q}}
\ar[r]\ar[d]
&
E_{1}^{t+1, N^{\prime\prime}}(A, \mathcal{F}_{\spec(A)})_{T_{q}}
\\
D_{t}
\ar[r]\ar[d]
&
E_{1}^{t-1, N^{\prime\prime}}(C, \mathcal{F}_{\spec(C)})_{T_{q}}
\ar[r]
&
E_{1}^{t, N^{\prime\prime}}(C, \mathcal{F}_{\spec(C)})_{T_{q}}  \\
0
&
}    
\end{equation}
\normalsize
where 
%
\begin{align*}
B_{t}=
\begin{cases}
\HO^{N^{\prime\prime}}_{\et}(A, 
\mathcal{F}_{\spec(A)})_{T_{q}}
& (\textrm{if}~~t=1) \\
\\
E_{1}^{t-2, N^{\prime\prime}}(A, 
\mathcal{F}_{\spec(A)})_{T_{q}}
& (\textrm{if}~~t\geq 2)
\end{cases}
\end{align*}
and
\begin{align*}
D_{t}=
\begin{cases}
\HO^{N^{\prime\prime}}_{\et}(C, 
\mathcal{F}_{\spec(C)})_{T_{q}}
& (\textrm{if}~~t=1) \\
\\
E_{1}^{t-2, N^{\prime\prime}}(C, 
\mathcal{F}_{\spec(C)}
)_{T_{q}}
& (\textrm{if}~~t\geq 2).
\end{cases}
\end{align*}
Then the columns except for the right map are exact by
Condition \ref{cdpb} 
and $\operatorname{P}_{2}(a+1)$ for $s\leq t$.
If $\mathcal{G}=j_{!}\mathcal{F}_{U^{\prime}}$ or $\mathcal{G}=\mathcal{F}_{\spec(A)}$, 
then the sequence
%
\begin{equation*}
E_{1}^{t-1, N^{\prime\prime}}(A, 
\mathcal{G}
)_{T_{q}}
\to
E_{1}^{t, N^{\prime\prime}}(A, 
\mathcal{G}
)_{T_{q}}
\to
E_{1}^{t+1, N^{\prime\prime}}(A, 
\mathcal{G}
)_{T_{q}}
\to
0
\end{equation*}
%
is exact by 
the exact sequence (\ref{exj1}),
Proposition \ref{exind}
and
$\operatorname{P}_{3}(a+1)$
for
$N^{\prime}\geq N^{\prime\prime}+1$.
Since $Y_{1}, \cdots, Y_{a}$ are the irreducible components of $Z$, 
the lower row in the diagram (\ref{comBD}) is exact by $\operatorname{P}_{3}(a)$.
So the homomorphism (\ref{injsupj2}) is injective for $s=t+1$ by Lemma \ref{DGC}.
Hence $\operatorname{P}_{2}(a+1)$ holds.

Finally, we prove $\operatorname{P}_{3}(a+1)$.
Let the notations $A$, $C$ and $U^{\prime}$ be the same as above.
By (\ref{exj1}) and $\operatorname{P}_{3}(a)$, the sequences
\begin{align*}
0\to 
&
\HO_{\et}^{N^{\prime}}(A, j_{!}\mathcal{F}_{U^{\prime}})_{T_{q}}
\to
\bigoplus_{x\in\spec(A)^{(0)}}
\HO^{N^{\prime}}_{x}(
A_{\et},
j_{!}\mathcal{F}_{U^{\prime}}
)_{T_{q}}   
\\
\to
&
\bigoplus_{x\in\spec(A)^{(1)}}
\HO^{N^{\prime}+1}_{x}(
A_{\et}, 
j_{!}\mathcal{F}_{U^{\prime}})_{T_{q}}
\to\cdots  
\end{align*}
and
\begin{align*}
0\to 
&
\HO_{\et}^{N^{\prime}}(C, 
\mathcal{F}_{\spec(C)})_{T_{q}}
\to
\bigoplus_{x\in\spec(C)^{(0)}}
\HO^{N^{\prime}}_{x}(
C_{\et},
\mathcal{F}_{\spec(C)}
)_{T_{q}}   
\\
\to
&
\bigoplus_{x\in\spec(C)^{(1)}}
\HO^{N^{\prime}+1}_{x}(
C_{\et}, 
\mathcal{F}_{\spec(C)})_{T_{q}}
\to\cdots
\end{align*}
are exact for $N^{\prime}\geq N_{1}$.
Since $A$ is henselian, we have isomorphisms
\begin{equation*}
\HO^{u}_{\et}(
A,
j_{!}\mathcal{F}_{U^{\prime}}
)    
=0
\end{equation*}
for any $u\geq 0$ and
\begin{equation*}
\HO^{N^{\prime}}_{\et}(A,
\mathcal{F}_{\spec(A)})
\xrightarrow{\sim}
\HO^{N^{\prime}}_{\et}(C,
\mathcal{F}_{\spec(C)}
)
\end{equation*}
for $N^{\prime}\geq N$
by Condition \ref{cdpb} and the isomorphism (\ref{isoAI}).
Hence $\operatorname{P}_{3}(a+1)$ follows from
$\operatorname{P}_{2}(a+1)$. This completes the proof.
\end{proof}
\begin{cor}\upshape\label{exBGO}
Let the notations be the same as in Theorem \ref{1gerl}. Suppose that $N_{1}=0$. Then the sequence 
\begin{align*}
0\to \HO^{t}_{\et}(A, \mathcal{F}_{\spec(A)})_{T_{q}}
\to
\bigoplus_{x\in \spec(A)^{(0)}}
\HO^{t}_{x}(A_{\et}, 
\mathcal{F}_{\spec(A)})_{T_{q}}
\\
\to
\bigoplus_{x\in \spec(A)^{(1)}}
\HO^{t+1}_{x}(A_{\et}, 
\mathcal{F}_{\spec(A)})_{T_{q}} 
\to
\cdots
\end{align*}
is exact for any integer $t$.
Especially, we have 
\begin{equation*}
\HO_{x}^{s+u}(A_{\et}, \mathcal{F})_{T_{q}}
=0
\end{equation*}
for $x\in \spec(A)^{(s)}$ and
$u<0$.
\end{cor}
\begin{proof}\upshape
Consider the spectral sequence
\begin{align*}
E^{u, v}_{1}
= 
\bigoplus_{x\in \spec(A)^{(u)}}
\HO_{x}^{u+v}(A_{\et}, \mathcal{F}_{\spec(A)})
\Rightarrow
E^{u+v}
=
\HO_{\et}^{u+v}(A, \mathcal{F}_{\spec(A)}).
\end{align*}
By Theorem \ref{1gerl} and \cite[p.93, III, Corollary 1.28]{M}, 
it suffices to prove that
\begin{equation}\label{negvan}
(E^{s, t}_{1})_{T_{q}}=0    
\end{equation}
for $t<0$ and $\operatorname{dim}(A)=s$. We prove the statement by induction on 
$s$.

If $s=1$, then the isomorphism (\ref{negvan}) for $t=-1$ follows from Lemma \ref{heninj}.
So the isomorphism (\ref{negvan}) holds for $s=1$ and $t<0$.

Assume that we have the isomorphism (\ref{negvan}) for $s\leq s^{\prime}$ and $s^{\prime}\geq 1$.
Suppose that $\operatorname{dim}(A)=s^{\prime}+1$. Then it suffices to prove that
\begin{equation*}
(E_{2}^{s^{\prime}+1, t})_{T_{q}}=0    
\end{equation*}
for $t<0$ by the assumption. 

By Lemma \ref{heninj}, the homomorphism
\begin{equation*}
E^{s^{\prime}+t+1}
\to 
E^{0, s^{\prime}+t+1}_{1}
\end{equation*}
is injective. Moreover, we have
\begin{align*}
(E^{u, v}_{2})_{T_{q}}
=
\begin{cases}
(E^{u+v})_{T_{q}}
& (\textrm{if}~~u=0) \\
\\
0
& (\textrm{if}~~s^{\prime}\geq 2~~\textrm{and}~~0<u\leq s^{\prime}-1)
\end{cases}
\end{align*}
for $u+v=s^{\prime}+t$. Hence, the statement follows from Lemma \ref{borTq}. 
This completes the proof.

\end{proof}
\begin{lem}\upshape\label{limnc}
Let $A$ be the strict henselian of
a local ring of a normal crossing variety over the spectrum of a field 
of positive characteristic. Then $A$
is a direct limit of a family of 
the strict henselian rings $A_{i}$ ($i\in I$) of a local ring of a normal crossing variety over 
the spectrum of the prime field
which satisfies the following:

Let $A/(a_{s})$ $(s\in S)$ be the irreducible components of $A$. Then
\begin{equation*}
\operatorname{Im}(
A_{i}\to 
A\to
A/(a_{s})
)    
\end{equation*}
$(s\in S)$ are regular for any $i\in I$.
\end{lem}
\begin{proof}\upshape
By the definition of a normal crossing variety $Y$, there exists 
an \'{e}tale map of local rings 
$\mathcal{O}_{Y, y}\to A^{\prime}$ such that $A^{\prime}$ is also
a quotient of a local ring $C$ of a smooth algebra over a field at a point
and
$\operatorname{dim}(C)=\operatorname{dim}(A^{\prime})+1$.
By Quillen's method 
(cf.\cite[\S 7, The proof of Theorem 5.11]{Q}),
there exists a local ring $C^{\prime}$ of a regular algebra
of finite type over a field $k^{\prime}$ such that $k^{\prime}$
is finitely generated over the prime field and
\begin{math}
C=k\otimes_{k^{\prime}}C^{\prime}.    
\end{math}
Then we have isomorphism
\begin{equation*}
C=\displaystyle\lim_{
\substack{\to\\ i\in I}
}
k_{i}\otimes_{k^{\prime}}C^{\prime}
\end{equation*}
where $k_{i}$ run over  subfields of $k$ which contain
$k^{\prime}$ and are finitely generated over the prime field. Let
$\mathfrak{m}$ be the maximal ideal of $C$,
$f_{i}: k_{i}\otimes_{k^{\prime}}C^{\prime}\to C$ 
the base change
and
\begin{math}
\mathfrak{m}_{i}
=(f_{i})^{-1}(\mathfrak{m}).
\end{math}
Put
\begin{math}
C_{i}=\left(
k_{i}\otimes_{k^{\prime}}C^{\prime}
\right)_{\mathfrak{m}_{i}}.
\end{math}
Then $C_{i}$ is regular by \cite[p.182, Theorem 23.7 (i)]{Ma}. 
Let $c^{\prime}_{s}$ ($s\in S$) be elements of $C$
such that
\begin{math}
C/(c^{\prime}_{s})
\simeq    
A^{\prime}/(a^{\prime}_{s})
\end{math}
where
$A^{\prime}/(a^{\prime}_{s})$ ($s\in S$) are the irreducible components 
of $A^{\prime}$.
Since we are able to choose a directed subset
$J$ of $I$ such that $c_{s}^{\prime}\in \mathfrak{m}_{j}$ and
\begin{equation*}
\mathfrak{m}_{j}/\mathfrak{m}_{j}^{2}
\otimes_{c_{j}/\mathfrak{m}_{j}}C/\mathfrak{m}
\simeq 
\mathfrak{m}/\mathfrak{m}^{2}
\end{equation*}
for any $s\in S$ and $j\in J$, we may assume that $I=J$
and
\begin{math}
\operatorname{dim}(C)
=
\operatorname{dim}(C_{i})
\end{math}
for any $i\in I$.
Put
\begin{equation*}
A_{i}^{\prime}
=
\operatorname{Im}\left(
C_{i}\to
C\to
A^{\prime}
\right)
\end{equation*}
for any $i\in I$ and $A$ (resp. $A_{i}$) denotes the strict henselian of $A^{\prime}$ (resp. $A^{\prime}_{i}$). 
Then 
$A$ is also the strict henselian of $\mathcal{O}_{Y, y}$
and
$A_{i}$ satisfies the property of the statement.
This completes the proof.
\end{proof}
\begin{prop}\upshape\label{hvan}
Let $Y$ be a normal crossing variety over 
the spectrum of
a field 
of positive characteristic $p>0$ and $y\in Y^{(s)}$.
Then we have
\begin{equation*}
\HO_{y}^{t}(Y_{\et}, \lambda_{r}^{n})=0    
\end{equation*}
for $t<s$.
\end{prop}
\begin{proof}\upshape
Let $\{Y_{i}\}_{i\in I}$
be the irreducible components of $Y$. For an 
integer $m$, we define
\begin{align*}
Y^{
\langle m\rangle
}
:=
\bigsqcup
_{
\{i_{1}, i_{2}, \cdots, i_{m}
\}
\subset
I
}
Y_{i_{1}}\cap
Y_{i_{2}}\cap
\cdots
\cap
Y_{i_{m}}
\end{align*}
where 
the indices are pairwise distinct
for each subset
$\{i_{1}, \cdots, i_{m}\}\subset I$.
Then
there is an exact sequence on $Y_{\et}$
\begin{align*}
0\to   
\lambda^{n}_{Y, r}
\to
a_{1*}W_{r}\Omega^{n}_{Y^{\langle 1\rangle}, \log}
\to
a_{2*}W_{r}\Omega^{n}_{Y^{\langle 2\rangle}, \log}
\to
\cdots
\to
a_{N+1-n*}W_{r}\Omega^{n}_{Y^{\langle N+1-n\rangle}, \log}
\to 0
\end{align*}
by \cite[p.727, Proposition 3.2.1]{SaL} and Lemma \ref{limnc}.
Here 
$N=\operatorname{dim}(Y)$ and
$a_{m}$ is the canonical finite morphism
$Y^{\langle m\rangle}\to Y$. 
Put
\begin{align*}
\Lambda^{n, m}_{r}
=
\operatorname{Im}\left(
a_{m*}W_{r}\Omega^{n}_{Y^{\langle m\rangle}, \log}
\to
a_{m+1*}W_{r}\Omega^{n}_{Y^{\langle m+1\rangle}, \log}
\right).
\end{align*}
Then there is an exact sequence
\begin{equation*}
0\to
\Lambda^{n, m-1}_{r}
\to
a_{m*}W_{r}\Omega^{n}_{Y^{\langle m\rangle}, \log}
\to
\Lambda^{n, m}_{r}
\to 
0
\end{equation*}
by the definition.
Since $Y^{\langle m\rangle}$ is a regular scheme of positive characteristic, 
we have isomorphisms
\begin{align*}
\HO^{t}_{y}(
Y_{\et}, 
a_{m*}W_{r}\Omega^{n}_{X^{\langle m\rangle}, \log}
)
\simeq
\HO^{t}_{y}(
(Y^{\langle m\rangle})_{\et},
W_{r}\Omega^{n}_{X^{\langle m\rangle}, \log}
)
\simeq
0
\end{align*}
for $y\in Y^{(s)}$ and $t<s-m+1$
by
\cite[Lemma 3.7]{Sak5} and
\cite[p.584, Corollary 3.4]{Sh}.
Hence we have isomorphisms
\begin{align*}
&\HO^{t}_{y}(
Y_{\et},
\lambda^{n}_{r}
)
\simeq
\HO^{t-1}_{y}(
Y_{\et},
\Lambda^{n, 1}_{r}
)
\simeq
\cdots
\simeq
\HO^{t-N+n}_{y}(
Y_{\et},
\Lambda^{n, N-n}_{r}
)
\\
\simeq
&\HO^{t-N+n}_{y}(
Y_{\et},
a_{N+1-n *}W_{r}\Omega_{X^{\langle N+1-n\rangle}, \log}^{n}
)
\simeq 
0
\end{align*}
for $y\in Y^{(s)}$ and $t<s$.
This completes the proof.
\end{proof}
\begin{thm}\upshape\label{Gtcl}
Let $A$ be the henselization  
of the local ring
$\mathcal{O}_{Y, y}$ of a normal crossing variety $Y$ 
over  the spectrum of
a field of positive characteristic $p>0$
at a point $y\in Y$. 
Then the sequence
\begin{align}\label{agerl}
0\to \HO^{s}_{\et}(A, \lambda^{n}_{r})
\to
\bigoplus_{x\in \spec(A)^{(0)}}
\HO^{s}_{x}(A_{\et}, \lambda_{r}^{n})
\to
\bigoplus_{x\in \spec(A)^{(1)}}
\HO^{s+1}_{x}(A_{\et}, \lambda_{r}^{n}) 
\to
\cdots
\end{align}
is exact for any integers $n\geq 0$, $s$ and $r>0$.
\end{thm}
\begin{proof}\upshape
By \cite[Part 1, (1.2)]{C-H-K},
we have a spectral sequence
\begin{equation*}
E_{1}^{u, v}
=
\bigoplus_{x\in\spec (A)^{(u)}}
\HO^{u+v}_{x}(
A_{\et}, \lambda^{n}_{r})
\Rightarrow
E^{u+v}
=\HO^{u+v}_{\et}(A, \lambda^{n}_{r}).
\end{equation*}
If $s<0$, the sequence (\ref{agerl}) is exact by Proposition \ref{hvan}.
If $s\geq 1$, the sequence (\ref{agerl}) is exact by Remark \ref{ExaCond} and Theorem \ref{1gerl}.
So it suffices to show that 
the sequence (\ref{agerl}) is exact for $s=0$, that is,
\begin{align*}
&
E_{2}^{0, 0}=E^{0} 
&& \textrm{and} 
&E^{u, 0}_{2}=0 ~~(u>0).
\end{align*}
By \cite[Expos\'{e} X, Th\'{e}or\`{e}me 5.1]{SGA4}, 
we have
\begin{equation}\label{VaniEp}
E^{u, 0}_{\infty}=E^{u}=0    
\end{equation}
for $u\geq 2$. By Proposition \ref{hvan}, we have
\begin{equation}\label{E2v}
E^{u, v}_{r}=0    
\end{equation}
for $v<0$ and $r\geq 2$. 
If $u>2$, we have
\begin{equation*}
E^{u-r, r-1}_{r}=0    
\end{equation*}
for $r\geq 2$ by Remark \ref{ExaCond} and
Theorem \ref{1gerl}.
So we have
\begin{equation*}
E^{u, 0}_{2}=0    
\end{equation*}
for $u>2$.
By 
(\ref{E2v}), the sequence
\begin{align*}
0\to
E^{1, 0}_{2}\to
E^{1}\to
E^{0, 1}_{2}\to
E^{2, 0}_{2}
\end{align*}
is exact.
By Theorem \ref{1gerl}, we have an isomorphism
\begin{equation*}
E^{1}
\xrightarrow{\simeq}
E^{0, 1}_{2}
\end{equation*}
and so we have isomorphisms
\begin{align*}
& E^{1, 0}_{2}\simeq 0
&& \textrm{and}
& \operatorname{Im}(E^{0, 1}_{2}\to E^{2, 0}_{2})
\simeq 0.
\end{align*}
Hence we have isomorphisms
\begin{align*}
& E^{0}\xrightarrow{\simeq}
E^{0, 0}_{2}
&& \textrm{and}
& E^{u, 0}_{2}\simeq 0~~(u=1, 2)
\end{align*}
by (\ref{VaniEp}) and (\ref{E2v}). This completes the proof.
\end{proof}

\subsection{Mixed characteristic case}
Let $B$ be a discrete valuation ring  of mixed characteristic $(0, p)$ and 
$K$ the quotient field of $B$. Let $\mathfrak{X}$ be a semistable family
over $\spec(B)$, that is, a regular scheme of pure dimension which is flat
of finite type over $\spec(B)$,
the generic fiber 
$\mathfrak{X}\otimes_{B}K$ is smooth over $\spec(K)$
and the special fiber $Y$ of $\mathfrak{X}$ is a reduced divisor with normal
crossings on $\mathfrak{X}$.

As an application of Theorem \ref{Gtcl}, we prove the relative version of the Gersten-type conjecture for
the $p$-adic \'{e}tale $\mathfrak{T}_{r}(n)$ (cf. Definition \ref{DefT}) 
over the henselization of the local ring $\mathcal{O}_{\mathfrak{X}, x}$
of a semistable family $\mathfrak{X}$ over the spectrum of a discrete valuation ring
of mixed characteristic $(0, p)$
in the case where
$B$ contains $p^{r}$-th roots of unity
(cf. Theorem \ref{RelGer}).  

\begin{prop}\upshape\label{jtxiso}
Let $B$ be a discrete valuation ring of mixed characteristic $(0, p)$
and $\pi$ a prime element of $B$. Let $\mathfrak{X}$ be a semistable
family over $\spec(B)$ and $R$ the local ring of $\mathfrak{X}$ at a 
point $x$ of the closed fiber $Z$ of $\mathfrak{X}$.
Let $j: U\to \mathfrak{X}$ be the inclusion of
the generic fiber of $\mathfrak{X}$.
Then the homomorphism
\begin{equation*}
\HO^{q}_{x}(X_{\et}, j_{!}\mu_{p^{r}}^{\otimes n})
\xrightarrow{\simeq}
\HO^{q}_{x}(X_{\et}, \mathfrak{T}_{r}(n))
\end{equation*}
is an isomorphism
for $x\in X^{(s)}\cap Z$ and $q\geq s+n+2$.
Moreover,
the homomorphism
\begin{equation*}
\HO^{q}_{R/(\pi)}(R_{\et}, j_{!}\mu_{p^{r}}^{\otimes n})
\xrightarrow{\simeq}
\HO^{q}_{R/(\pi)}(R_{\et}, \mathfrak{T}_{r}(n))
\end{equation*}
is an isomorphism for $q\geq n+3$.
\end{prop}
\begin{proof}\upshape
Let $i: Z\to \mathfrak{X}$ be the inclusion of
the closed fiber of $\mathfrak{X}$.
Since we have a spectral sequence
\begin{equation*}
E_{2}^{s, t}
=
\HO_{\et}^{s}(
R/(\pi), \mathcal{H}^{t}(i^{*}\mathfrak{T}_{r}(n))
)
\Rightarrow
E^{s+t}
=
\HO^{s+t}_{\et}(
R/(\pi), i^{*}\mathfrak{T}_{r}(n)
),
\end{equation*}
we have an isomorphism
\begin{equation}\label{itvan}
\HO_{\et}^{q}(R/(\pi), i^{*}\mathfrak{T}_{r}(n))  
\simeq  
0
\end{equation}
for $q\geq n+2$
by \cite[Expos\'{e} X, Th\'{e}or\`{e}me 5.1]{SGA4}.
By the similar argument as above,
we have an isomorphism
\begin{equation}\label{itxvan}
\HO^{q}_{x}\left(
(R/(\pi))_{\et},
i^{*}\mathfrak{T}_{r}(n)
\right)
\simeq 0
\end{equation}
for $x\in (\spec(R))^{(s)}\cap Z$
and $q\geq n+s+1$ by
\cite[Proposition 2.6]{Sak4}.
Since we have a distinguished triangle
\begin{align*}
\cdots\to  
j_{!}\mu_{p^{r}}^{\otimes n}
\to \mathfrak{T}_{r}(n)
\to
i_{*}i^{*}\mathfrak{T}_{r}(n)
\to\cdots,
\end{align*}
the statement follows from
(\ref{itvan}) and (\ref{itxvan}).
\end{proof}
\begin{thm}\upshape\label{RelGer}
Let $B$ be a discrete valuation ring of mixed characteristic $(0, p)$
and $\pi$ a prime element of $B$. Let $\mathfrak{X}$ be a semistable
family over $\spec(B)$ and $R$ the henselization of the local ring $\mathcal{O}_{\mathfrak{X}, x}$ of 
$\mathfrak{X}$ at a 
point $x$ of the closed fiber of $\mathfrak{X}$.
Put $Z=\spec(R/(\pi))$.
Suppose that $B$ contains $p^{r}$-th roots of unity.
Then the sequence
\begin{align*}
0\to 
\HO_{Z}^{q}(R_{\et}, \mathfrak{T}_{r}(n))
\to
\displaystyle\bigoplus_{z\in Z^{(0)}}
\HO_{z}^{q}(R_{\et}, \mathfrak{T}_{r}(n))
\to
\displaystyle\bigoplus_{z\in Z^{(1)}}
\HO_{z}^{q+1}(R_{\et}, \mathfrak{T}_{r}(n))
\to
\cdots
\end{align*}
is exact for $q\geq n+2$.
\end{thm}
\begin{proof}\upshape
Let $i: Z\to\spec(R)$ be the inclusion of the closed fiber of $\spec(R)$
and
$j: U\to\spec(R)$ the inclusion of the generic fiber of $\spec(R)$.
Since
\begin{align*}
\HO^{s}_{z}(
R_{\et}, i_{*}Ri^{!}\mathfrak{T}_{r}(n)
)
=
\begin{cases}
\HO^{s}_{z}(R_{\et}, \mathfrak{T}_{r}(n))
& (z\in Z) \\
\\
0
& (z\in U),
\end{cases}
\end{align*}
we have a spectral sequence
\begin{equation}\label{ConiST}
E^{s, q}_{1}(n)
=
\displaystyle\bigoplus_{z\in\spec(R)^{(s)}\cap Z}
\HO^{s+q}_{z}(
R_{\et},
\mathfrak{T}_{r}(n)
)
\Rightarrow   
E^{s+q}(n)
=
\HO_{Z}^{s+q}
(R_{\et}, \mathfrak{T}_{r}(n))
\end{equation}
by \cite[Proposition 3.8]{Sak4}.
Moreover, we have a commutative diagram
\scriptsize
\begin{equation}\label{snZ}
\xymatrix{
&
\HO_{\et}^{n+1}(R, i_{*}i^{*}\mathfrak{T}_{r}(n))
\ar[r]\ar[d]
&
\HO_{Z}^{n+2}(R_{\et}, j_{!}\mu_{p^{r}}^{\otimes n})
\ar[r]\ar[d]
&
\HO_{Z}^{n+2}(R_{\et}, \mathfrak{T}_{r}(n))
\ar[d]\ar[r]
&
0
\\
0\ar[r]
&
\displaystyle\bigoplus_{z\in Z^{(0)}}
\HO_{z}^{n+1}(R_{\et}, i_{*}i^{*}\mathfrak{T}_{r}(n))
\ar[r]
&
\displaystyle\bigoplus_{z\in Z^{(0)}}
\HO_{z}^{n+2}(R_{\et}, j_{!}\mu_{p^{r}}^{\otimes n})
\ar[r]
&
\displaystyle\bigoplus_{z\in Z^{(0)}}
\HO_{z}^{n+2}(R_{\et}, \mathfrak{T}_{r}(n))
}    
\end{equation}
\normalsize
where the sequences are exact 
by \cite[Theorem 1.3]{Sak5} and (\ref{itvan}).
Since we have an isomorphism
\begin{equation*}
\HO^{n+1}_{\et}(
R,
\mathfrak{T}_{r}(n)
)    
\simeq 
\HO^{n+1}_{\et}(
R,
i_{*}i^{*}
\mathfrak{T}_{r}(n)
)
\end{equation*}
by \cite[p.777, The proof of Proposition 2.2.b)]{Ge},
we have an isomorphism
\begin{equation}\label{iso1iTl}
\HO^{n+1}_{\et}(
R,
i_{*}i^{*}\mathfrak{T}_{r}(n)
)    
\simeq
\HO^{1}_{\et}(
Z,
\lambda_{r}^{n}
).
\end{equation}
by \cite[Theorem 1.2]{Sak4} and \cite[Theorem 1.4]{Sak4}.
Moreover, 
we have an isomorphism
\begin{equation*}
\HO^{n+1}_{z}(
R_{\et},
i_{*}i^{*}\mathfrak{T}_{r}(n)
) 
\simeq 
\HO^{1}_{z}(
Z_{\et},
\lambda_{r}^{n}
)
\end{equation*}
for $z\in Z^{(0)}$
by \cite[Theorem 3.5]{Sak5} and \cite[Lemma 3.7]{Sak5}.
So the left map in the diagram (\ref{snZ})
corresponds to
the canonical map
\begin{equation*}
\HO^{1}_{\et}(Z, \lambda^{n}_{r})
\to
\displaystyle\bigoplus_{z\in Z^{(0)}}
\HO^{1}_{z}(Z, \lambda^{n}_{r})
\end{equation*}
and the left map in the diagram (\ref{snZ}) is injective 
by Theorem \ref{Gtcl}.
Moreover, the right map in the diagram (\ref{snZ}) is injective 
by \cite[Theorem 3.10]{Sak4}.
Hence the middle map in the diagram (\ref{snZ}) is injective.
Since $B$ contains $\mu_{p^{r}}$ by the assumption,
we have isomorphisms
\begin{equation}\label{Ztmiso}
\HO_{Z}^{v}(R, \mathfrak{T}_{r}(n))
\simeq
\HO_{Z}^{v}(R, j_{!}\mu_{p^{r}}^{\otimes (v-3)}
)
\xrightarrow{\simeq}
\HO_{Z}^{v}(R, \mathfrak{T}_{r}(v-3))
\end{equation}
for $v\geq n+3$ by Proposition \ref{jtxiso}. 
Moreover, we have isomorphisms
\begin{equation}\label{xtmiso}
\HO_{z}^{v}(R_{\et}, \mathfrak{T}_{r}(n)) 
\simeq
\HO_{z}^{v}(R_{\et}, j_{!}\mu_{p^{r}}^{\otimes(v-s-3)}
)
\xrightarrow{\simeq}
\HO_{z}^{v}(R_{\et}, \mathfrak{T}_{r}(v-s-3))
\end{equation}
for 
$z\in\spec(R)^{(s+1)}\cap Z$ and $v\geq n+s+3$
by Proposition \ref{jtxiso} and the assumption.
So the homomorphism
\begin{equation}\label{EtInf}
E^{v}(n)
\to
E_{\infty}^{1, v-1}(n)
\end{equation}
is injective for $v\geq n+2$. Hence
we have an isomorphism
\begin{equation}\label{2inf0}
E^{s, q}_{\infty}(n)=0    
\end{equation}
for $s>1$ and $s+q\geq n+2$ by the spectral sequence (\ref{ConiST}). 
Then we prove that
\begin{align}\label{E2n}
E^{s, q}_{2}(n)
=
\begin{cases}
E^{q+1}(n)
& (s=1) \\
\\
0
& (s\geq 2)
\end{cases}
\end{align}
for any integers $n\geq 0$ and $q\geq n+1$ by induction on $s$.

In the case where $s=1$, we have
\begin{equation}\label{Supinj}
E_{2}^{1, n+1}(n)=E^{n+2}(n)    
\end{equation}
by \cite[Theorem 3.10]{Sak4}.
Consider a commutative diagram
\scriptsize
\begin{equation*}
\xymatrix
{
&
\displaystyle\bigoplus_{z\in Z^{(0)}}
\HO_{z}^{n+1}(R_{\et}, i_{*}i^{*}\mathfrak{T}_{r}(n))
\ar[r]\ar[d]
&
\displaystyle\bigoplus_{z\in Z^{(0)}}
\HO_{z}^{n+2}(R_{\et}, j_{!}\mu_{p^{r}}^{\otimes n})
\ar[r]\ar[d]
&
\displaystyle\bigoplus_{z\in Z^{(0)}}
\HO_{z}^{n+2}(R_{\et}, \mathfrak{T}_{r}(n))
\ar[r]\ar[d]
&
0
\\
0
\ar[r]
&
\displaystyle\bigoplus_{z\in Z^{(1)}}
\HO_{z}^{n+2}(R_{\et}, i_{*}i^{*}\mathfrak{T}_{r}(n))
\ar[r]
&
\displaystyle\bigoplus_{z\in Z^{(1)}}
\HO_{z}^{n+3}(R_{\et}, j_{!}\mu_{p^{r}}^{\otimes n})
\ar[r]
&
\displaystyle\bigoplus_{z\in Z^{(1)}}
\HO_{z}^{n+3}(R_{\et}, \mathfrak{T}_{r}(n))
}    
\end{equation*}
\normalsize
where the sequences are exact by \cite[Theorem 1.3]{Sak5}. 
Then the kernel of the left map is
$\HO_{\et}^{n+1}(R, i_{*}i^{*}\mathfrak{T}_{r}(n))$ 
by Theorem \ref{1gerl},
the isomorphism (\ref{iso1iTl})
and \cite[Theorem 1.2]{Sak5}. 
Moreover, the kernel of the right map is 
\begin{math}
\HO_{Z}^{n+2}(
R_{\et},
\mathfrak{T}_{r}(n)
)    
\end{math}
by the isomorphism (\ref{Supinj}).
By the upper sequence in the diagram (\ref{snZ}), the sequence
\begin{equation*}
0\to  
\HO^{n+2}_{Z}
(R_{\et}, j_{!}\mu_{p^{r}}^{\otimes n})
\to
\displaystyle\bigoplus_{z\in Z^{(0)}}
\HO^{n+2}_{z}
(R_{\et}, j_{!}\mu_{p^{r}}^{\otimes n})
\to
\displaystyle\bigoplus_{z\in Z^{(1)}}
\HO^{n+3}_{z}
(R_{\et}, j_{!}\mu_{p^{r}}^{\otimes n})
\end{equation*}
is exact and so we have
\begin{equation}\label{E1eqE}
E^{1, q-1}_{2}(n)
=
E^{q}(n)
\end{equation}
for $q\geq n+3$ by (\ref{Ztmiso}) and (\ref{xtmiso}).
Hence we have the equation (\ref{E2n}) for 
any integer $n\geq 0$ and $s=1$. Moreover, we have
\begin{equation*}
E^{1, q-1}_{\infty}(n)
=
E^{1, q-1}_{2}(n)
\end{equation*}
for $q\geq n+3$ by the injectivity of the homomorphism (\ref{EtInf})
and the isomorphism (\ref{E1eqE}).
So
we have
\begin{equation}\label{1d0}
\operatorname{Im}
\left(
d^{1, u+n-1}_{u-1}:
E^{1, u+n-1}_{u-1}(n)
\to
E^{u, n+1}_{u-1}(n)
\right)
=0
\end{equation}
for $u\geq 3$.

Suppose that the equation (\ref{E2n}) holds for any integers $n\geq 0$ and $s\leq t$.
By the assumption and (\ref{1d0}), we have
\begin{equation*}
\operatorname{Im}\left(
E^{t+1-r, n+r}_{r}(n)
\to
E^{t+1, n+1}_{r}(n)
\right)
=0
\end{equation*}
for $r\geq 2$. Moreover, we have
\begin{equation*}
E^{t+r+1, n+2-r}_{r}(n)=0    
\end{equation*}
for $r\geq 2$ by \cite[p.540, Theorem 4.4.7]{SaP} 
and \cite[Theorem 2.7]{Sak4}. 
So we have
\begin{equation}\label{2n10}
E_{2}^{t+1, n+1}(n)
=
E_{\infty}^{t+1, n+1}(n)
=0
\end{equation}
by (\ref{2inf0}). 
By \cite[Theorem 1.3]{Sak5} and Proposition \ref{jtxiso}, 
the sequence
\begin{align*}
0\to 
\displaystyle\bigoplus_{z\in\spec(R)^{(u)}\cap Z}
\HO_{z}^{u}(Z_{\et}, \lambda_{r}^{n+1})
\to
E_{1}^{u, n+2}(n)
\to
E_{1}^{u, (n+1)+1}(n+1)
\to
0
\end{align*}
is exact for any integer $u\geq 1$. So we have
\begin{equation*}
E^{t+1, n+2}_{2}(n)=0    
\end{equation*}
by Theorem \ref{Gtcl}  and (\ref{2n10}). 
Hence the equation (\ref{E2n}) holds 
for any integer $n\geq 0$ and $s=t+1$.
This completes the proof.
\end{proof}

\section{A relation between the Gersten-type conjecture and motivic cohomology groups}\label{TrelGtM}

Let $\mathbb{Z}(n)$ be Bloch's cycle complex for Zariski topology and
$\mathbb{Z}/m(n):=\mathbb{Z}(n)\otimes \mathbb{Z}/m\mathbb{Z}$
for any positive integer $m$. 
As
an application of the Beilinson-Lichtenbaum conjecture (cf. \cite[p.774, Theorem 1.2.2]{Ge} and \cite{V2}),
we obtain the following:

\begin{prop}\upshape\label{RGCM}
Let $B$ be a discrete valuation ring 
of mixed characteristic $(0, p)$
and $R$ a local ring 
(resp. a henselian local ring)
of a regular scheme of finite type over $\spec(B)$.
Let $U$ be the generic fiber of $\spec(R)$ and $n$ a non-negative integer.
Suppose that $B$ contains $p$-th roots of unity. 

Then the sequence
\begin{align}\label{Uex}
0\to& 
\HO^{n+1}_{\et}(U, \mu_{p}^{\otimes n})
\to
\displaystyle\bigoplus_{x\in U^{(0)}}
\HO^{n+1}_{\et}(
\kappa(x),
\mu_{p}^{\otimes n}
)
\to
\displaystyle\bigoplus_{x\in U^{(1)}}
\HO^{n}_{\et}(
\kappa(x),
\mu_{p}^{\otimes (n-1)}
) \nonumber
\\
\to&\cdots
\end{align}
is exact 
if and only if
we have isomorphisms
\begin{equation*}
\HO^{n+1}_{\Zar}(U, \mathbb{Z}/p(n))
\simeq 
0
\end{equation*}
and
\begin{equation*}
\HO^{n+1+s}_{\Zar}(
U,
\mathbb{Z}/p(n)
)    
\simeq
\HO^{n+1+s}_{\Zar}(
U,
\mathbb{Z}/p(n+1)
)
\end{equation*}
for any integer $s\geq 1$. Here $\mu_{p}$ is the sheaf of $p$-th roots of unity.
\end{prop}
\begin{proof}\upshape
The sequence (\ref{Uex}) is exact
if and only if
\begin{align*}
\HO^{s}_{\Zar}(
U,
R^{n+1}\epsilon_{*}
\mu_{p}^{\otimes n}
)
=
\begin{cases}
\HO^{n+1}_{\et}(
U, 
\mu_{p}^{\otimes n}
)
& (\textrm{if}~~s=0) \\
\\
0
& (\textrm{if}~~s\geq 1)
\end{cases}
\end{align*}
(cf. The proof of \cite[Lemma 5.3 (i)]{Sak5}). 
Here $\epsilon: U_{\et}\to U_{\Zar}$ is the canonical map
of sites and $\epsilon_{*}$ is the forgetful functor.
Moreover, 
we have a distinguished triangle
\begin{align*}
\cdots
\to  
\tau_{\leq n}R\epsilon_{*}\mu_{p}^{\otimes n}
\to
\tau_{\leq (n+1)}R\epsilon_{*}\mu_{p}^{\otimes n}
\to
R^{n+1}\epsilon_{*}\mu_{p}^{\otimes n}[-(n+1)]
\to\cdots
\end{align*}
and a quasi-isomorphism
\begin{equation*}
\mathbb{Z}/p(N)_{\Zar}^{U}
\simeq
\tau_{\leq N}R\epsilon_{*}
\mu_{p}^{\otimes N}
\end{equation*}
for $N=n, n+1$ by \cite[p.774, Theorem 1.2]{Ge} and \cite{V2}. 
Since $B$ contains $p$-th roots of unity, we have an exact sequence
\begin{align*}
\cdots\to
&\HO^{n+1+s}_{\Zar}(
U, \mathbb{Z}/p(n)
)
\to
\HO^{n+1+s}_{\Zar}(
U, \mathbb{Z}/p(n+1)
)
\to
\HO^{s}_{\Zar}(
U, 
R^{n+1}\epsilon_{*}\mu_{p}^{\otimes n}
)
\\
\to
&\HO^{n+2+s}_{\Zar}(
U,
\mathbb{Z}/p(n)
)
\to\cdots
\end{align*}
for any integer $s\geq 0$ by the above.
So the statement follows.
\end{proof}
\begin{rem}\upshape\label{revanimot}
Let $R$ be a local ring of a semistable family over the spectrum of a discrete valuation ring
and $U$ the generic fiber of $\spec(R)$.
Let $m$ be a positive integer. Then we have an isomorphism
\begin{equation*}
\HO^{q}_{\Zar}(
R,
\mathbb{Z}/m(n)
)
\xrightarrow{\simeq}
\HO^{q}_{\Zar}(
U,
\mathbb{Z}/m(n)
)
\end{equation*}
for $q\geq n+2$ by the localization theorem (\cite[p.779, Theorem 3.2]{Ge}) and \cite[Proposition 2.1]{Sak4}.
So we have 
\begin{math}
\HO^{q}_{\Zar}(
R, \mathbb{Z}/m(n)
)    
=0
\end{math}
for $q\geq n+2$
if and only if
\begin{math}
\HO^{q}_{\Zar}(
U, \mathbb{Z}/m(n)
)    
=0
\end{math}
for $q\geq n+2$.
\end{rem}

The following is an affirmative answer to \cite[p.524, Conjecture 1.4.1]{SaP} in a special case.

\begin{cor}\upshape
Let $B$ be a discrete valuation ring of mixed characteristic $(0, p)$, 
$k$ the residue field of $B$ and $[k:k^{p}]\leq p^{t}$. 
Let $\mathfrak{X}$ be a semistable family over
$\spec(B)$ and $n=\operatorname{dim}(\mathfrak{X})+t$. Suppose that 
$B$ contains $p$-th roots of unity. Then we have a quasi-isomorphism
\begin{equation*}
\mathbb{Z}/p^{r}(n)_{\et}^{\mathfrak{X}}
\simeq 
\mathfrak{T}_{r}(n)
\end{equation*}
for any integer $r\geq 1$.
\end{cor}
\begin{proof}\upshape
Let $R$ be the strict Henselization $\mathcal{O}_{\mathfrak{X}, x}^{sh}$
of $\mathfrak{X}$ at a point $x$ of the closed fiber of $\mathfrak{X}$.
By \cite[p.209, Remark 7.2]{SaR}, we have a quasi-isomorphism
\begin{equation*}
\tau_{\leq n}\mathbb{Z}/p^{r}(n)_{\et}^{\mathfrak{X}}  
\simeq
\mathfrak{T}_{r}(n)
\end{equation*}
for any positive integer $r$. Moreover, the sequence
\begin{align*}
\cdots\to  
\HO^{u}_{\Zar}(R, \mathbb{Z}/p^{r}(n))
\to
\HO^{u}_{\Zar}(R, \mathbb{Z}/p^{r+r^{\prime}}(n))
\to
\HO^{u}_{\Zar}(R, \mathbb{Z}/p^{r^{\prime}}(n))
\to
\cdots
\end{align*}
is exact for positive integers $r$ and $r^{\prime}$. So it suffices to show that we have an isomorphism
\begin{equation*}
\HO^{u}_{\Zar}(
R, \mathbb{Z}/p(n)
)    
\simeq 0
\end{equation*}
for $u\geq n+1$. Let $U$ be the generic fiber of $\spec(R)$. Then it suffices to show that
we have an isomorphism
\begin{equation*}
\HO^{u}_{\Zar}(
U, \mathbb{Z}/p(n)
)    
\simeq 
0
\end{equation*}
for $u\geq n+2$ by 
\cite[Proposition 1.5]{Sak4}, \cite[Proposition 2.1]{Sak4} 
and the localization theorem (\cite[p.779, Theorem 3.2]{Ge}).

Let $x^{\prime}$ be an inverse image of a point $x$ under the morphism
\begin{math}
\mathfrak{X}\times B^{sh}
\to
\mathfrak{X}.
\end{math}
Then we have an isomorphism
\begin{equation*}
\mathcal{O}^{sh}_{\mathfrak{X}, x}
\simeq 
\mathcal{O}^{sh}_{\mathfrak{X}\times_{B}B^{sh}, x^{\prime}}
\end{equation*}
and an equation
\begin{equation*}
[k_{s}: k^{p}_{s}]=[k: k^{p}]\leq p^{t}    
\end{equation*}
where $k_{s}$ is a separable closed field of $k$. So we are able to assume that 
$k$ is a separable closed field.

We consider the cohomological $p$-dimension
$\operatorname{cd}_{p}(k(\mathcal{O}_{\mathfrak{X}, x}^{sh}))$
of $k(\mathcal{O}_{\mathfrak{X}, x}^{sh})$. We have an equation
\begin{equation*}
\operatorname{cd}_{p}(k(B))=1+t    
\end{equation*}
by \cite[p.48, Th\'{e}or\`{e}me 1.2]{G-O}.
Since $k(\mathcal{O}_{\mathfrak{X}, x}^{sh})$
is an algebraic extension of $k(\mathcal{O}_{\mathfrak{X}, x})$,
we have
\begin{equation*}
\operatorname{tr.deg}_{k(B)}k(\mathcal{O}^{sh}_{\mathfrak{X}, x})
=
\operatorname{dim}(\mathfrak{X})-1
\end{equation*}
by \cite[pp.118--119, Theorem 15.5]{Ma} and \cite[p.119, Theorem 15.6]{Ma}.
Here 
$\operatorname{tr.deg}_{k(B)}k(\mathcal{O}^{sh}_{\mathfrak{X}, x})$
is the transcendence degree of $k(\mathcal{O}^{sh}_{\mathfrak{X}, x})$
over $k(B)$. So we have
\begin{equation*}
\operatorname{cd}_{p}(k(\mathcal{O}^{sh}_{\mathfrak{X}, x}))
\leq
\operatorname{dim}(\mathfrak{X})+t=n
\end{equation*}
by \cite[p.57, Corollaire 2.3.2]{G-O}.

Since the homomorphism
\begin{equation*}
\HO^{u}_{\et}(U, \mu_{p}^{\otimes n})
\to
\HO^{u}_{\et}(k(R), \mu_{p}^{\otimes n})
\end{equation*}
is injective for $u\geq n+1$ by \cite[Proposition 4.2]{Sak4} and \cite[Theorem 3.10]{Sak4}, 
we have an isomorphism
\begin{equation*}
\HO^{u}_{\et}(
U, \mu_{p}^{\otimes n}
)    
\simeq 
0
\end{equation*}
for $u\geq n+1$. Here $\mu_{p}$ is the sheaf of $p$-th roots of unity.
Moreover, we have
\begin{equation*}
\operatorname{cd}_{p}(\kappa(x))
\leq 
n-\operatorname{dim}(
\mathcal{O}_{\mathfrak{X}, x}
)
\end{equation*}
for $x\in U$
by \cite[p.57, Corollaire 2.3.2]{G-O} and \cite[pp.118--119, Theorem 15.5]{Ma}. 
Hence we have isomorphisms
\begin{align*}
\HO^{n+1+s}_{\Zar}(
U,
\mathbb{Z}/p(n)
)
\simeq    
\HO^{n+1+s}_{\Zar}(
U,
\mathbb{Z}/p(n+1)
)
\simeq 
\cdots
\simeq
\HO^{n+1+s}_{\Zar}(
U,
\mathbb{Z}/p(n+s)
)
\simeq
0
\end{align*}
for any integer $s\geq 1$ by Proposition \ref{RGCM}. This completes the proof.
\end{proof}
\begin{cor}\upshape\label{2GerT}
Let $B$ be a discrete valuation ring of mixed characteristic $(0, p)$ 
and $R$ the henselization of the local ring $\mathcal{O}_{\mathfrak{X}, x}$ of a semistable family $\mathfrak{X}$ over $\spec(B)$
at a point $x\in \mathfrak{X}$. 
Suppose that $\operatorname{dim}(R)=2$ and $B$ contains $p$-th roots of unity.
Then the sequence
\begin{align*}
0
\to  
&\HO^{s}_{\et}(
R, \mathfrak{T}_{1}(n)
)
\to
\bigoplus_{x\in \spec(R)^{(0)}}
\HO^{s}_{x}(
R_{\et},
\mathfrak{T}_{1}(n)
)
\to
\\
&\bigoplus_{x\in \spec(R)^{(1)}}
\HO^{s+1}_{x}(
R_{\et},
\mathfrak{T}_{1}(n)
)
\to
\bigoplus_{x\in \spec(R)^{(2)}}
\HO^{s+2}_{x}(
R_{\et},
\mathfrak{T}_{1}(n)
)
\to 0
\end{align*}
is exact for any integer $s\geq 0$.
\end{cor}
\begin{proof}\upshape
Let $U$ be the generic fiber of $\spec(R)$ and 
$N$ an integer larger than or equal to $n$. 
Then we have an isomorphism
\begin{equation*}
\HO^{N+1}_{\Zar}(U, \mathbb{Z}/p(N))
\simeq 0
\end{equation*}
by \cite[Proposition 4.5]{Sak4}, \cite[Proposition 2.1]{Sak4}
and
the localization theorem (\cite[p.779, Theorem 3.2]{Ge}).
Moreover, we have  an isomorphism
\begin{equation*}
\HO^{N+t}_{\Zar}(
U,
\mathbb{Z}/p(N)
) 
\simeq 
0
\end{equation*}
for any integer $t\geq 2$
by the proof of \cite[Corollary 4.6]{Sak4}.

In the case where $s\leq n$, the statement follows from \cite[p.540, Theorem 4.4.7]{SaP},
\cite[Theorem 1.1]{Sak4} and Proposition \ref{RGCM}.
In the case where $s\geq n+1$, the statement follows from Theorem \ref{RelGer} and Proposition \ref{RGCM}.
This completes the proof.
\end{proof}
\begin{cor}\upshape
Let $B$ be a discrete valuation ring of mixed characteristic $(0, p)$ 
and $R$ the henselization of the local ring $\mathcal{O}_{\mathfrak{X}, x}$
of a semistable family $\mathfrak{X}$ over $\spec(B)$ at a point $x\in \mathfrak{X}$. 
Suppose that $B$ contains $p$-th roots of unity. Then the sequence
\begin{align*}
0
&\to    
\HO^{1}_{\et}(
R, 
\mathbb{Z}/p\mathbb{Z}
)
\to
\bigoplus_{x\in \spec(R)^{(0)}}
\HO^{1}_{x}(
R_{\et},
\mathbb{Z}/p\mathbb{Z}
)
\to
\bigoplus_{x\in \spec(R)^{(1)}}
\HO^{2}_{x}(
R_{\et},
\mathbb{Z}/p\mathbb{Z}
)
\\
&\to\cdots
\end{align*}
is exact.
\end{cor}
\begin{proof}\upshape
By the definition, we have a quasi-isomorphism
\begin{equation*}
\mathbb{Z}(0)_{\Zar}^{\spec(R)}
\simeq 
\mathbb{Z}.
\end{equation*}
Moreover, we have a quasi-isomorphism
\begin{equation*}
\mathbb{Z}(1)^{\spec(R)}_{\Zar} 
\simeq 
\mathbb{G}_{m}[-1]
\end{equation*}
by \cite[Lemma 11.2]{L}. So we have an isomorphism
\begin{equation*}
\HO_{\Zar}^{N+t}(
U, \mathbb{Z}/p(N)
)    
\simeq 0
\end{equation*}
for $N=0, 1$ and an integer $t\geq 1$ by \cite[Proposition 2.1]{Sak4} 
and the localization theorem (\cite[p.779, Theorem 3.2]{Ge}). 
Here $U$ is the generic fiber of $\spec(R)$.
Hence the statement follows from Theorem \ref{RelGer} and Proposition \ref{RGCM}.
\end{proof}
\begin{rem}\upshape
If $\mathfrak{X}$ is a smooth scheme over the spectrum of a discrete valuation ring, 
then we have
$\mathcal{H}^{s}(\mathbb{Z}(n)^{\mathfrak{X}})=0$ for $s>n$ 
by \cite[p.786, Corollary 4.4]{Ge}. Hence
\cite[Theorem 1.4]{Sak5} follows from Theorem \ref{RelGer} and Proposition \ref{RGCM}.
\end{rem}
\section{Computations of motivic cohomology groups on global cases}\label{Tcommt}

\subsection{Equi-characteristic cases}

We compute motivic cohomology groups in equi-characteristic cases and
prepare to compute motivic cohomology groups in mixed characteristic cases.

\begin{prop}\upshape\label{VaniNGL}
Let $k$ be a field and 
\begin{equation*}
D=k[T_{0}, \cdots, T_{N}]/(T_{0}\cdots T_{a})    
\end{equation*}
for $0\leq a\leq N$. Then we have
\begin{equation*}
\HO^{q}_{\Zar}(
D,
\mathbb{Z}(n)
)    
=0
\end{equation*}
for any integer $n\geq 0$ and $q\geq n+1$.
\end{prop}
\begin{proof}\upshape
By \cite[p.781, Corollary 3.5]{Ge}, it suffices to show the statement in the case where $a=N$. 
Then
$\spec(D)$ is isomorphic to
\begin{equation*}
\spec(
k[T_{0}, \cdots, T_{a}]
/(T_{0}\cdots T_{a-1})
)
\cup
\spec(
k[T_{0}, \cdots, T_{a}]
/(T_{a})
).
\end{equation*}
Let
\begin{equation*}
Z=\spec(
k[T_{0}, \cdots, T_{a}]/
(T_{0}\cdots T_{a-1}, T_{a})
)    
\simeq
\spec(
k[T_{0}, \cdots, T_{a-1}]/(T_{0}\cdots T_{a-1})
)
\end{equation*}
and 
\begin{equation*}
U=
\spec(D)
\setminus
Z.
\end{equation*}
Then we have
\begin{align*}
U\simeq
\spec\left(
(
k[T_{0}, \cdots, T_{a}]/(T_{0}\cdots T_{a-1})
)_{T_{a}}
\right)
\oplus
\spec\left(
(
k[T_{0}, \cdots, T_{a-1}]
)_{T_{0}\cdots T_{a-1}}
\right).
\end{align*}
Moreover,  the homomorphism
\begin{equation*}
\HO^{n}_{\Zar}\left(
(k[T_{0}, \cdots, T_{a-1}])_{T_{0}\cdots T_{a-1}},
\mathbb{Z}(n)
\right)
\to
\HO^{n-1}_{\et}(
Z, \mathbb{Z}(n-1)
)
\end{equation*}
is surjective by 
the localization theorem (\cite[p.779, Theorem 3.2]{Ge})
and
\cite[p.781, Corollary 3.5]{Ge}. 
Since we have a commutative diagram
\begin{equation*}
\xymatrix{
\HO^{n}_{\Zar}\left(
(k[T_{0}, \cdots, T_{a-1}])_{T_{0}\cdots T_{a-1}},
\mathbb{Z}(n)
\right)
\ar[r]\ar[d]
&
\HO^{n-1}_{\Zar}(
Z, 
\mathbb{Z}(n-1)
)\ar@{=}[d]
\\
\HO^{n}_{\Zar}(
U,
\mathbb{Z}(n)
)\ar[r]
&
\HO^{n-1}_{\Zar}(
Z, 
\mathbb{Z}(n-1)
),
}    
\end{equation*}
the homomorphism
\begin{equation*}
\HO^{n}_{\Zar}(U,
\mathbb{Z}(n))
\to
\HO^{n-1}_{\Zar}(
Z, \mathbb{Z}(n-1)
)
\end{equation*}
is surjective. So there exists an exact sequence
\begin{align*}
0\to&
\HO^{n+1}_{\Zar}(
D, \mathbb{Z}(n)
)
\to 
\HO^{n+1}_{\Zar}(
U, \mathbb{Z}(n)
)
\to 
\HO^{n}_{\Zar}(
Z, \mathbb{Z}(n-1)
)\\
\to&
\HO^{n+2}_{\Zar}(
D, \mathbb{Z}(n)
)
\to 
\HO^{n+2}_{\Zar}(
U, \mathbb{Z}(n)
)
\to
\cdots
\end{align*}
by the localization theorem (\cite[p.779, Theorem 3.2]{Ge}). 
Moreover, we have exact sequences
\begin{align*}
&\HO^{q}_{\Zar}\left(
k[T_{0}, \cdots, T_{a}]/
(T_{0}\cdots T_{a-1}),
\mathbb{Z}(n)
\right)    
\to
\HO^{q}_{\Zar}\left(
(k[T_{0}, \cdots, T_{a}]/
(T_{0}\cdots T_{a-1}))_{T_{a}},
\mathbb{Z}(n)
\right)   \\  
&\to
\HO^{q-1}_{\Zar}\left(
k[T_{0}, \cdots, T_{a-1}]/
(T_{0}\cdots T_{a-1}),
\mathbb{Z}(n-1)
\right)    
\end{align*}
and
\begin{align*}
&\HO^{q}_{\Zar}\left(
k[T_{0}, \cdots, T_{a-1}],
\mathbb{Z}(n)
\right)
\to
\HO^{q}_{\Zar}\left(
(k[T_{0}, \cdots, T_{a-1}])_{T_0\cdots T_{a-1}},
\mathbb{Z}(n)
\right)  \\
&\to
\HO^{q-1}_{\Zar}\left(
k[T_{0}, \cdots, T_{a-1}]/(T_{0}\cdots T_{a-1}),
\mathbb{Z}(n-1)
\right)
\end{align*}
for any integer $q$. Hence, the statement follows by induction on $a$.
\end{proof}

We use the following results later.

\begin{prop}\upshape\label{pBL}
Let $Y$ be a scheme which is essentially of finite type over the spectrum of a field $k$. Then
we have a quasi-isomorphism
\begin{equation}\label{posBL}
\tau_{\leq n+1}\left(
\mathbb{Z}(n)^{Y}_{\Zar}
\right)
\xrightarrow{\simeq}
\tau_{\leq n+1}\left(
R\epsilon_{*}\mathbb{Z}(n)_{\et}^{Y}
\right)
\end{equation}
for any integer $n$. 
Here $\epsilon: Y_{\et}\to Y_{\Zar}$ is the canonical map of sites.
\end{prop}
\begin{proof}\upshape
Let $A$ be a local ring $\mathcal{O}_{Y, y}$ of $Y$
at a point $y\in Y$. It suffices to prove the quasi-isomorphism (\ref{posBL}) for 
$\spec(A)$. We prove the statement by induction on
$\operatorname{dim}(A)$.

Assume that
\begin{math}
\operatorname{dim}(A)
=0,
\end{math}
that is,
$A$ is a field. Then we have the quasi-isomorphism 
(\ref{posBL}) by \cite[p.774, Theorem 1.2.2]{Ge} and \cite{V2}.

Assume that the quasi-isomorphism (\ref{posBL}) holds for
$\spec(A^{\prime})$
in the case where $A^{\prime}$ is a local ring of 
a scheme of finite type over the spectrum of a field
at a point
and
\begin{math}
\operatorname{dim}(A^{\prime})
<
\operatorname{dim}(A).
\end{math}

Let 
$j: U\to \spec(A)$ be an open immersion of $\spec(A)$
such that $U$ is a regular scheme.
Let $\iota: Z\to \spec(A)$ be a closed subscheme of
codimension $c\geq 1$ with open complement 
$j: U\to \spec(A)$.
Since $A$ is an excellent ring, we are able to
choose such a non-empty open subset $U$ of $\spec(A)$.

Let $A^{\prime\prime}$ be a regular local ring which is essentially finite type
over a field $k$. By Quillen's method (cf. \cite[\S 7, The proof of Theorem 5.11]{Q}),
\begin{align*}
A^{\prime\prime}
=
\lim_{\substack{\to  \\ s}}
A^{\prime\prime}_{s}
\end{align*}
where $A^{\prime\prime}_{s}$ is a local ring of a 
smooth algebra over the prime field and the maps
$A^{\prime\prime}_{s}\to A^{\prime\prime}_{t}$
are flat. So we have a quasi-isomorphism
\begin{equation}\label{UBL}
\mathbb{Z}(n)^{U}_{\Zar}
\xrightarrow{\simeq}
\tau_{\leq n+1}R\epsilon_{*}\mathbb{Z}(n)^{U}_{\et}
\end{equation}
by \cite[p.774, Theorem 1.2.2]{Ge} and \cite{V2}.

In the case where
$U=\spec(A)$, then the quasi-isomorphism (\ref{posBL}) holds. 
So it suffices to prove the quasi-isomorphism (\ref{posBL})
in the case where $U\neq \spec(A)$.

Consider a map of distinguished triangles in the derived category of \'{e}tale
sheaves on $\spec(A)$
\begin{equation}\label{mpdis}
\xymatrix{
\iota_{*}\mathbb{Z}(n-c)[-2c]_{\et} 
\ar[r]\ar[d]
&
\mathbb{Z}(n)_{\et}
\ar[r]\ar@{=}[d]
&
\epsilon^{*}Rj_{*}\mathbb{Z}(n)_{\Zar}
\ar[d]
\\
\iota_{*}R\iota^{!}\mathbb{Z}(n)_{\et}
\ar[r]
&
\mathbb{Z}(n)_{\et}
\ar[r]
&
Rj_{*}\mathbb{Z}(n)_{\et}.
}
\end{equation}
Then the right map
agrees with the composite
\begin{align*}
&\epsilon^{*}Rj_{*}\mathbb{Z}(n)_{\Zar}
\xrightarrow[\simeq]{(*)}
\epsilon^{*}Rj_{*}\left(
\tau_{\leq n+1}R\epsilon_{*}\mathbb{Z}(n)_{\et}
\right)  \\
\to
& \epsilon^{*}Rj_{*}\left(
R\epsilon_{*}\mathbb{Z}(n)_{\et}
\right)
\xleftarrow[\simeq]{(**)}
Rj_{*}\mathbb{Z}(n)_{\et}
\end{align*}
where the maps $(*)$ and $(**)$ are quasi-isomorphisms 
by the quasi-isomorphism (\ref{UBL}) and \cite[p.776, Proposition 2.2. a)]{Ge}. 
Since the homomorphism
\begin{equation*}
\epsilon^{*}R^{q}j_{*}\left(
\tau_{\leq n+1}R\epsilon_{*}\mathbb{Z}(n)_{\et}
\right)
\to
\epsilon^{*}R^{q}j_{*}\left(
R\epsilon_{*}
\mathbb{Z}(n)_{\et}
\right)
\end{equation*}
is an isomorphism for $q\leq n+1$
and injective for $q=n+2$,
the homomorphism
\begin{equation*}
\epsilon^{*}R^{q}j_{*}\mathbb{Z}(n)_{\Zar}
\to
R^{q}j_{*}\mathbb{Z}(n)_{\et}
\end{equation*}
is an isomorphism for $q\leq n+1$ and
injective for $q=n+2$. So we have a quasi-isomorphism
\begin{equation}\label{Gyset}
\tau_{\leq n+2}\left(
\iota_{*}\mathbb{Z}(n-c)_{\et}[-2c]
\right)
\xrightarrow{\simeq}
\tau_{\leq n+2}\left(
\iota_{*}R\iota^{!}\mathbb{Z}(n)_{\et}
\right)
\end{equation}
by the five lemma. Moreover, we consider
a map of distinguished triangles in the derived category of Zariski
sheaves on $\spec(A)$
\begin{equation*}
\xymatrix{
\iota_{*}\mathbb{Z}(n-c)[-2c]_{\Zar}
\ar[r]\ar[d]
&
\mathbb{Z}(n)_{\Zar}
\ar[r]\ar[d]
&
Rj_{*}\mathbb{Z}(n)_{\Zar}
\ar[d]
\\
R\epsilon_{*}\iota_{*}R\iota^{!}\mathbb{Z}(n)_{\et}
\ar[r]
&
R\epsilon_{*}\mathbb{Z}(n)_{\et}
\ar[r]
&
R\epsilon_{*}Rj_{*}\mathbb{Z}(n)_{\et}.
}    
\end{equation*}
Then the left map is the composition of the map
\begin{equation*}
\iota_{*}\mathbb{Z}(n-c)[-2c]_{\Zar}
\to 
\iota_{*}R\epsilon_{*}\mathbb{Z}(n-c)[-2c]_{\et}
\end{equation*}
and the map
\begin{equation*}
R\epsilon_{*}\iota_{*}\mathbb{Z}(n-c)[-2c]_{\et}
\to
R\epsilon_{*}\iota_{*}R\iota^{!}\mathbb{Z}(n)_{\et}
\end{equation*}
which is obtained by applying $R\epsilon_{*}$ to the left map in 
the commutative diagram (\ref{mpdis}).
So we have a quasi-isomorphism
\begin{equation*}
\tau_{\leq n+2}\left(
\iota_{*}\mathbb{Z}(n-c)[-2c]_{\Zar}
\right)
\xrightarrow{\simeq}
\tau_{\leq n+2}\left(
R\epsilon_{*}\iota_{*}R\iota^{!}\mathbb{Z}(n)_{\et}
\right)
\end{equation*}
by the assumption and the quasi-isomorphism (\ref{Gyset}). 
Hence, the quasi-isomorphism (\ref{posBL}) holds for $\spec(R)$ by the five lemma. 
This completes the proof.
\end{proof}
\begin{prop}\upshape\label{constm}
Let $A$ be a local ring of a normal crossing variety over the spectrum of  
a field $k$. Let $n\geq 0$, $l>0$ be integers with $(l, \operatorname{char}(k))=1$.
Then we have isomorphisms
\begin{equation}\label{Vanineg}
\mathcal{H}^{q}(\mathbb{Z}/l(n)_{\Zar}^{\spec(A)})=0    
\end{equation}
for $q<0$ and
\begin{align}\label{Zerods}
\HO^{0}_{\Zar}(
A, 
\mathbb{Z}/l(n)
)  
\xrightarrow{\simeq}
\bigoplus_{x\in\spec(A)^{(0)}}
\HO^{0}_{\Zar}(
\kappa(x),
\mathbb{Z}/l(n)
)
\simeq
\underbrace{
\mu_{l}^{\otimes n}
\oplus\cdots
\oplus
\mu_{l}^{\otimes n}
}_{\#\left((\spec(A))^{(0)}\right) ~ \textrm{times}}.
\end{align}
Moreover, we have an isomorphism
\begin{equation}\label{Vanipos}
\mathcal{H}^{q}(
\mathbb{Z}/l(n)_{\et}^{\spec(A)}
)    
=0
\end{equation}
for $q\geq\#\left((\spec(A))^{(0)}\right)$.
\end{prop}
\begin{proof}\upshape
First we prove the isomorphism (\ref{Vanineg}).
By \cite[p.779, Theorem 3.2.b)]{Ge}, we have isomorphisms
\begin{align*}
\HO^{q}_{\Zar}(
A_{x},
\mathbb{Z}/l(n)
) 
\simeq 
\Gamma(
\spec(A_{x}),
\mathcal{H}^{q}(
\mathbb{Z}/l(n)_{\Zar}^{\spec(A)}
)
)
\simeq
\mathcal{H}^{q}(
\mathbb{Z}/l(n)^{\spec(A)}_{\Zar}
)_{x}
\end{align*}
for $x\in \spec(A)$ and any integer $q$ where $A_{x}$
is the local ring of $A$ at a point
$x\in\spec(A)$.

Let $Z_{x}$ be an irreducible component of $\spec(A_{x})$ and 
$U_{x}=\spec(A_{x})\setminus Z_{x}$. Then there exists an exact
sequence
\begin{align*}
\HO^{q}_{\Zar}(
Z_{x},
\mathbb{Z}/l(n)
)    
\to
\HO^{q}_{\Zar}(
A_{x},
\mathbb{Z}/l(n)
)
\to
\HO^{q}_{\Zar}(
U_{x},
\mathbb{Z}/l(n)
)
\end{align*}
by the localization theorem (\cite[p.779, Theorem 3.2]{Ge}). Since we have isomorphisms
\begin{equation*}
\HO^{q}_{\Zar}(
Z_{x},
\mathbb{Z}/l(n)
)
\simeq
\HO^{q}_{\et}(
Z_{x},
\mu_{l}^{\otimes n}
)
=0
\end{equation*}
for $q<0$ by \cite[p.774, Theorem 1.2]{Ge} and \cite{V2},
it suffices to prove that
\begin{equation*}
\HO^{q}_{\Zar}(
U_{x}, 
\mathbb{Z}/l(n)
)    
=0
\end{equation*}
for $q<0$.

Let $Z^{\prime}_{x}$ be the union of the irreducible components of $\spec(A_{x})$
excluding $Z_{x}$. Then $Z^{\prime}_{x}$ and 
$Z^{\prime}_{x}\cap Z_{x}$ are the spectrum of a local ring of a normal crossing
variety over $\spec(k)$ and there exists an exact sequence
\begin{align*}
\HO^{q}_{\Zar}(
Z_{x}^{\prime},
\mathbb{Z}/l(n)
)    
\to
\HO^{q}_{\Zar}(
U_{x},
\mathbb{Z}/l(n)
)
\to
\HO^{q-1}_{\Zar}(
Z_{x}\cap Z_{x}^{\prime},
\mathbb{Z}/l(n-1)
)
\end{align*}
by the localization theorem (\cite[p.779, Theorem 3.2]{Ge}).
Since we have
\begin{align*}
\#\left(
(Z^{\prime}_{x})^{(0)}
\right)    
=
\#\left(
(Z^{\prime}_{x}\cap Z_{x})^{(0)}
\right)
=
\#\left(
(\spec(A_{x}))^{(0)}
\right)-1,
\end{align*}
the isomorphism (\ref{Vanineg}) follows by induction on
$\#\left((\spec(A_{x}))^{(0)}\right)$.

Next we prove the isomorphism (\ref{Zerods}). We also prove the isomorphism (\ref{Zerods})
by induction on 
$\#\left((\spec(A))^{(0)}\right)$.

In the case where 
$\#\left((\spec(A))^{(0)}\right)=1$,
the isomorphism (\ref{Zerods}) follows from \cite[p.774, Theorem 1.2.4]{Ge}
and \cite{V2}.

Assume that the isomorphism (\ref{Zerods}) holds in the case where
$\#\left((\spec(A))^{(0)}\right)\leq s$. Suppose that
$\#\left((\spec(A))^{(0)}\right)=s+1$. Let $\bar{\{y\}}$ be an irreducible component of
$\spec(A)$ and $U_{0}=\spec(A)\setminus\bar{\{y\}}$. Then there exists an exact sequence
\begin{align*}
0\to   
\HO^{0}_{\Zar}(
\bar{\{y\}},
\mathbb{Z}/l(n)
)
\to
\HO^{0}_{\Zar}(
\spec(A),
\mathbb{Z}/l(n)
)
\to
\HO^{0}_{\Zar}(
U_{0},
\mathbb{Z}/l(n)
)
\end{align*}
by the localization theorem (\cite[p.779, Theorem 3.2]{Ge})
and the isomorphism (\ref{Vanineg}). 
Let $Z_{1}$ be the union of the irreducible components
of $\spec(A)$ excluding $\bar{\{y\}}$. Since $Z_{1}\cap\bar{\{y\}}$ is the spectrum of a 
local ring of a normal crossing variety over $\spec(k)$, we have
\begin{equation*}
\HO^{q}_{\Zar}(
Z_{1}\cap\bar{\{y\}},
\mathbb{Z}/l(n-1)
)
=
0
\end{equation*}
for $q<0$ by the isomorphism (\ref{Vanineg}). Moreover, we have
\begin{equation*}
U_{0}
=
Z_{1}\setminus
\left(Z_{1}\cap\bar{\{y\}}\right)
\end{equation*}
and $\#((Z_{1})^{(0)})=s$.
So we have isomorphisms
\begin{equation*}
\HO^{0}_{\Zar}(
Z_{1}, \mathbb{Z}/l(n)
)    
\xrightarrow{\simeq}
\HO^{0}_{\Zar}(
U_{0},
\mathbb{Z}/l(n)
)
\xrightarrow{\simeq}
\bigoplus_{z\in (Z_{1})^{(0)}}
\HO^{0}_{\Zar}(
\kappa(z),
\mathbb{Z}/l(n)
)
\end{equation*}
by the localization theorem (\cite[p.779, Theorem 3.2]{Ge}) and the assumption. 
Moreover,
the composition
\begin{align*}
\HO^{0}_{\Zar}(
\bar{\{z\}},
\mathbb{Z}/l(n)
)    
\to
\HO^{0}_{\Zar}(
A,
\mathbb{Z}/l(n)
)
\to
\HO^{0}_{\Zar}(
U_{0},
\mathbb{Z}/l(n)
)
\to
\HO^{0}_{\Zar}(
\kappa(z),
\mathbb{Z}/l(n)
)
\end{align*}
is an isomorphism for $z\in (Z_{1})^{(0)}$.
So we have the isomorphism (\ref{Zerods}) in the case where
$\#((\spec(A))^{(0)})=s+1$. Hence, the isomorphism (\ref{Zerods}) holds.

Finally, we prove the isomorphism (\ref{Vanipos}). By \cite[Proposition 2.1]{Sak4},
it suffices to prove the isomorphism (\ref{Vanipos}) in the case where 
$\#((\spec(A))^{(0)})<n+1$. Let $A_{\bar{x}}$ be the strict henselian of $A$
at a point $x\in \spec(A)$. If $l^{\prime}$ is a prime number with $(l^{\prime}, \operatorname{char}(k))=1$,
we have $\operatorname{cd}_{l^{\prime}}(A_{\bar{x}})=0$. So we have isomorphisms
\begin{align*}
\HO^{q}_{\et}(
A_{\bar{x}},
\mathbb{Z}/l(n)
)  
\simeq
\Gamma\left(
A_{\bar{x}}, 
\mathcal{H}^{q}(
\mathbb{Z}/l(n)
)_{\et}
\right)
\simeq 
\mathcal{H}^{q}\left(
\mathbb{Z}/l(n)_{\et}
\right)_{\bar{x}}
\end{align*}
for $q>0$. By \cite[p.774, Theorem 1.2.2]{Ge} and \cite{V2}, 
it suffices to prove that
\begin{equation*}
\HO^{q}_{\Zar}(
A_{\bar{x}},
\mathbb{Z}/l(n)
)    
=0
\end{equation*}
for 
\begin{math}
q\geq 
\#((\spec(A_{\bar{x}}))^{(0)})
=
\#((\spec(A_{x}))^{(0)}).
\end{math}
By a similar argument as in the proof of the isomorphism (\ref{Vanineg}), 
the isomorphism (\ref{Vanipos}) holds as follows.

Let $Z_{\bar{x}}$ be an irreducible component of $\spec(A_{\bar{x}})$
and $Z_{\bar{x}}^{\prime}$ the union of the irreducible components of
$\spec(A_{\bar{x}})$ excluding $Z_{\bar{x}}$. Then we have
\begin{align*}
U_{\bar{x}}
:=
\spec(A_{\bar{x}})\setminus
Z_{\bar{x}}
=
Z^{\prime}_{\bar{x}}
\setminus
Z^{\prime}_{\bar{x}}
\cap
Z_{\bar{x}}
\end{align*}
and isomorphisms
\begin{equation*}
\HO^{q}_{\Zar}(
Z_{\bar{x}},
\mathbb{Z}/l(n)
)    
\simeq
\HO^{q}_{\et}(
Z_{\bar{x}},
\mathbb{Z}/l(n)
)    
\simeq
\HO^{q}_{\et}(
Z_{\bar{x}},
\mu_{l}^{\otimes n}
) 
\simeq
0
\end{equation*}
for $q>0$ by \cite[p.774, Theorem 1.2.2 and Theorem 1.2.4]{Ge}, \cite{V2}
and \cite[p.786, Corollary 4.4]{Ge}. 
So we have an isomorphism
\begin{equation*}
\HO^{q}_{\Zar}(
A_{\bar{x}},
\mathbb{Z}/l(n)
)    
\simeq 
\HO^{q}_{\Zar}(
U_{\bar{x}},
\mathbb{Z}/l(n)
) 
\end{equation*}
for $q>0$. Moreover, we have an exact sequence
\begin{align*}
\HO^{q}_{\Zar}(
Z^{\prime}_{\bar{x}},
\mathbb{Z}/l(n)
)    
\to
\HO^{q}_{\Zar}(
U_{\bar{x}}, 
\mathbb{Z}/l(n)
)
\to
\HO^{q-1}_{\Zar}(
Z^{\prime}_{\bar{x}}\cap
Z_{\bar{x}},
\mathbb{Z}/l(n-1)
\end{align*}
for $q>0$ by the localization theorem (\cite[p.779, Theorem 3.2]{Ge}). Since we have
\begin{align*}
\#((Z^{\prime}_{\bar{x}})^{(0)})
=
\#((
Z^{\prime}_{\bar{x}}
\cap
Z_{\bar{x}}
)^{(0)})
=
\#(
(
\spec(A_{\bar{x}})
)^{(0)}
)
-1,
\end{align*}
the isomorphism (\ref{Vanipos}) follows by induction on
$\#((\spec(A_{\bar{x}}))^{(0)})$. This completes the proof.
\end{proof}
\begin{rem}\upshape
Let $A_{\bar{x}}$ be the strict henselization of a normal crossing variety $Y$ over the spectrum of a field $k$
at a point $x\in Y$ and $l$ an integer which is prime to $\operatorname{char}(k)$.

If $\#((\spec(A_{\bar{x}}))^{(0)})<n+1$ and
$s=\#((\spec(A_{\bar{x}}))^{(0)})-1$, then
\begin{equation*}
\mathcal{H}^{s}(
\mathbb{Z}/l(n)_{\et}^{Y}
)_{\bar{x}}
\simeq
\HO^{s}_{\et}(
A_{\bar{x}},
\mathbb{Z}/l(n)
) 
\neq 0.
\end{equation*}
In fact, we have isomorphisms
\begin{equation*}
\HO^{s}_{\et}(
A_{\bar{x}},
\mathbb{Z}/l(n)
)    
\simeq
\HO^{s-1}_{\et}(
Z_{\bar{x}}\cap Z_{\bar{x}}^{\prime},
\mathbb{Z}/l(n-1)
)
\simeq 
\cdots
\simeq
\mu_{l}^{\otimes (n-s)}
\end{equation*}
by the same argument as in the proof of the isomorphism (\ref{Vanipos}).
\end{rem}

\subsection{Mixed characteristic cases}
We compute the motivic cohomology in mixed characteristic.
\begin{prop}\upshape\label{VSC}
Let $B$ be a discrete valuation ring of mixed characteristic $(0, p)$
and 
$\pi$ a prime element of $B$.
Let
\begin{equation*}
C=
B[T_{0}, T_{1}, \cdots, T_{N}]/
(
T_{0}\cdots T_{a}
-\pi
)
\end{equation*}
for 
$0\leq a\leq N$.
Then we have
\begin{equation}\label{Vsg}
\HO^{q}_{\Zar}(
C, \mathbb{Z}(n)
)    
=0
\end{equation}
for $q> n+1$.
Moreover, we have
\begin{equation}\label{mVsg}
\HO^{q}_{\Zar}(
C, \mathbb{Z}/m(n)
)    
=0
\end{equation}
for any integers $m\geq 0$ and $q> n+1$.
\end{prop}
\begin{proof}\upshape
Since we have an exact sequence
\begin{equation*}
\cdots 
\to
\HO^{s}_{\Zar}(C, \mathbb{Z}(n))
\xrightarrow{\times m}
\HO^{s}_{\Zar}(C, \mathbb{Z}(n))
\to
\HO^{s}_{\Zar}(C, \mathbb{Z}/m(n))
\to \cdots
\end{equation*}
for any integer $s$, the equation (\ref{mVsg}) follows from the equation (\ref{Vsg}).
So it suffices to prove the equation (\ref{Vsg}).

If $a<N$, we have an isomorphism
\begin{equation*}
\HO^{q}_{\Zar}(
C,
\mathbb{Z}(n)
)    
\simeq
\HO^{q}_{\Zar}(
B[T_{0}, \cdots, T_{a}]/(
T_{0}\cdots T_{a}-\pi
),
\mathbb{Z}(n)
)
\end{equation*}
by \cite[p.781, Corollary 3.5]{Ge}.
So it suffices to show (\ref{Vsg}) in the case where $a=N$.
We prove the statement by induction on $a=N$.

Let $a=N=0$. Then we have an isomorphism
$C\simeq B$. So we have the equation (\ref{Vsg}) 
by 
\cite[p.786, Corollary 4.4]{Ge}.

Assume that the equation (\ref{Vsg}) holds in the case where $a=N\leq b$.
We have a homomorphism of polynomial rings over $B$
\begin{equation*}
B[T_{0}^{\prime}, \cdots, T_{b}^{\prime}, T_{b+1}^{\prime}]
\to
B[T_{0}, \cdots, T_{b}, T_{b+1}]
\end{equation*}
which sends $T_{b}^{\prime}$ to $T_{b}T_{b+1}$
and sends $T_{r}^{\prime}$ to $T_{r}$ for $r\neq b$.
Then this homomorphism induces an isomorphism
\begin{align}\label{isofo}
&\left(
B[T_{0}^{\prime}, \cdots, T_{b}^{\prime}, T_{b+1}^{\prime}]
/(
T_{0}^{\prime}\cdots 
T_{b}^{\prime}-\pi
)
\right)_{T_{b+1}^{\prime}} \nonumber 
\\
\xrightarrow{\simeq}
&\left(
B[T_{0}, \cdots, T_{b}, T_{b+1}]
/(
T_{0}\cdots 
T_{b}\cdot T_{b+1}-\pi
)
\right)_{T_{b+1}}. 
\end{align}
Moreover, we have an isomorphism
\begin{align*}
&\left(
B[T_{0}^{\prime}, \cdots, T_{b-1}^{\prime}, T_{b}^{\prime}, T_{b+1}^{\prime}]
/(
T_{0}^{\prime}\cdots T_{b-1}^{\prime}
\cdot T_{b}^{\prime}-\pi
)
\right)_{T_{b+1}^{\prime}}  
\\
\simeq
&\bigl(
\left(
B[T_{0}^{\prime}, \cdots, T_{b-1}^{\prime}, T_{b}^{\prime}]
/(
T_{0}^{\prime}\cdots T_{b-1}^{\prime}
\cdot T_{b}^{\prime}-\pi
)
\right)
\otimes_{B}
B[T_{b+1}^{\prime}]
\bigr)_{T_{b+1}^{\prime}}.
\end{align*}
Since there exists an exact sequence
\begin{align*}
&\HO^{q}_{\Zar}\left(
(B[T_{0}^{\prime}, \cdots, T_{b}^{\prime}]/(T_{0}^{\prime}\cdots T_{b}^{\prime}-\pi))
\otimes_{B}B[T_{b+1}^{\prime}],
\mathbb{Z}(n)
\right)  \\  
\to  
&\HO^{q}_{\Zar}(
C_{T_{b+1}},
\mathbb{Z}(n)
) 
\to
\HO^{q-1}_{\Zar}(
B[T_{0}^{\prime}, \cdots, T_{b}^{\prime}]/(T_{0}^{\prime}\cdots T_{b}^{\prime}-\pi
),
\mathbb{Z}(n-1)
)
\end{align*}
for $N=a=b+1$
by the localization theorem (\cite[p.779, Theorem 3.2]{Ge}),
we have 
\begin{equation}\label{CO}
\HO^{q}_{\Zar}
(C_{T_{b+1}},
\mathbb{Z}(n)
)
=
0
\end{equation}
for $N=a=b+1$ and $q>n+1$ 
by 
\cite[p.781, Corollary 3.5]{Ge} 
and the assumption.
Moreover, we have isomorphisms
\begin{align}\label{CC}
\HO_{\Zar}^{q}(
C/(T_{b+1}),
\mathbb{Z}(n-1)
)    
\simeq
\HO^{q}_{\Zar}(
(B/(\pi))[
T_{0}, \cdots, T_{b}
],
\mathbb{Z}(n-1)
)
\simeq
0
\end{align}
for $N=a=b+1$ and $q\geq n$
by \cite[p.781, Corollary 3.5]{Ge}.
Hence we have the equation (\ref{Vsg}) for $q>n+1$ by
(\ref{CO}), (\ref{CC}) and the localization theorem
(\cite[p.779, Theorem 3.2]{Ge}).
This completes the proof.
\end{proof}
\begin{prop}\upshape\label{SBL}
Let $\mathfrak{X}$ be a regular scheme which is essentially of 
finite type over the spectrum of a discrete valuation ring in
mixed characteristic $(0, p)$. Then we have a quasi-isomorphism
\begin{equation}\label{eBL}
\tau_{\leq n+1}\left(
\mathbb{Z}(n)_{\Zar}^{\mathfrak{X}}
\right)
\xrightarrow{\simeq}
\tau_{\leq n+1}\left(
R\epsilon_{*}
\mathbb{Z}(n)_{\et}^{\mathfrak{X}}
\right)
\end{equation}
where $\epsilon: \mathfrak{X}_{\et}\to \mathfrak{X}_{\Zar}$ 
is the canonical map of sites and $\epsilon_{*}$ is the
forgetful functor. Moreover, 
we have a quasi-isomorphism
\begin{equation}\label{aBL}
\tau_{\leq n+1}\left(
\mathbb{Z}(n)_{\Nis}^{\mathfrak{X}}
\right)
\xrightarrow{\simeq}
\tau_{\leq n+1}\left(
R\alpha_{*}
\mathbb{Z}(n)_{\et}^{\mathfrak{X}}
\right)
\end{equation}
where $\alpha: \mathfrak{X}_{\et}\to \mathfrak{X}_{\Nis}$ 
is the canonical map of sites.
\end{prop}
\begin{proof}\upshape
Let $i: Y\to \mathfrak{X}$ be the inclusion of the closed fiber of $\mathfrak{X}$.
Then we have a quasi-isomorphism
\begin{equation*}
\tau_{\leq n+2}\left(
R\epsilon_{*}\mathbb{Z}(n-1)_{\et}^{Y}
[-2]
\right)
\xrightarrow{\simeq}
\tau_{\leq n+2}i_{*}Ri^{!}\left(
R\epsilon_{*}
\mathbb{Z}(n)_{\et}^{\mathfrak{X}}
\right)
\end{equation*}
by \cite[p.33, Proposition 2.1]{SM}. Moreover, we have a quasi-isomorphism
\begin{equation*}
\mathbb{Z}(n-1)_{\Zar}^{Y}
[-2]
\xrightarrow{\simeq}
i_{*}Ri^{!}
\mathbb{Z}(n)_{\Zar}^{\mathfrak{X}}
\end{equation*}
(cf. \cite[p.780, (6)]{Ge}).
Let $Z$ be a scheme which is finite type over the spectrum of a field. Then we have 
a quasi-isomorphism
\begin{equation*}
\tau_{\leq n+1}
\mathbb{Z}(n)_{\Zar}^{Z}
\xrightarrow{\simeq}
\tau_{\leq n+1}
R\epsilon_{*}\mathbb{Z}(n)_{\et}^{Z}
\end{equation*}
by Proposition \ref{pBL}. So the quasi-isomorphism (\ref{eBL}) follows from the five lemma.
Let $\beta: \mathfrak{X}_{\Nis}\to\mathfrak{X}_{\Zar}$ be the canonical map of sites. 
By applying $\beta^{*}$ to the quasi-isomorphism (\ref{eBL}), 
we obtain the quasi-isomorphism (\ref{aBL}). This completes the proof.
\end{proof}
\begin{prop}\upshape\label{cset}
Let $X$ be a regular scheme which is essentially
finite type over the spectrum of a discrete
valuation ring of mixed characteristic $(0, p)$, 
$Z$ an irreducible component of the closed fiber of $X$
and
$i: Z\to X$ the corresponding closed immersion.
Let $n$ be a non-negative integer.
Then we have a quasi-isomorphism
\begin{equation}\label{qiss}
\tau_{\leq n+2}\left(
\mathbb{Z}(n-1)_{\et}[-2]
\right)
\xrightarrow{\simeq}
\tau_{\leq n+2}Ri^{!}\mathbb{Z}(n)_{\et}
\end{equation}
and a quasi-isomorphism
\begin{equation*}
\tau_{\leq n+1}\left(
\mathbb{Z}/m(n-1)_{\et}[-2]
\right)
\xrightarrow{\simeq}
\tau_{\leq n+1}Ri^{!}\mathbb{Z}/m(n)_{\et}
\end{equation*}
for any positive integer $m$.
\end{prop}
\begin{proof}\upshape
It suffices to show the quasi-isomorphism (\ref{qiss}). 
Let $\mathcal{F}^{\bullet}$ be a bounded below complex on $X_{\et}$.
Then we have an isomorphism
\begin{equation*}
\HO^{q}\left(
(\mathcal{O}_{X, \bar{x}})_{\et},
\mathcal{F}^{\bullet}
\right)
\xrightarrow{\simeq}
\Gamma\left(
(\mathcal{O}_{X, \bar{x}})_{\et},
\mathcal{H}^{q}(\mathcal{F}^{\bullet})
\right)
\end{equation*}
for any integer $q\geq 0$ where
$\mathcal{O}_{X, \bar{x}}$ is the strict henselian of $X$ 
at a point $x\in X$. So it suffices to prove that 
the homomorphism
\begin{equation*}
\HO^{q-2}(
Z_{\et}, 
\mathbb{Z}(n-1)
)    
\xrightarrow{\simeq}
\HO^{q}_{Z}(
X_{\et},
\mathbb{Z}(n)
)
\end{equation*}
is an isomorphism for $q\leq n+2$.

Let $Y$ be the closed fiber of $X$
and $U=X\setminus Z$.
Let $j: W\to Y$ be an open immersion with 
$W=Y\cap U$. Then we have quasi-isomorphisms
\begin{equation*}
\tau_{\leq n}\epsilon^{*}Rj_{*}\mathbb{Z}(n-1)_{\Zar} 
\simeq
\tau_{\leq n}\epsilon^{*}Rj_{*}
\left(
\tau_{\leq n}R\epsilon_{*}\mathbb{Z}(n-1)_{\et}
\right)
\simeq
\tau_{\leq n}Rj_{*}\mathbb{Z}(n-1)_{\et}
\end{equation*}
by Proposition \ref{pBL}. Here 
$\epsilon: W_{\et}\to W_{\Zar}$ is the canonical map
of sites.
So we have a commutative diagram
\footnotesize
\begin{equation*}
\xymatrix{
\cdots
\ar[r]
&
\HO^{q-2}_{\et}(
Z,
\mathbb{Z}(n-1)
)
\ar[r]\ar[d]
&
\HO^{q-2}_{\et}(
Y,
\mathbb{Z}(n-1)
)
\ar[r]\ar[d]
&
\HO^{q-2}_{\et}(
W,
\mathbb{Z}(n-1)
)  \ar[d]
\\
\cdots 
\ar[r]
&
\HO^{q}_{Z}(
X_{\et},
\mathbb{Z}(n)
)
\ar[r]
&
\HO^{q}_{Y}(
X_{\et},
\mathbb{Z}(n)
)
\ar[r]
&
\HO^{q}_{W}(
U_{\et},
\mathbb{Z}(n)
)
}    
\end{equation*}
\normalsize
for $q\leq n+2$ by
\cite[p.776, Proposition 2.2.a)]{Ge}
and Lemma \ref{Gyscom}.
Here the sequences are exact.
Moreover, we have isomorphisms
\begin{align*}
\HO_{\et}^{q-2}(
Y,
\mathbb{Z}(n-1)
)    
\xrightarrow{\simeq}
\HO^{q}_{Y}(
X_{\et},
\mathbb{Z}(n)
)
\end{align*}
and
\begin{align*}
\HO_{\et}^{q-2}(
W,
\mathbb{Z}(n-1)
)    
\xrightarrow{\simeq}
\HO^{q}_{W}(
U_{\et},
\mathbb{Z}(n)
)
\end{align*}
for $q\leq n+2$ by \cite[p.33, Proposition 2.1]{SM}.
Therefore the statement follows from the five lemma.
\end{proof}

\begin{cor}\upshape\label{PurC}
Let $R$ be a local ring of a regular scheme which is finite type 
over the spectrum of a discrete valuation ring of mixed characteristic
$(0, p)$. Let $\mathfrak{m}$ be the maximal ideal of $\spec(R)$,
$i: \spec(R/\mathfrak{m})\to \spec(R)$ the corresponding closed immersion
and $c=\operatorname{dim}(R)$.
Let $n$ be a non-negative integer. 
Then we have a quasi-isomorphism
\begin{equation}\label{purlm}
\tau_{\leq n+2}\left(
\mathbb{Z}(n-c)_{\et}^{\spec(R/\mathfrak{m})}[-2c]
\right)
\xrightarrow{\sim}
\tau_{\leq n+2}
Ri^{!}\mathbb{Z}(n)_{\et}^{\spec(R)}
\end{equation}
and a quasi-isomorphism
\begin{equation*}
\tau_{\leq n+1}\left(
\mathbb{Z}/m(n-c)_{\et}^{\spec(R/\mathfrak{m})}[-2c]
\right)
\xrightarrow{\sim}
\tau_{\leq n+1}
Ri^{!}\mathbb{Z}/m(n)_{\et}^{\spec(R)}
\end{equation*}
for any positive integer $m$.
\end{cor}
\begin{proof}\upshape
It suffices to show the quasi-isomorphism (\ref{purlm}).  
Let $Z$ be an irreducible component of the closed fiber of $X$. 
Put $U=\spec(R)\setminus\spec(R/\mathfrak{m})$
and $W=Z\setminus\spec(R/\mathfrak{m})$. 
Since we have a commutative diagram
\small
\begin{equation*}
\xymatrix{
\cdots\ar[r]
&
\HO^{q-2c}_{\et}(
\kappa(\mathfrak{m}),
\mathbb{Z}(n-c)
)
\ar[r]\ar[d]
&
\HO^{q-2}_{\et}(
Z,
\mathbb{Z}(n-1)
)
\ar[r]\ar[d]
&
\HO^{q-2}_{\et}(
W,
\mathbb{Z}(n-1)
) \ar[d]
\\
\cdots
\ar[r]
&
\HO^{q}_{\mathfrak{m}}(
R_{\et},
\mathbb{Z}(n)
)
\ar[r]
&
\HO^{q}_{Z}(
R_{\et},
\mathbb{Z}(n)
)\ar[r]
&
\HO^{q}_{W}(
U_{\et},
\mathbb{Z}(n)
)
}    
\end{equation*}
\normalsize
for $q\leq n+2$, the statement follows from Proposition \ref{cset}.
\end{proof}
\begin{cor}\upshape
Let $X$ be a regular scheme 
which is
essentially finite type over the
spectrum of a discrete valuation ring of mixed characteristic 
$(0, p)$,
$Z$ an irreducible closed subscheme of $X$
and $i: Z\to X$ the corresponding closed immersion of codimension $c$.
Let $n$ be a non-negative integer.
Then we have a quasi-isomorphism
\begin{equation*}
\tau_{\leq n+1}\left(
\mathbb{Z}/p^{r}(n-c)^{Z}_{\et}[-2c]
\right)
\simeq
\tau_{\leq n+1}\left(
Ri^{!}\mathbb{Z}/p^{r}(n)^{X}_{\et}
\right)
\end{equation*}
for any positive integer $r$.
\end{cor}
\begin{proof}\upshape
It suffices to prove the statement in the case where 
$X$ is the spectrum of a local ring $R$ of a regular scheme
which is finite type over the spectrum of a discrete valuation ring
of mixed characteristic $(0, p)$. We prove the statement by induction on 
$\operatorname{dim}(R)$.

Assume that $\operatorname{dim}(R)\leq 2$.
Then the statement follows from \cite[p.33, Proposition 2.1]{SM}.

Assume that the statement holds 
in the case where $\operatorname{dim}(R)=d\geq 2$.
Suppose that $\operatorname{dim}(R)=d+1$. 
Put $U=\spec(R)\setminus \spec(R/\mathfrak{m})$
and
$W=Z\setminus\spec(R/\mathfrak{m})$.
Since 
$\operatorname{char}(R/\mathfrak{m})=p>0$
and
\begin{equation*}
(n-2d)
-
(n-d-1)
=1-d<0,
\end{equation*}
we have
\begin{equation*}
\HO^{n-2d}_{\et}\left(
R/\mathfrak{m},\mathbb{Z}/p^{r}(n-d-1)
\right)
=0
\end{equation*}
by \cite[Theorem 8.5]{Ge-L}. So we have a commutative diagram
\scriptsize
\begin{equation*}
\xymatrix@C=16pt@R=20pt{
\cdots\ar[r]
&
\HO^{n-2d-1}_{\et}
(
R/\mathfrak{m},
\mathbb{Z}/p^{r}(n-d-1)
)
\ar[r]\ar[d]
&
\HO^{n-2c+1}_{\et}
(
Z,
\mathbb{Z}/p^{r}(n-c)
)
\ar[r]\ar[d]
&
\HO^{n-2c+1}_{\et}
(
W,
\mathbb{Z}/p^{r}(n-c)
)
\ar[r]\ar[d]
&
0  \\
\cdots\ar[r]
&
\HO_{\mathfrak{m}}^{n+1}(
R_{\et},
\mathbb{Z}/p^{r}(n)
)
\ar[r]
&
\HO^{n+1}_{Z}(
R_{\et},
\mathbb{Z}/p^{r}(n)
)
\ar[r]
&
\HO^{n+1}_{W}(
U_{\et},
\mathbb{Z}/p^{r}(n)
).
}    
\end{equation*}
\normalsize
By using the five lemma, 
the statement in the case where $\spec(R)=d+1$
follows from Corollary \ref{PurC} and the assumption of induction. 
This completes the proof.
\end{proof}
\begin{prop}\upshape\label{VaniMtM1}
Let the notations be the same as in Proposition \ref{VSC}. Then we have
\begin{equation*}
\HO^{n+1}_{\Zar}(
C,
\mathbb{Z}/m(n)
)    
=0
\end{equation*}
for any positive number $m$.
\end{prop}
\begin{proof}\upshape
Since we have an exact sequence
\begin{align*}
\HO^{n+1}_{\Zar}(
C, \mathbb{Z}/m(n)
)    
\to
\HO^{n+1}_{\Zar}(
C, \mathbb{Z}/(m+m^{\prime})(n)
)    
\to
\HO^{n+1}_{\Zar}(
C, \mathbb{Z}/m^{\prime}(n)
)    
\end{align*}
for any positive numbers $m$ and $m^{\prime}$,
it suffices to prove the statement in the case where
$m$ is a prime number. By \cite[p.781, Corollary 3.5]{Ge}, 
it suffices to prove the statement
in the case where
\begin{equation*}
C=
B[T_{0}, \cdots, T_{a}]/
(T_{0}\cdots T_{a}-\pi)
\end{equation*}
for any positive integer $a$. Moreover, we have exact sequences
\begin{align*}
&\HO^{n+1}_{\Zar}(
B[T_{0}^{\prime}, \cdots, T_{a-1}^{\prime}]/(T_{0}^{\prime}\cdots T_{a-1}^{\prime}-\pi)\otimes_{B}B[T_{a}^{\prime}],
\mathbb{Z}/m(n)
)  \\   
\to
&\HO^{n+1}_{\Zar}(
C_{T_{a}}, \mathbb{Z}/m(n)
)
\to
\HO^{n}_{\Zar}(
B[T_{0}^{\prime}, \cdots, T_{a-1}^{\prime}]/(T_{0}^{\prime}\cdots T_{a-1}^{\prime}-\pi),
\mathbb{Z}/m(n-1)
)
\end{align*}
and
\begin{align*}
&\HO^{n}_{\Zar}(
C_{T_{a}}, \mathbb{Z}/m(n)
)    
\to
\HO^{n-1}_{\Zar}\left(
(B/(\pi))[T_{0}, \cdots, T_{a-1}],
\mathbb{Z}/m(n-1)
\right)  \\
\to
&\HO^{n+1}_{\Zar}(
C, \mathbb{Z}/m(n)
)
\to
\HO^{n+1}_{\Zar}(
C_{T_{a}},
\mathbb{Z}/m(n)
)
\end{align*}
by the isomorphism (\ref{isofo}) and the localization theorem (\cite[p.779, Theorem 3.2]{Ge}). 
So, by induction on $a$, 
it suffices to prove that the
homomorphism
\begin{align*}
\HO^{n}_{\Zar}(
C_{T_{a}}, \mathbb{Z}/m(n)
)    
\to
\HO^{n-1}_{\Zar}\left(
(B/(\pi))[T_{0}, \cdots, T_{a-1}],
\mathbb{Z}/m(n-1)
\right) 
\end{align*}
is surjective for any positive integer $a$ and any prime number $m$.
Moreover, it suffices to prove that the
homomorphism
\begin{align}\label{connm}
\HO^{n}_{\et}(
C_{T_{a}}, \mathbb{Z}/m(n)
)    
\to
\HO^{n-1}_{\et}\left(
(B/(\pi))[T_{0}, \cdots, T_{a-1}],
\mathbb{Z}/m(n-1)
\right) 
\end{align}
is surjective for any positive integer $a$ and any prime number $m$
by Proposition \ref{SBL} and Proposition \ref{cset}.

Let $\mathfrak{X}$ be a semistable family over the spectrum of a discrete valuation ring $A$
of mixed-characteristic $(0, p)$. Then we have quasi-isomorphisms
\begin{equation*}
\tau_{\leq n}\left(
\mathbb{Z}/p(n)_{\et}^{\mathfrak{X}}
\right)
\simeq
\mathfrak{T}_{1}(n)^{\mathfrak{X}}
\end{equation*}
and
\begin{equation*}
\tau_{\leq n}\left(
\mathbb{Z}/m(n)_{\et}^{\mathfrak{X}}
\right)
\simeq
\mu_{m}^{\otimes n}
\end{equation*}
in the case where $m$ is a prime number which is prime to 
$\operatorname{char}(B/(\pi))$ by \cite[p.209, Remark 7.2]{SaR}
and \cite[Remark 4.7]{Sak4}.
So we have a commutative diagram
\begin{equation*}
\xymatrix{
\HO^{1}_{\et}(
\mathfrak{X}, 
\mathbb{Z}/m(1)
)
\otimes\cdots\otimes
\HO^{1}_{\et}(
\mathfrak{X}, 
\mathbb{Z}/m(1)
)
\ar[r]^-{\cup}\ar[d]
&
\HO^{n}_{\et}(
\mathfrak{X}, 
\mathbb{Z}/m(n)
)  \ar[d]
\\
\HO^{1}_{\et}(
k(\mathfrak{X}), 
\mathbb{Z}/m(1)
)
\otimes\cdots\otimes
\HO^{1}_{\et}(
k(\mathfrak{X}), 
\mathbb{Z}/m(1)
)
\ar[r]_-{\cup}
&
\HO^{n}_{\et}(
k(\mathfrak{X}), 
\mathbb{Z}/m(n)
) 
}    
\end{equation*}
by \cite[p.538, Proposition 4.2.6]{SaP}. Moreover, we have a commutative diagram
\begin{equation*}
\xymatrix{
\HO^{n}_{\et}(C_{T_{a}},
\mathbb{Z}/m(n)
)
\ar[r]\ar[d]
&
\HO^{n-1}_{\et}\left(
(B/(\pi))[T_{0}, \cdots, T_{a-1}],
\mathbb{Z}/m(n-1)
\right)
\ar@{^{(}-_>}[d]
\\
\HO^{n}_{\et}(
k(C_{T_{a}}),
\mathbb{Z}/m(n)
)
\ar[r]_-{\delta}
&
\HO^{n-1}_{\et}(
k((B/(\pi))[T_{0}, \cdots, T_{a-1}]),
\mathbb{Z}/m(n-1)
)
}    
\end{equation*}
where the right map is injective by 
\cite[p.781, Corollary 3.5]{Ge} and \cite[Corollary 6.9]{Sak5}.
Then an element 
\begin{equation*}
\{a_{0}, \cdots, a_{n-1}, T_{a}\}    
\end{equation*}
$(a_{0}, \cdots, a_{n-1}\in B^{*})$
maps to
\begin{equation*}
\{\bar{a_{0}}, \cdots, \bar{a_{n-1}}\}    
\end{equation*}
via $\delta$. So the homomorphism (\ref{connm}) is surjective by the Bloch-Kato conjecture (\cite{V2})
and \cite[p.113, Theorem (2.1)]{B-K}. This completes the proof.
\end{proof}
\begin{cor}\upshape
Let $B$ be a discrete valuation ring 
and $\pi$ a prime element of $B$.
Let
\begin{equation*}
C=B[T_{0}, T_{1}]/(
T_{0}T_{1}-\pi
)    
\end{equation*}
and $l$ be a positive integer which is
invertible in $A$.
Suppose that $A$ contains $l$-th roots of unity.
Then there exists an exact sequence
\begin{align*}
0\to  
&\HO^{n+1}_{\et}(
C, \mu_{l}^{\otimes n}
)
\to
\displaystyle\bigoplus_{x\in \spec(C)^{(0)}}
\HO^{n+1}_{\et}(
\kappa(x), 
\mu_{l}^{\otimes n}
)
\to
\displaystyle\bigoplus_{x\in \spec(C)^{(1)}}
\HO^{n}_{\et}(
\kappa(x), 
\mu_{l}^{\otimes (n-1)}
)
\\
\to
&\displaystyle\bigoplus_{x\in \spec(C)^{(2)}}
\HO^{n-1}_{\et}(
\kappa(x), 
\mu_{l}^{\otimes (n-2)}
)
\to
0
\end{align*}
for any positive integer $n$.
\end{cor}
\begin{proof}\upshape
Let $R$ be 
the henselization of a regular local ring 
which is essentially of finite type over $\spec(B)$. 
Let $\mathfrak{m}$ be the maximal ideal of $R$ and 
$g\in \mathfrak{m}\setminus\mathfrak{m}^{2}$.
Suppose that $\operatorname{dim}(R)=2$.
Then we have
\begin{align*}
\HO^{q-2}_{\Zar}(
R/(g), \mathbb{Z}/l(n-1)
)
=
\HO^{q}_{\Zar}(
R_{g}, \mathbb{Z}/l(n)
)
=0
\end{align*}
for $q\geq n+2$
by \cite[p.786, Corollary 4.4]{Ge} and the vanishing theorem.
So we have 
\begin{align*}
\HO^{q}_{\Zar}(
R, \mathbb{Z}/l(n)
) 
=0
\end{align*}
for $q\geq n+1$ by the localization theorem
(\cite[p.779, Theorem 3.2]{Ge}) and 
\cite[p.37, Corollary 7]{SM19}.
Moreover, we have a quasi-isomorphism
\begin{equation*}
\tau_{\leq n}
\left(
\mathbb{Z}/l(n)_{\Nis}^{\spec(C)}
\right)
\simeq
\tau_{\leq n}
\left(
R\alpha_{*}\mathbb{Z}/l(n)_{\et}^{\spec(C)}
\right)
\end{equation*}
by 
Proposition \ref{SBL}.
Hence we have
a quasi-isomorphism
\begin{align*}
\mathbb{Z}/l(n)_{\Nis}^{\spec(C)}
\simeq
\tau_{\leq n}
R\alpha_{*}\mu_{l}^{\otimes n}
\end{align*}
by \cite[Remark 4.7]{Sak4}. Since $B$ contains $l$-th roots of unity, 
we have quasi-isomorphisms
\begin{align*}
\tau_{\leq N}R\alpha_{*}\mu_{l}^{\otimes n}
\simeq 
\tau_{\leq N}R\alpha_{*}\mu_{l}^{\otimes N}
\simeq
\mathbb{Z}/l(N)_{\Nis}^{\spec(C)}
\end{align*}
for any integer $N\geq 0$ by the above. So we have a distinguished triangle
\begin{align*}
\cdots\to 
\mathbb{Z}/l(n)_{\Nis}^{\spec(C)}
\to
\mathbb{Z}/l(n+1)_{\Nis}^{\spec(C)}
\to
R^{n+1}\alpha_{*}\mu_{l}^{\otimes n}[-(n+1)]
\to 
\cdots.
\end{align*}
Hence we have
isomorphisms
\begin{align*}
\HO^{n+1}_{\et}(
C, \mu_{l}^{\otimes n}
)
\simeq
\HO^{0}_{\Nis}(
C, R\alpha_{*}^{n+1}
\mu_{l}^{\otimes n}
)
&&
\textrm{and}
&&
\HO_{\Nis}^{j}
(C,
R\alpha_{*}^{n+1}\mu_{l}^{\otimes n}
)
\simeq 0
\end{align*}
for $j>0$ by Proposition \ref{VSC} and Proposition \ref{VaniMtM1}.
Moreover, we have a flabby resolution 
\begin{align*}
0
\to
R\alpha_{*}^{n+1}\mu_{l}^{\otimes n}
\to
\displaystyle\bigoplus_{x\in \spec(C)^{(0)}}
(i_{x})_{*}R^{n+1}
\alpha_{*}\mu_{l}^{\otimes n}
\to
\displaystyle\bigoplus_{x\in \spec(C)^{(1)}}
(i_{x})_{*}R^{n}
\alpha_{*}\mu_{l}^{\otimes (n-1)}  \\
\to
\displaystyle\bigoplus_{x\in \spec(C)^{(2)}}
(i_{x})_{*}R^{n-1}
\alpha_{*}\mu_{l}^{\otimes (n-2)}
\to 0
\end{align*}
by 
\cite[p.38, Theorem 9]{SM19}.
Here
we write
$i_{x}$ 
for
the canonical map
$\spec(\kappa(x))\to \spec(C)$.
Therefore the statement follows.
\end{proof}

\subsection{Application}
Let $B$ be a henselian valuation ring of mixed characeristic $(0, p)$,
$\pi$ a prime element of $B$ 
and $C=B[T_{0}, \cdots, T_{N}]/(T_{0}\cdots T_{a}-\pi)$
for $a\leq N$. Then we observe a relation between
$\Gamma\left(C, R^{n+1}\alpha_{*}\mu_{l}^{\otimes n}\right)$
and
$\Gamma\left(C/(\pi), R^{n+1}\alpha_{*}\mu_{l}^{\otimes n}\right)$
for an integer $l$ which is prime to $p$.
\begin{lem}\upshape\label{Vanis}
Let $B$ be a discrete valuation ring and $\pi$ a prime element of $B$.
Let $n\geq 0$ and $m>0$ be integers.
Then we have
\begin{equation*}
\HO^{q}_{\Zar}\left(
(
B[T_{0}, \cdots, T_{N}]/
(
T_{b}\cdots T_{a}-\pi
)
)_{T_{0}\cdots T_{b}},
\mathbb{Z}/m(n)
\right)
=0
\end{equation*}
for integers $a\leq N$, $b<a$ and $q\geq n+1$.
\end{lem}
\begin{proof}
We have a homomorphism of polynomial rings over $B$
\begin{align*}
B[T_{0}^{(1)}, 
\cdots, T^{(1)}_{N}]
\to
B[T_{0}, \cdots, T_{N}]
\end{align*}
which sends $T^{(1)}_{b+1}$ to $T_{b}T_{b+1}$ and 
sends $T^{(1)}_{r}$ to $T_{r}$ for $r\neq b+1$.
Then this homomorphism induces an isomorphism
\begin{align}\label{dowb}
&\left(
B[T^{(1)}_{0}, \cdots, T^{(1)}_{N}]/
(T^{(1)}_{b+1}\cdots T^{(1)}_{a}-\pi)
\right)_{T^{(1)}_{0}\cdots T^{(1)}_{b}}
\nonumber
\\
\xrightarrow{\simeq}
&\left(
B[T_{0}, \cdots, T_{N}]/
(T_{b}\cdots T_{a}-\pi)
\right)_{T_{0}\cdots T_{b}}.
\end{align}
Moreover, we have a homomorphism of polynomial rings over $B$
\begin{equation*}
B[T^{(1)}_{0}, \cdots T^{(1)}_{N}] 
\to
B[T^{(2)}_{0}, \cdots T^{(2)}_{N}] 
\end{equation*}
which sends $T^{(1)}_{r}$ to $T^{(2)}_{r}$ 
for $0\leq r\leq b-1$, 
sends $T_{b}^{(1)}$ to $T_{N}^{(2)}$
and
sends $T^{(1)}_{r}$ to $T^{(2)}_{r-1}$
for $b+1\leq r\leq N$. 
Then this homomorphism induces
\begin{align*}
&\left(
B[T^{(1)}_{0}, \cdots, T^{(1)}_{N}]/
(T^{(1)}_{b+1}\cdots T^{(1)}_{a}-\pi)
\right)_{T^{(1)}_{0}\cdots T^{(1)}_{b}}
\\
\xrightarrow{\simeq}
&\left(
B[T^{(2)}_{0}, \cdots, T^{(2)}_{N}]/
(T^{(2)}_{b}\cdots T^{(2)}_{a-1}-\pi)
\right)_{(T^{(2)}_{0}\cdots T^{(2)}_{b-1})\cdot T^{(2)}_{N}}.
\end{align*}
So we have an exact sequence
\begin{align}\label{index}
&\HO^{q+1}_{\Zar}\left(
\left(
B[T^{(2)}_{0}, \cdots, T^{(2)}_{N}]/
(T^{(2)}_{b}\cdots T^{(2)}_{a-1}-\pi)
\right)_{T^{(2)}_{0}\cdots T^{(2)}_{b-1}},
\mathbb{Z}/m(n)
\right)  
\nonumber
\\
\to
&\HO^{q+1}_{\Zar}\left(
\left(
B[T_{0}, \cdots, T_{N}]/
(T_{b}\cdots T_{a}-\pi)
\right)_{T_{0}\cdots T_{b}},
\mathbb{Z}/m(n)
\right)  
\nonumber
\\
\to
&\HO^{q}_{\Zar}\left(
\left(
B[T^{(2)}_{0}, \cdots, T^{(2)}_{N-1}]/
(T^{(2)}_{b}\cdots T^{(2)}_{a-1}-\pi)
\right)_{T^{(2)}_{0}\cdots T^{(2)}_{b-1}},
\mathbb{Z}/m(n-1)
\right) 
\end{align}
for any integer $q$ 
by the localization theorem (\cite[p.779, Theorem 3.2]{Ge}).
By induction on $a$, the statement follows from
the isomorphism (\ref{dowb}) and the exact sequence (\ref{index}).
\end{proof}

Let $B$ be a discrete valuation of mixed characteristic $(0, p)$ and 
$\pi$ a prime element of $B$. Let 
\begin{equation*}
C=
B[
T_{0}, \cdots, T_{N}
]/
(
T_{0}\cdots T_{a}-\pi
)
\end{equation*}
for $0\leq a\leq N$ and
\begin{math}
i: \spec(
C/(\pi)
)
\to
\spec(
C
)
\end{math}
be the closed fiber of 
$\spec(C)$
and
$l$ an integer which is prime to $p$. Then we have the composition
\begin{align}\label{mfmtm}
&\tau_{\leq n+1}(
\mathbb{Z}/l(n)_{\et}^{\spec(C)}
)    
\simeq 
\mu_{l}^{\otimes n}
\to
i_{*}\mu_{l}^{\otimes n}  \nonumber 
\\
\to
&i_{*}\left(
\tau_{\leq 0}(\mathbb{Z}/l(n)_{\et}^{\spec(C/(\pi))}
)
\right)
\to 
i_{*}\mathbb{Z}/l(n)_{\et}^{\spec(C/(\pi))}
\end{align}
by \cite[Remark 4.7]{Sak4} and Proposition \ref{constm}.
\begin{prop}\upshape\label{fscH}
Let the notations be the same as above. Let $n$ be a non-negative integer.
Suppose that $B$ is a henselian discrete valuation ring. 
Then the homomorphism
\begin{align}\label{embaffl}
&\HO^{n+1}_{\et}(
B[T_{0}, \cdots, T_{N}]/(T_{0}\cdots T_{a}-\pi),
\mathbb{Z}/l(n)
)  \nonumber
\\    
\to
&\HO^{n+1}_{\et}\left(
(B/(\pi))[T_{0}, \cdots, T_{N}]/(T_{0}\cdots T_{a}),
\mathbb{Z}/l(n)
\right)
\end{align}
is injective where the homomorphism (\ref{embaffl}) is induced by the composition (\ref{mfmtm}).
\end{prop}
\begin{proof}\upshape
By the localization theorem (\cite[p.779, Theorem 3.2]{Ge}), Lemma \ref{Vanis} and the isomorphism (\ref{dowb}), the homomorphism 
\begin{align*}
&\HO^{n}_{\Zar}\left(
(B[S_{0}, \cdots, S_{N}]/(S_{b}\cdots S_{a}-\pi))_{S_{0}\cdots S_{b-1}\cdot S_{b}},
\mathbb{Z}/l(n)
\right)
\\
\to
&
\HO^{n-1}_{\Zar}\left(
(B[S_{0}, \cdots, S_{N}]/(S_{b}, \pi))_{S_{0}\cdots S_{b-1}},
\mathbb{Z}/l(n-1)
\right)
\end{align*}
is surjective. So the homomorphism
\begin{align*}
&\HO^{n+1}_{\et}\left(
(B[S_{0}, \cdots, S_{N}]/(S_{b}\cdots S_{a}-\pi))_{S_{0}\cdots S_{b-1}},
\mathbb{Z}/l(n)
\right)  
\\
\to
&\HO^{n+1}_{\et}\left(
(
B[S_{0}, \cdots, S_{N}]/(S_{b}\cdots S_{a}-\pi))_{S_{0}\cdots S_{b-1}\cdot S_{b}},
\mathbb{Z}/l(n)
\right)
\end{align*}
is injective by Proposition \ref{SBL} and Proposition \ref{cset}. 
Moreover, we have a homomorphism of polynomial ring over $B$
\begin{align*}
B[S^{\prime}_{0}, \cdots, S^{\prime}_{N}] 
\to
B[S_{0}, \cdots, S_{N}]
\end{align*}
which $S_{r}^{\prime}$ sends to $S_{r}$ for $r\neq b+1$ and
$S_{b+1}^{\prime}$ sends to $S_{b}S_{b+1}$. Then this homomorphism induces an isomorphism
\begin{align*}
\left(
B[S^{\prime}_{0}, \cdots, S^{\prime}_{N}]/
(S^{\prime}_{b+1}\cdots S^{\prime}_{a}-\pi)
\right)_{S^{\prime}_{0}\cdots S^{\prime}_{b}}
\xrightarrow{\simeq}
\left(
B[S_{0}, \cdots, S_{N}]/
(S_{b}\cdots S_{a}-\pi)
\right)_{S_{0}\cdots S_{b}}
\end{align*}
and so the homomorphism
\begin{align*}
&\HO^{n+1}_{\et}\left(
\left(B[S_{0}, \cdots, S_{N}]/(S_{b}\cdots S_{a}-\pi)\right)_{S_{0}\cdots S_{b-1}},
\mathbb{Z}/l(n)
\right)  \\
\to
&
\HO^{n+1}_{\et}\left(
\left(
B[S^{\prime}_{0}, \cdots, S^{\prime}_{N}]/
(S^{\prime}_{b+1}\cdots S^{\prime}_{a}-\pi)
\right)_{S^{\prime}_{0}\cdots S^{\prime}_{b}},
\mathbb{Z}/l(n)
\right)
\end{align*}
is injective. Moreover, we have a homomorphism
of polynomial rings over $B$
\begin{align*}
B[S_{0}, \cdots, S_{N}]
\to
B[S^{\prime}_{0}, \cdots, S^{\prime}_{N}]
\end{align*}
which $S_{r}$ sends to $S^{\prime}_{r}$ for $0\leq r\leq a-1$,
$S_{a}$ sends to $S^{\prime}_{N}+\pi$ and 
$S_{r}$ sends to $S^{\prime}_{r-1}$ for $a+1\leq r\leq N$.
Then this homomorphism induces an isomorphism
\begin{align*}
\left(
B[S_{0}\cdots S_{N}]
/(S_{a}-\pi)
\right)_{S_{0}\cdots S_{a-1}}
\xrightarrow{\simeq}
\left(
B[S^{\prime}_{0}\cdots S^{\prime}_{N}]
/(S^{\prime}_{N})
\right)_{S^{\prime}_{0}\cdots S^{\prime}_{a-1}}.
\end{align*}
So we have an injective homomorphism
\begin{align*}
&\HO^{n+1}_{\et}\left(
B[T_{0}, \cdots, T_{N}]/(T_{0}\cdots T_{a}-\pi),
\mathbb{Z}/l(n)
\right)  \\
\to
&\HO^{n+1}_{\et}\left(
B[R_{0}, \cdots, R_{N-1}]_{R_{0}\cdots R_{a-1}},
\mathbb{Z}/l(n)
\right)
\end{align*}
where $R_{0}, \cdots, R_{N-1}$ are indeterminates over $B$.
Since we have a commutative diagram
\footnotesize
\begin{equation*}
\xymatrix{
\HO^{n+1}_{\et}(
B[T_{0}, \cdots, T_{N}]/(T_{0}\cdots T_{a}-\pi),
\mathbb{Z}/l(n)
)
\ar[r]\ar[d]
&
\HO^{n+1}_{\et}(
B[R_{0}, \cdots, R_{N-1}]_{R_{0}\cdots R_{a-1}},
\mathbb{Z}/l(n)
)
\ar[d]
\\
\HO^{n+1}_{\et}\left(
(B/(\pi))[T_{0}, \cdots, T_{N}]/(T_{0}\cdots T_{a}),
\mathbb{Z}/l(n)
\right)
\ar[r]
&
\HO^{n+1}_{\et}\left(
(B/(\pi))[R_{0}, \cdots, R_{N-1}]_{R_{0}\cdots R_{a-1}},
\mathbb{Z}/l(n)
\right)
}    
\end{equation*}
\normalsize
and the upper map is injective, it suffices to prove that the right map
\begin{align}\label{intcet}
&\HO^{n+1}_{\et}\left(
B[R_{0}, \cdots, R_{N-1}]_{R_{0}\cdots R_{a-1}},
\mathbb{Z}/l(n)
\right)  \nonumber \\
\to
&\HO^{n+1}_{\et}\left(
(B/(\pi))[R_{0}, \cdots, R_{N-1}]_{R_{0}\cdots R_{a-1}},
\mathbb{Z}/l(n)
\right)
\end{align}
is injective.

By the localization theorem (\cite[p.779, Theorem 3.2]{Ge}), \cite[p.781, Corollary 3.5]{Ge} 
and \cite[p.786, Corollary 4.4]{Ge}, 
we have
\begin{equation*}
\HO^{n+1}_{\Zar}\left(
(B/(\pi))[R_{0}, \cdots R_{N-1}]_{R_{0}\cdots R_{b}},
\mathbb{Z}/l(n)
\right)
=0
\end{equation*}
for $0\leq b\leq a-1$ and so the homomorphism
\begin{align*}
\HO^{n+1}_{\et}\left(
(B/(\pi))[
R_{0}, \cdots, R_{N}
]_{R_{0}\cdots R_{b}},
\mathbb{Z}/l(n)
\right)    
\to
\HO^{n+1}_{\et}\left(
(B/(\pi))[
R_{0}, \cdots, R_{N}
]_{R_{0}\cdots R_{b+1}},
\mathbb{Z}/l(n)
\right)
\end{align*}
is injective for $0\leq b\leq a-1$ by the same argument as above.
Let $E^{M}_{b}$ be a ring which is isomorphic to
$B[R_{0}, \cdots, R_{M}]_{R_{0}\cdots R_{b}}$.
Then we have a commutative diagram
\scriptsize
\begin{equation*}
\xymatrix{
&
\HO^{n+1}_{\et}(
E^{N-1}_{b-1},
\mathbb{Z}/l(n)
)
\ar[r]\ar[d]
&
\HO^{n+1}_{\et}(
E^{N-1}_{b},
\mathbb{Z}/l(n)
)
\ar[r]\ar[d]
&
\HO^{n}_{\et}(
E^{N-2}_{b-1},
\mathbb{Z}/l(n-1)
) 
\ar[d]
\\
0\ar[r]
&
\HO^{n+1}_{\et}(
E^{N-1}_{b-1}/(\pi),
\mathbb{Z}/l(n)
)
\ar[r]
&
\HO^{n+1}_{\et}(
E^{N-1}_{b}/(\pi),
\mathbb{Z}/l(n)
)
\ar[r]
&
\HO^{n}_{\et}(
E^{N-2}_{b-1}/(\pi),
\mathbb{Z}/l(n-1)
)
}    
\end{equation*}
\normalsize
by \cite[p.774, Theorem 1.2.4]{Ge} and the absolute purity conjecture (\cite{G}) where the sequences are exact. 
Since $A$ is a henselian discrete valuation ring, the homomorphism
\begin{equation*}
\HO^{n}_{\et}(
B[R_{0}, \cdots, R_{M}],
\mathbb{Z}/l(n-1)
)
\to
\HO^{n}_{\et}\left(
(B/(\pi))[R_{0}, \cdots, R_{M}],
\mathbb{Z}/l(n-1)
\right)
\end{equation*}
is an isomorphism
by \cite[p.774, Theorem 1.2.3]{Ge},
\cite[p.240, VI, Corollary 4.20]{M}
and the isomorphism (\ref{isoAI}).
Hence the homomorphism (\ref{intcet})
is injective by induction on $b$. This completes the proof.
\end{proof}
\begin{lem}\upshape\label{inVaniTr}
Let $\mathfrak{X}$ be a semistable family over the spectrum of 
a Dedekind domain
and 
$\alpha: \mathfrak{X}_{\et}\to \mathfrak{X}_{\Nis}$
the canonical map of sites.
Let $m$ be any positive integer.
Then the homomorphism
\begin{align}\label{ftrt}
\HO^{n+2}_{\Nis}\left(
\mathfrak{X}, 
\tau_{\leq n}(R\alpha_{*}\mathbb{Z}/m(n)_{\et})
\right)
\to
\HO^{n+2}_{\Nis}(
\mathfrak{X},
\mathbb{Z}/m(n)
)
\end{align}
is injective. Here the homomorphism (\ref{ftrt}) is induced by the composite
\begin{align*}
\tau_{\leq n}(R\alpha_{*}\mathbb{Z}/m(n)_{\et})
\xleftarrow{\simeq}
\tau_{\leq n}(
\mathbb{Z}/m(n)_{\Nis}
)
\to
\mathbb{Z}/m(n)_{\Nis}
\end{align*}
where the first map is a quasi-isomorphism by Proposition \ref{SBL}.
\end{lem}
\begin{proof}\upshape
By \cite[p.37, Corollary 7]{SM19} and \cite[Proposition 4.5]{Sak4},
we have a quasi-isomorphism
\begin{equation*}
\tau_{\leq n+1}\left(
\mathbb{Z}/m(n)_{\Nis}^{\mathfrak{X}}
\right)
\xrightarrow{\simeq}
\tau_{\leq n}\left(
R\alpha_{*}\mathbb{Z}/m(n)_{\et}^{\mathfrak{X}}
\right).
\end{equation*}
Hence the statement follows.
\end{proof}
\begin{cor}\upshape\label{inclaffmt}
Let $B$ be a henselian discrete valuation ring of mixed characteristic $(0, p)$ and
$\pi$ a prime element of $B$. Let $n$ be a non-negative integer and 
$l$ an integer which is prime to $p$.
Let 
\begin{equation*}
C=
B[T_{0}, \cdots, T_{N}]/
(T_{0}\cdots T_{a}-\pi)
\end{equation*}
for 
$0\leq a\leq N$
and
$\alpha: \spec(C)_{\et}\to \spec(C)_{\Nis}$
the canonical map of sites. Then the homomorphism
\begin{equation*}
\Gamma\left(
\spec(C),
R^{n+1}\alpha_{*}\mathbb{Z}/l(n)_{\et}
\right)
\to
\Gamma\left(
\spec(C/(\pi)),
R^{n+1}\alpha_{*}\mathbb{Z}/l(n)_{\et}
\right)
\end{equation*}
is injective where the homomorphism is induced by (\ref{mfmtm}).
\end{cor}
\begin{proof}\upshape
By Proposition \ref{fscH}, it suffices to show that 
the homomorphisms
\begin{equation}\label{GBoS}
\HO^{n+1}_{\et}(
C,
\mathbb{Z}/l(n)
)    
\xrightarrow{\simeq}
\Gamma\left(
\spec(C),
R^{n+1}\alpha_{*}
\mathbb{Z}/l(n)_{\et}
\right)
\end{equation}
and
\begin{equation}\label{GBoN}
\HO^{n+1}_{\et}(
C/(\pi),
\mathbb{Z}/l(n)
)    
\xrightarrow{\simeq}
\Gamma\left(
\spec(C/(\pi)),
R^{n+1}\alpha_{*}
\mathbb{Z}/l(n)_{\et}
\right)
\end{equation}
are isomorphisms. Since the sequence
\begin{align*}
&\HO^{n+1}_{\Nis}\left(
C,
\tau_{\leq n}(
R\alpha_{*}\mathbb{Z}/l(n)_{\et}
)
\right)
\to
\HO^{n+1}_{\et}(
C,
\mathbb{Z}/l(n)
)
\\
\to
&\Gamma\left(
\spec(C),
R^{n+1}\alpha_{*}\mathbb{Z}/l(n)_{\et}
\right) 
\to 
\HO^{n+2}_{\Nis}\left(
C,
\tau_{\leq n}(
R\alpha_{*}\mathbb{Z}/l(n)_{\et}
)
\right)
\end{align*}
is exact, the isomorphism (\ref{GBoS}) follows 
from Proposition \ref{VaniMtM1}, Lemma \ref{inVaniTr} and Proposition \ref{VSC}.
Since we have a quasi-isomorphism
\begin{align*}
\tau_{\leq n}\left(
R\alpha_{*}\mathbb{Z}/l(n)_{\et}^{\spec(C/(\pi))}
\right)
\simeq
\mathbb{Z}/l(n)_{\Nis}^{\spec(C/(\pi))}
\end{align*}
by Proposition \ref{pBL} and \cite[Proposition 2.1]{Sak4},
the isomorphism (\ref{GBoN})
follows from  Proposition \ref{VaniNGL}.
This completes the proof.
\end{proof}
\begin{lem}\upshape\label{injpr}
Let $A$ be a henselian local ring with 
$\operatorname{dim}(A)=1$ and $l$ an integer which is prime to $\operatorname{char}(A)$.  
Suppose that the maximal ideal of $A$ is principal. Then the homomorphism
\begin{equation*}
\HO^{n+1}_{\et}(A, \mu_{l}^{\otimes n})
\to
\HO^{n+1}_{\et}(
\kappa(x),
\mu_{l}^{\otimes n}
)
\end{equation*}
is injective for any integer $n\geq 0$ and $x\in \spec(A)^{(0)}$.
\end{lem}
\begin{proof}\upshape
Consider a commutative diagram
\begin{equation*}
\xymatrix{
\HO^{n+1}_{\et}(
A,
\mu_{l}^{\otimes n}
)
\ar[r]\ar[d]
&
\HO^{n+1}_{\et}(
\kappa(x),
\mu_{l}^{\otimes n}
)
\ar@{=}[d]
\\
\HO^{n+1}_{\et}(
A/x,
\mu_{l}^{\otimes n}
)
\ar[r]
&
\HO^{n+1}_{\et}(
\kappa(x),
\mu_{l}^{\otimes n}
)
}    
\end{equation*}
where $x\in \spec(A)^{(0)}$. Since $A/x$ is a discrete valuation ring 
by \cite[p.78, Theorem 11.1]{Ma}, 
the bottom map
is injective by \cite[Theorem B.2.1 and Examples B.1.1.(2)]{C-H-K}. 
Moreover, the left map is an isomorphism 
by \cite[p.777, The proof of Proposition 2.2.b)]{Ge}.
So the upper map is injective. This completes the proof.
\end{proof}
\begin{lem}\upshape\label{nrhenij}
Let $B$ be a discrete valuation ring and $\pi$ a prime element of $B$. Let $R$ 
be the henselization of the local ring of
\begin{math}
C=B[T_{0}, \cdots, T_{N}]
/(T_{0}\cdots T_{a}-\pi T_{N}^{b})    
\end{math}
at a point $x$ of $\spec(C)$
and $l$ an integer which is prime to $\operatorname{char}(B)$.
Then the homomorphism
\begin{equation}\label{homnr}
\HO^{n}_{\et}(R, \mu_{l}^{\otimes n})
\to
\displaystyle\bigoplus_{x\in \spec(R)^{(0)}}
\HO^{n}_{\et}(
\kappa(x),
\mu_{l}^{\otimes n}
)
\end{equation}
is injective for any integer $n\geq 0$.
\end{lem}
\begin{proof}\upshape
Let $\mathfrak{p}\in \spec(B[T_{0}, \cdots, T_{N}])$ be
the image of $x\in\spec(C)$ under
$\spec(C)\to \spec(B[T_{0}, \cdots, T_{N}])$.
Then 
\begin{math}
T_{0}\cdots T_{a}-\pi T_{N}^{b}\in \mathfrak{p}.    
\end{math}

First we assume that $T_{N}\notin \mathfrak{p}$. We have a homomorphism of polynomial rings over 
$B$
\begin{equation*}
B[T_{0}, \cdots, T_{N}]
\to
B[T_{0}^{\prime}, \cdots, T_{N}^{\prime}]
\end{equation*}
which sends $T_{i}$ to $T_{i}^{\prime}$ for $i\neq 0$
and
sends $T_{0}$ to $T_{0}^{\prime}(T_{N}^{\prime})^{b}$.
Then this homomorphism induces an isomorphism
\begin{equation*}
\left(
B[T_{0}, \cdots, T_{N}]/
(T_{0}\cdots T_{a}
-\pi T_{N}^{b})
\right)_{T_{N}}
\simeq 
\left(
B[
T_{0}^{\prime},
\cdots,
T_{N}^{\prime}]/
(T^{\prime}_{0}\cdots T_{a}^{\prime}-\pi)
\right)_{T_{N}^{\prime}}
\end{equation*}
and so $R$ is the henselization of the local ring of a semistable family over $\spec(B)$.
By \cite[p.35, Proposition 5]{SM19}, the homomorphism (\ref{homnr}) is injective in the case where 
$T_{N}\notin\mathfrak{p}$.

Next we assume that $T_{N}\in \mathfrak{p}$. Then $T_{i}\in \mathfrak{p}$ for some $i$
($0\leq i\leq a$) and we may assume that $T_{0}\in \mathfrak{p}$.
So we have $(T_{0}, T_{N})\subset \mathfrak{p}$.
Since $T_{1}, \cdots, T_{a}$ are invertible elements of
\begin{equation*}
(B[T_{0}, \cdots, T_{N}]_{\mathfrak{p}})_{(T_{0}, T_{N})}
=
B[T_{0}, \cdots, T_{N}]_{(T_{0}, T_{N})},    
\end{equation*}
so $T_{1}, \cdots, T_{a}$ are also invertible elements of
\begin{equation*}
(B[T_{0}, \cdots, T_{N}]_{\mathfrak{p}}^{h})_{(T_{0}, T_{N})}^{h}   
\end{equation*}
where $D^{h}_{\mathfrak{q}}$ is the henselization of the local ring of 
a ring $D$ at a point $\mathfrak{q}\in \spec(D)$.
Hence
\begin{align*}
\left(
T_{N}, T_{0}\cdots T_{a}-\pi T_{N}^{b}
\right)
=
\left(
T_{N}, T_{0}\cdots T_{a}
\right)
=
(T_{0}, T_{N})
\end{align*}
in 
\begin{math}
(B[T_{0}, \cdots, T_{N}]_{\mathfrak{p}}^{h})_{(T_{0}, T_{N})}^{h}    
\end{math}
and the maximal ideal of 
\begin{equation*}
E=
\left(
B[T_{0}, \cdots, T_{N}]_{\mathfrak{p}}^{h}
/(T_{0}\cdots T_{a}-\pi T_{N}^{b})
\right)_{(T_{0}, T_{N})}^{h}
\end{equation*}
is principal. 
Since the residue field 
$\kappa((T_{0}, T_{N}))$
of $(T_{0}, T_{N})\in \spec(R)$
agrees with $E/(T_{0}, T_{N})$,
we have a commutative diagram
\begin{equation*}
\xymatrix{
\HO^{n+1}_{\et}(R, \mu_{l}^{\otimes n})
\ar[r]\ar[d]
&
\HO^{n+1}_{\et}(
E, 
\mu_{l}^{\otimes n}
)  \ar[d]
\\
\HO^{n+1}_{\et}(
R^{\prime},
\mu_{l}^{\otimes n}
)
\ar[r]
&
\HO^{n+1}_{\et}(
\kappa((T_{0}, T_{N})),
\mu_{l}^{\otimes n}
)
}    
\end{equation*}
where $R^{\prime}=R/(T_{0}, T_{N})$ and so $R^{\prime}$ is 
the henselization of the local ring of 
$B[T_{1}, \cdots, T_{N-1}]$.
Since the left map is an isomorphism 
by 
the isomorphism (\ref{isoAI})
and 
the bottom map is injective by \cite[p.35, Proposition 5]{SM19},
the upper map is injective. Moreover, the homomorphism
\begin{equation*}
\HO^{n+1}_{\et}(
E, 
\mu_{l}^{\otimes n}
)
\to 
\displaystyle\bigoplus_{x\in \spec(E)^{(0)}}
\HO^{n+1}_{\et}(
\kappa(x),
\mu_{l}^{\otimes n}
)
\end{equation*}
is injective by Lemma \ref{injpr}. 
Hence, the homomorphism (\ref{homnr}) is injective
in the case where $T_{N}\in \mathfrak{p}$.
This completes the proof.
\end{proof}
\begin{cor}\upshape\label{injproEx}
Let $B$ be a henselian discrete valuation ring of mixed characteristic $(0, p)$
and $\pi$ a prime element of $B$. Let
\begin{equation*}
\mathfrak{X}
=
\operatorname{Proj}\left(
B[T_{0}, \cdots, T_{N+1}]/
(T_{0}\cdots
T_{a}-\pi T_{N+1}^{a+1}
)
\right)
\end{equation*}
for $0\leq a\leq N$ and
$Y$ the closed fiber of $\mathfrak{X}$. 
Then the homomorphism
\begin{equation}\label{probnc}
\Gamma\left(
\mathfrak{X},
R^{n+1}\alpha_{*}\mu_{l}^{\otimes n}
\right)
\to
\Gamma\left(
Y,
R^{n+1}\alpha_{*}\mu_{l}^{\otimes n}
\right)
\end{equation}
is injective for any integers $n\geq 0$ and $l>0$ with
$(l, p)=1$. Here 
$\alpha: \mathfrak{X}_{\et}\to \mathfrak{X}_{\Nis}$
is the canonical map of sites.
\end{cor}
\begin{proof}\upshape
Let 
\begin{math}
j: D(T_{N+1})\to  
\mathfrak{X}
\end{math}
be an open immersion of $\mathfrak{X}$ such that
\begin{align*}
D(T_{N+1})
:=
\{
\mathfrak{p}\in \mathfrak{X}|
T_{N+1}\notin\mathfrak{p}
\}
=\spec(
B[S_{0}, \cdots, S_{N}]/
(S_{0}\cdots S_{a}-\pi)
).
\end{align*}
Then the homomorphism
\begin{equation*}
R^{n+1}\alpha_{*}\mu_{l}^{\otimes n}
\to 
j_{*}j^{*}R^{n+1}\alpha_{*}\mu_{l}^{\otimes n}
\end{equation*}
is injective by Lemma \ref{nrhenij}. So the homomorphism
\begin{align}\label{inclpta}
\Gamma\left(
\mathfrak{X},
R^{n+1}\alpha_{*}\mu_{l}^{\otimes n}
\right)
\to
\Gamma\left(
\spec(C),
R^{n+1}\alpha_{*}\mu_{l}^{\otimes n}
\right)
\end{align}
is injective where
\begin{equation*}
C=
B[S_{0}, \cdots, S_{N}]/
(S_{0}\cdots S_{a}-\pi).    
\end{equation*}
Then the composite of the homomorphism (\ref{inclpta}) and the homomorphism
\begin{align*}
\Gamma\left(
\spec(C),
R^{n+1}\alpha_{*}\mathbb{Z}/l(n)
\right)
\to
\Gamma\left(
\spec(C/(\pi)),
R^{n+1}\alpha_{*}\mathbb{Z}/l(n)
\right)
\end{align*}
agrees with the composite of the homomorphism (\ref{probnc}) and
the composition
\footnotesize
\begin{align*}
\Gamma\left(
Y, 
R^{n+1}\alpha_{*}
\mu_{l}^{\otimes n}
\right) 
\to
\Gamma\left(
\spec(C/(\pi)),
R^{n+1}\alpha_{*}
\mu_{l}^{\otimes n}
\right)
\to
\Gamma\left(
\spec(C/(\pi)),
R^{n+1}\alpha_{*}
\mathbb{Z}/l(n)
\right)
\end{align*}
\normalsize
(cf. Proposition \ref{constm}).
Hence the statement follows from Corollary \ref{inclaffmt}.
\end{proof}
\begin{rem}\upshape\label{comia}
Let the notations be the same as above and
$i: Y\to \mathfrak{X}$
the inclusion of the closed fiber. Then we have an isomorphism
\begin{equation*}
\Gamma\left(
Y,
i^{*}R^{n+1}\alpha_{*}\mu^{\otimes n}_{l}
\right)
\xrightarrow{\simeq}
\Gamma\left(
Y,
R^{n+1}\alpha_{*}\mu^{\otimes n}_{l}
\right)
\end{equation*}
by \cite[p.776, Proposition 2.2.b)]{Ge}.
\end{rem}

\section{Proper base change theorem}\label{TPrb}

\subsection{A generalization of Artin's theorem}

In this subsection, we prove a generalization of Artin's theorem (\cite[p.98, Th\'{e}or\`{e}me (3.1)]{Gr}).
For a scheme $X$, $\alpha: X_{\et}\to X_{\Nis}$ denotes the canonical map of sites.

\begin{lem}\upshape\label{1BUis}
Let $Y$ be a one dimensional scheme. Then we have an isomorphism
\begin{equation*}
\HO^{2}_{\et}(
Y, \mathbb{G}_{m}
)
\xrightarrow{\simeq}
\Gamma\left(
Y,
R^{2}\alpha_{*}\mathbb{G}_{m}
\right).
\end{equation*}
\end{lem}
\begin{proof}\upshape
By \cite[p.124, III, Proposition 4.9]{M}, we have an isomorphism
\begin{equation*}
R^{1}\alpha_{*}\mathbb{G}_{m}
\simeq 
0
\end{equation*}
and so we have a distinguished triangle
\begin{align*}
\cdots 
\to
\tau_{\leq 0}R\alpha_{*}\mathbb{G}_{m}^{Y}
\to
R\alpha_{*}\mathbb{G}_{m}^{Y}
\to
\tau_{\geq 2}R\alpha_{*}\mathbb{G}_{m}^{Y}
\to\cdots.
\end{align*}
By \cite[pp.279--280, 1.32. Theorem]{Ni}, 
we have an isomorphism
\begin{equation*}
\HO^{q}_{\Nis}(
Y, 
\tau_{\leq 0}R\alpha_{*}\mathbb{G}_{m}
)    
=0
\end{equation*}
for $q\geq 2$. Hence the statement follows.
\end{proof}
\begin{thm}\upshape(cf. Artin's theorem \cite[p.98, Th\'{e}or\`{e}me (3.1)]{Gr})\label{BrAr}
Let $B$ be a henselian excellent discrete valuation ring,
$\mathfrak{X}$ a two dimensional regular scheme which is proper and flat over 
$\spec(B)$ and $Y$ the closed fiber of $\mathfrak{X}$. 

Then the canonical map
\begin{align*}
\Gamma\left(
\mathfrak{X}, 
R^{2}\alpha_{*}\mathbb{G}_{m}
\right)
\to
\Gamma\left(
Y,
R^{2}\alpha_{*}\mathbb{G}_{m}
\right)
\end{align*}
is an isomorphism. 
\end{thm}

\begin{proof}\upshape
By the same argument as in the proof of Lemma \ref{1BUis}, we have a distinguished triangle 
\begin{align*}
\cdots 
\to
\tau_{\leq 0}R\alpha_{*}\mathbb{G}_{m}^{\mathfrak{X}}
\to
R\alpha_{*}\mathbb{G}_{m}^{\mathfrak{X}}
\to
\tau_{\geq 2}R\alpha_{*}\mathbb{G}_{m}^{\mathfrak{X}}
\to\cdots.
\end{align*}
Moreover, we have a commutative diagram
\begin{equation*}
\xymatrix{
\HO^{2}_{\et}(
\mathfrak{X},
\mathbb{G}_{m}
)
\ar[r]\ar[d]
&
\Gamma\left(
\mathfrak{X},
R^{2}\alpha_{*}
\mathbb{G}_{m}
\right)
\ar[d]
\\
\HO^{2}_{\et}(
k(\mathfrak{X}),
\mathbb{G}_{m}
)
\ar[r]
&
\Gamma\left(
k(\mathfrak{X}),
R^{2}\alpha_{*}\mathbb{G}_{m}
\right).
}    
\end{equation*}
Then the left map is injective by \cite[p.145, IV, Corollary 2.6]{M}. 
Moreover, the bottom map is an isomorphism. 
So the homomorphism
\begin{equation}\label{fBtUn}
\HO^{2}_{\et}(
\mathfrak{X},
\mathbb{G}_{m}
)    
\to
\Gamma\left(
\mathfrak{X},
R^{2}\alpha_{*}\mathbb{G}_{m}
\right)
\end{equation}
is injective. Moreover, we have
\begin{equation*}
\HO^{3}_{\Nis}(
\mathfrak{X},
\tau_{\leq 0}
R\alpha_{*}\mathbb{G}_{m}
)  
=0
\end{equation*}
by \cite[pp.279--280, 1.32. Theorem]{Ni}. So the homomorphism (\ref{fBtUn}) is an isomorphism. 
Hence the statement follows from \cite[p.98, Th\'{e}or\`{e}me (3.1)]{Gr} and Lemma \ref{1BUis}.
\end{proof}
\begin{lem}\upshape\label{PTj}
Let $B$ be a local ring and 
$X$ a scheme over $\spec(B)$. 
Let $i: Y\to X$ be the inclusion of the closed fiber of $X$
with open complement
$j: U\to X$
and 
$\mathcal{F}^{\bullet}$ a bounded below complex of 
sheaves on $X_{\et}$.
Then we have a spectral sequence
\begin{align}\label{NEsp}
E^{s, t}_{2}=
\HO^{s}_{\Nis}(
X, 
j_{!}j^{*}R^{t}\alpha_{*}\mathcal{F}^{\bullet}
)
\Rightarrow
E^{s+t}
=
\HO^{s+t}_{\et}(
X,
j_{!}j^{*}\mathcal{F}^{\bullet}
).
\end{align}
Suppose that $B$ is a henselian local ring, $X$ is proper over $\spec(B)$
and $\mathcal{H}^{t}(\mathcal{F}^{\bullet})$ are torsion sheaves
for any integer $t$. 
Then we have
\begin{align}\label{prbz}
E^{n}
=
\HO^{n}_{\et}(
X, 
j_{!}j^{*}\mathcal{F}^{\bullet}
)
=0
\end{align}
\end{lem}
\begin{proof}\upshape
By \cite[p.776, Proposition 2.2]{Ge},
we have a quasi-isomorphism
\begin{equation*}
i_{*}i^{*}R\alpha_{*}\mathcal{F}^{\bullet} 
\simeq 
R\alpha_{*}i_{*}i^{*}\mathcal{F}^{\bullet}.
\end{equation*}
Since $i_{*}$, $i^{*}$, $j_{!}$ and $j^{*}$
are exact (cf.\cite[p.76, II, Proposition 3.14 (b)]{M}), 
we have 
a distinguished triangle
\begin{align*}
\cdots\to 
j_{!}j^{*}\mathcal{G}^{\bullet}
\to
\mathcal{G}^{\bullet}
\to
i_{*}i^{*}\mathcal{G}^{\bullet}
\to\cdots
\end{align*}
for a bounded below complex $\mathcal{G}^{\bullet}$
of sheaves on $X_{\et}$.
So we have
a quasi-isomorphism
\begin{equation*}
j_{!}j^{*}R\alpha_{*}\mathcal{F}^{\bullet}
\simeq 
R\alpha_{*}j_{!}j^{*}\mathcal{F}^{\bullet}
\end{equation*}
and the spectral sequence (\ref{NEsp}).

Suppose that $B$ is a henselian local ring, $X$ is proper over $\spec(B)$
and $\mathcal{H}^{t}(\mathcal{F}^{\bullet})$ are torsion sheaves
for any integer $t$.
Then we have the equation (\ref{prbz})
by a spectral sequence
\begin{align*}
E^{s, t}_{2}
=
\HO^{s}_{\et}(
X,
j_{!}j^{*}\mathcal{H}^{t}(\mathcal{F}^{\bullet})
) 
\Rightarrow
E^{s+t}
=
\HO^{s+t}_{\et}(
X,
j_{!}j^{*}
\mathcal{F}^{\bullet}
)
\end{align*}
and \cite[p.224, VI, Corollary 2.7]{M}.
\end{proof}
\begin{lem}\upshape\label{HrG}
Let $B$ be a discrete valuation ring of mixed characteristic $(0, p)$, $X$
a semistable family over $\spec(B)$ and $Y$ the closed fiber of $X$. 
Suppose that $\operatorname{dim}(X)=2$
and $B$ contains $p$-th roots of unity. Then we have
\begin{align*}
\HO^{s}_{\Nis}(
X, R^{t}\alpha_{*}\mathfrak{T}_{1}(n)
)
=
\dfrac{\operatorname{Ker}\left(
\displaystyle
\bigoplus_{x\in X^{(s)}}\HO^{s+t}_{x}(X, \mathfrak{T}_{1}(n))
\to
\bigoplus_{x\in X^{(s+1)}}\HO^{s+t+1}_{x}(X, \mathfrak{T}_{1}(n))
\right)}
{\operatorname{Im}
\left(
\displaystyle
\bigoplus_{x\in X^{(s-1)}}\HO^{s+t-1}_{x}(X, \mathfrak{T}_{1}(n))
\to
\bigoplus_{x\in X^{(s)}}\HO^{s+t}_{x}(X, \mathfrak{T}_{1}(n))
\right)
}
\end{align*}
and
\begin{align*}
\HO^{s}_{\Nis}(
Y, 
R^{t}\alpha_{*}\lambda_{1}^{n}
)    
=
\dfrac{\operatorname{Ker}\left(
\displaystyle
\bigoplus_{y\in Y^{(s)}}\HO^{s+t}_{y}(Y, \lambda_{1}^{n}
)
\to
\bigoplus_{y\in Y^{(s+1)}}\HO^{s+t+1}_{y}(Y, \lambda_{1}^{n}
)
\right)}
{\operatorname{Im}
\left(
\displaystyle
\bigoplus_{y\in Y^{(s-1)}}\HO^{s+t-1}_{y}(Y, \lambda_{1}^{n})
\to
\bigoplus_{y\in Y^{(s)}}\HO^{s+t}_{y}(Y, \lambda_{1}^{n})
\right)
}
\end{align*}
for any non-negative integers $s$, $t$ and $n$.
\end{lem}
\begin{proof}\upshape
By the same argument as in the proof of \cite[Lemma 5.3]{Sak5},
the statement follows from Corollary \ref{2GerT} and Theorem \ref{Gtcl}.
\end{proof}
\begin{thm}\upshape\label{Nisproper}
Let $B$ be a henselian excellent discrete valuation ring of mixed characteristic 
$(0, p)$
and $\mathfrak{X}$ a semistable family and proper over $\spec(B)$.
Let $i: Y\to \mathfrak{X}$ be the inclusion of 
the closed fiber of $\mathfrak{X}$.
Suppose that
$\operatorname{dim}(\mathfrak{X})=2$
and
$B$ contains $p$-th roots of unity.
Then the homomorphism
\begin{equation}\label{Nisis}
\HO^{s}_{\Nis}(\mathfrak{X}, 
R^{t}\alpha_{*}\mathfrak{T}_{1}(n)
)
\xrightarrow{\simeq}
\HO^{s}_{\Nis}(Y, 
i^{*}R^{t}\alpha_{*}\mathfrak{T}_{1}(n)
)
\end{equation}
is an isomorphism for integers $s\geq 0$ and $t\geq 2$. 
Moreover, we have an isomorphism
\begin{equation}\label{NisUr}
\HO^{s}_{\Nis}(\mathfrak{X}, 
R^{n+1}\alpha_{*}\mathfrak{T}_{1}(n)
)
\xrightarrow{\simeq}
\HO^{s}_{\Nis}
(Y, R^{1}\alpha_{*}\lambda_{1}^{n}
)
\end{equation}
for integers $s\geq 0$ and $n\geq 1$. Thus, the sequence
\begin{align}\label{NisHM}
0\to  
\displaystyle\bigoplus_{x\in \mathfrak{X}^{(0)}}
\HO_{x}^{n+r}(\mathfrak{X}_{\et}, \mathfrak{T}_{1}(n))
\to
\displaystyle\bigoplus_{x\in \mathfrak{X}^{(1)}}
\HO_{x}^{n+r+1}(\mathfrak{X}_{\et}, \mathfrak{T}_{1}(n))
\nonumber
\\
\to
\displaystyle\bigoplus_{x\in \mathfrak{X}^{(2)}}
\HO_{x}^{n+r+2}(\mathfrak{X}_{\et}, \mathfrak{T}_{1}(n))
\to
0
\end{align}
is exact for integers $n\geq 1$ and $r\geq 2$.
\end{thm}
\begin{proof}\upshape
Assume that the isomorphism (\ref{Nisis}) holds. 
By \cite[Theorem 1.2 and Theorem 1.4]{Sak4}, 
we have an isomorphism
\begin{equation}\label{isoTla}
i^{*}R^{n+1}\alpha_{*}\mathfrak{T}_{1}(n)
\simeq
R^{1}\alpha_{*}\lambda_{1}^{n}
\end{equation}
for any integer $n\geq 0$. So the isomorphism (\ref{NisUr}) follows from
the isomorphism (\ref{Nisis}). Moreover, we have an isomorphism
\begin{equation*}
i^{*}R^{t}\alpha_{*}\mathfrak{T}_{1}(n)
\simeq 0
\end{equation*}
for $t\geq n+2$ by the analog results of 
\cite[p.69, II, Theorem 3.2]{M} and 
\cite[p.88, III, Theorem 1.15]{M} 
for the Nisnevich sheaves 
and \cite[Expos\'{e} X, Th\'{e}or\`{e}me 5.1]{SGA4}. So we
have an isomorphism
\begin{equation*}
\HO^{s}_{\Nis}(
\mathfrak{X}, 
R^{t}\alpha_{*}\mathfrak{T}_{1}(n)
)    
\simeq
0
\end{equation*}
for integers $s\geq 0$ and $t\geq n+2$ by the isomorphism (\ref{Nisis}). So the sequence (\ref{NisHM}) is exact by Lemma \ref{HrG}. 
Hence it suffices to prove the isomorphism (\ref{Nisis}).

Let $j: U\to \mathfrak{X}$ be the inclusion of the generic fiber of $\mathfrak{X}$. 
Since $j^{*}$ is exact by \cite[p.76, II, Proposition 3.14 (b)]{M}, 
we have isomorphisms
\begin{equation*}
j^{*}R^{t}\alpha_{*}\mathfrak{T}_{1}(n)
\simeq
R^{t}(j^{*}\alpha_{*})\mathfrak{T}_{1}(n)
\simeq
R^{t}(\alpha_{*}j^{*})\mathfrak{T}_{1}(n)
\simeq
R^{t}\alpha_{*}\mu_{p}^{\otimes n}
\end{equation*}
by \cite[p.776, Proposition 2.2 a)]{Ge}. So the sequence
\begin{equation*}
0\to    
j_{!}R^{t}\alpha_{*}\mu_{p}^{\otimes n}
\to
R^{t}\alpha_{*}\mathfrak{T}_{1}(n)
\to
i_{*}i^{*}
R^{t}\alpha_{*}\mathfrak{T}_{1}(n)
\to 0
\end{equation*}
is exact.
Hence the isomorphism (\ref{Nisis})
follows from the isomorphism
\begin{equation}\label{VanNic}
\HO^{s}_{\Nis}
(
\mathfrak{X}, 
j_{!}R^{t}\alpha_{*}\mu_{p}
)
\simeq 0    
\end{equation}
for any integers $s\geq 0$ and $t\geq 2$. Hence it suffices to prove the isomorphism (\ref{VanNic}).

By \cite[p.279--280, 1.32. Theorem]{Ni}, we have an isomorphism (\ref{VanNic}) for $s\geq 3$.

By \cite[p.31, Theorem 1.2]{SM} (or \cite[Theorem 1.4]{Sak4}), we have an exact sequence
\begin{align*}
0\to   
\HO^{1}_{\et}(
\kappa(x),
\lambda_{1}^{n}
)
\to
\HO^{n+1}_{\et}(
k(\mathcal{O}_{\mathfrak{X}, x}^{h}),
\mathfrak{T}_{1}(n)
)
\to
\HO^{n+2}_{x}(
\mathfrak{X},
\mathfrak{T}_{1}(n)
)
\end{align*}
for $x\in \mathfrak{X}^{(1)}\cap Y$.
Here $\mathcal{O}_{\mathfrak{X}, x}^{h}$ is the henselization of
the local ring $\mathcal{O}_{\mathfrak{X}, x}$ of $\mathfrak{X}$ at $x\in \mathfrak{X}^{(1)}\cap Y$
and $k(\mathcal{O}^{h}_{\mathfrak{X}, x})$ is the fraction field of
$\mathcal{O}^{h}_{\mathfrak{X}, x}$. So the homomorphism
\begin{equation*}
\HO^{0}_{\Nis}(
\mathfrak{X}, 
R^{n+1}\alpha_{*}\mathfrak{T}_{1}(n)
)    
\to
\HO^{0}_{\Nis}(
Y,
R^{1}\alpha_{*}\lambda_{1}^{n}
)
\end{equation*}
is injective for $n\geq 1$ by Lemma \ref{HrG} and
\cite[Theorem 1.7]{Sak4}.
Hence we have an isomorphism (\ref{VanNic}) for $s=0$ and
$t\geq 2$ by the isomorphism (\ref{isoTla}).

Moreover,
we have an isomorphism
\begin{equation}\label{pbc}
\HO_{\et}^{s}(\mathfrak{X}, j_{!}j^{*}
\mathfrak{T}_{1}(n)
)
\simeq
0
\end{equation}
for any integer $s$ by Lemma \ref{PTj}. Since we have an isomorphism (\ref{VanNic}) for 
$s\neq 1, 2$ and $t\geq 2$ by the above, we have an isomorphism (\ref{VanNic})
for $s=1, 2$ and $t\geq 2$ by the spectral sequence (\ref{NEsp}) and the isomorphism (\ref{pbc}).
This completes the proof.

\end{proof}
\begin{rem}\upshape\label{Expgl}
Let the notations be the same as in Theorem \ref{Nisproper}.
\begin{enumerate}
\item 
Let $l$ be an integer which is prime to $p$. Then the natural map
\begin{align*}
\HO^{n+1}_{\et}(
k(\mathfrak{X}),
\mu_{l}^{\otimes n}
)    
\to
\displaystyle
\bigoplus_{
\mathfrak{p}\in \mathfrak{X}^{(1)}\setminus Y^{(0)}}
\HO^{n}_{\et}(
\kappa(\mathfrak{p}),
\mu_{l}^{\otimes (n-1)}
)
\oplus
\displaystyle
\bigoplus_{
\mathfrak{p}\in Y^{(0)}}
\HO^{n+1}_{\et}(
k(\tilde{\mathcal{O}_{\mathfrak{X}, \mathfrak{p}}}),
\mu_{l}^{\otimes n}
)
\end{align*}
is injective (cf. \cite[The remark below Theorem 1.7]{Sak4}). 
Here $\tilde{\mathcal{O}_{\mathfrak{X}, \mathfrak{p}}}$
is the henselization of a local ring
$\mathcal{O}_{\mathfrak{X}, \mathfrak{p}}$
of $\mathfrak{p}$ in $\mathfrak{X}$.
So
we have an isomorphism
\begin{equation*}
\HO^{s}_{\Nis}(
\mathfrak{X},
R^{t}\alpha_{*}\mu_{l}^{\otimes n}
)    
\xrightarrow{\simeq}
\HO^{s}_{\Nis}(Y,
i^{*}R^{t}
\alpha_{*}\mu_{l}^{\otimes n}
)
\end{equation*}
for integers $s\geq 0$ and $t\geq 2$ by
the same argument as in the proof of Theorem \ref{Nisproper}. 
Since we have quasi-isomorphisms
\begin{equation*}
\tau_{\leq n+1}\left(
\mathbb{Z}/p^{r}(n)^{\mathfrak{X}}_{\et}
\right)
\simeq 
\mathfrak{T}_{r}(n)
\end{equation*}
and
\begin{equation*}
\tau_{\leq n+1}\left(
\mathbb{Z}/l(n)^{\mathfrak{X}}_{\et}
\right)
\simeq 
\mu_{l}^{\otimes n}
\end{equation*}
for an integer with $(l, p)=1$
by (\cite[p.209, Remark 7.2]{SaR}, \cite[Proposition 4.5]{Sak4}) and \cite[Remark 4.7]{Sak4},
we have an isomorphism
\begin{equation}\label{proetmt}
\HO^{s}_{\Nis}(
\mathfrak{X},
R^{n+1}\alpha_{*}\mathbb{Z}/m(n)_{\et}
)    
\xrightarrow{\simeq}
\HO^{s}_{\Nis}(Y,
i^{*}R^{n+1}\alpha_{*}\mathbb{Z}/m(n)_{\et}
)   
\end{equation}
for integers $n\geq 1$, $s\geq 0$ and a prime number $m$
in the case where $B$ contains $m$-th roots of unity.
\item By \cite[Proposition 4.2]{Sak4} and \cite[p.35, Proposition 5]{SM19}, 
there exists an exact sequence
\begin{align*}
0\to 
R^{n+1}\alpha_{*}\mathbb{Z}/m(n)^{\mathfrak{X}}_{\et}
\to
R^{n+2}\alpha_{*}\mathbb{Z}(n)^{\mathfrak{X}}_{\et}
\xrightarrow{\times m}
R^{n+2}\alpha_{*}\mathbb{Z}(n)^{\mathfrak{X}}_{\et}
\end{align*}
for any integers $n\geq 0$ and $m>0$.
Moreover, we have a quasi-isomorphism
\begin{equation*}
\mathbb{Z}(1)_{\et}^{\mathfrak{X}}
\simeq
\mathbb{G}_{m}[-1]
\end{equation*}
by \cite[Lemma 11.2]{L}. 
Thus, we have an isomorphism
\begin{equation*}
i^{*}R^{2}\alpha_{*}\mathbb{G}_{m}^{\mathfrak{X}}
\simeq
R^{2}\alpha_{*}\mathbb{G}_{m}^{Y}
\end{equation*}
by
\cite[p.776, Proposition 2.2.b)]{Ge} and
\cite[p.148, IV, Corollary 2.13]{M}.
So we can regard the isomorphism (\ref{proetmt}) for $s=0$ and $n=1$
as a generalization of Theorem \ref{BrAr}.
\end{enumerate}
\end{rem}
\subsection{Questions}

In this subsection, we raise questions (see Question \ref{Qpr} and Question \ref{QG})
which relate to Kato conjecture (cf. \cite{K}, \cite{K-S}).

\begin{lem}\upshape\label{vnc}
Let $X$ be an integral scheme and $\mathcal{F}$ 
a constant sheaf on $X_{\Nis}$.
Then we have
\begin{equation*}
\HO^{q}_{\Nis}(
X, \mathcal{F}
) 
=0
\end{equation*}
for $q>0$.
\end{lem}
\begin{proof}\upshape
Since $\mathcal{F}$ is a constant sheaf, we have an isomorphism
\begin{equation*}
\mathcal{F}
\simeq
(i_{\eta})_{*}(i_{\eta})^{*}\mathcal{F}
\end{equation*}
where $i_{\eta}: \spec(\kappa(\eta))\to X$
is the generic point. Since 
$(i_{\eta})^{*}\mathcal{F}$ is a flabby sheaf, 
the statement follows from
\cite[p.89, III, Lemma 1.19]{M}.    
\end{proof}
\begin{lem}\upshape\label{smdf}
Let $X$ be a smooth scheme over the spectrum of a regular ring $B$ with 
$\operatorname{dim}(B)\leq 1$. 
Then 
we have
\begin{equation}\label{eqeg}
\HO^{1}_{\et}(
X,
\mathbb{Q}/\mathbb{Z}
)    
=
\Gamma\left(
X,
R^{1}\alpha_{*}\mathbb{Q}/\mathbb{Z}
\right)
\end{equation}
where $\alpha: X_{\et}\to X_{\Nis}$ is 
the canonical map of sites.
Moreover, the sequence  
\begin{equation}\label{exdf}
0\to 
\HO^{1}_{\et}(X, \mathbb{Q}/\mathbb{Z})
\to
\bigoplus_{x\in X^{(0)}}
\HO^{1}_{x}(
X_{\et},
\mathbb{Q}/\mathbb{Z}
)
\to
\bigoplus_{x\in X^{(1)}}
\HO^{2}_{x}(
X_{\et},
\mathbb{Q}/\mathbb{Z}
)
\end{equation}
is exact.
\end{lem}
\begin{proof}\upshape
Since we have a distinguished triangle
\begin{equation*}
\cdots\to  
\mathbb{Q}/\mathbb{Z}
\to 
R\alpha_{*}\mathbb{Q}/\mathbb{Z}
\to
\tau_{\geq 1}\alpha_{*}\mathbb{Q}/\mathbb{Z}
\to\cdots,
\end{equation*}
we have the isomorphism (\ref{eqeg}) by Lemma \ref{vnc}.
So the sequence (\ref{exdf}) is exact by \cite[p.30, Theorem 1.1]{SM}. 
This completes the proof.
\end{proof}
\begin{prop}\upshape
Let $\mathfrak{X}$ be a proper and smooth scheme
over the spectrum of a henselian discrete valuation ring
and $i: Y\to \mathfrak{X}$
the inclusion of the closed fiber. Then the homomorphism
\begin{equation*}
\Gamma\left(
\mathfrak{X},
R^{1}\alpha_{*}\mathbb{Q}/\mathbb{Z}
\right)
\to
\Gamma\left(
Y,
i^{*}R^{1}\alpha_{*}\mathbb{Q}/\mathbb{Z}
\right)
\end{equation*}
is an isomorphism.
\end{prop}
\begin{proof}
The statement follows from the proper base change theorem (cf. \cite[p.224, VI, Corollary 2.7]{M}) and Lemma \ref{smdf}.    
\end{proof}
\begin{Qu}\upshape\label{Qpr}
Let $\mathfrak{X}$ be a regular scheme which is proper and flat over the spectrum of   
a henselian discrete valuation ring $B$ and 
$i: Y\to \mathfrak{X}$ the closed fiber of $\mathfrak{X}$. Then
\begin{description}
\item{(a)} When is the homomorphism
\begin{equation}\label{prbm}
\Gamma\left(
\mathfrak{X},
R^{n+1}\alpha_{*}\mathbb{Q}/\mathbb{Z}(n)_{\et}
\right)
\to
\Gamma\left(
Y,
i^{*}R^{n+1}\alpha_{*}\mathbb{Q}/\mathbb{Z}(n)_{\et}
\right)
\end{equation}
an isomorphism ?
\item{(b)} If $\mathfrak{X}$ is smooth over $\spec(B)$, then
the homomorphism (\ref{prbm}) is an isomorphism ?
\end{description}
\end{Qu}
\begin{rem}\upshape
Let $\mathcal{O}_{K}$ be the ring of integers of a $p$-adic field $K$,
$\mathfrak{X}$ a $d$-dimensional regular scheme which is proper and smooth over
$\spec(\mathcal{O}_{K})$ and $i: Y\to \mathfrak{X}$ the inclusion of the closed fiber of $\mathfrak{X}$. 
Then
\begin{equation*}
i^{*}R^{d+1}\alpha_{*}\mathbb{Q}/\mathbb{Z}(d)^{\mathfrak{X}}_{\et}
=
R^{d+1}\alpha_{*}\mathbb{Q}/\mathbb{Z}(d)^{Y}_{\et}=0   
\end{equation*}
and so we have
\begin{equation*}
\Gamma\left(
Y,
i^{*}R^{d+1}\alpha_{*}
\mathbb{Q}/\mathbb{Z}(d)_{\et}
\right)
=0
\end{equation*}
where 
$\alpha: \mathfrak{X}_{\et}\to \mathfrak{X}_{\Nis}$
is the canonical map of sites and
$\alpha_{*}$ is the forgetful functor. 
On the other hand, if Kato conjecture holds 
(cf.\cite[p.125, Conjecture 0.2]{K-S}, \cite[p.125, Theorem 0.4]{K-S}),
then
\begin{equation*}
\Gamma\left(
\mathfrak{X},
R^{d+1}\alpha_{*}
\mathbb{Q}/\mathbb{Z}(d)_{\et}
\right)
=0
\end{equation*}
by \cite[p.51, Theorem 4.6]{SM}.
\end{rem}

For a regular scheme of finite type over
the ring of integers of a number field,
we have the following conjectures:

\begin{con}\upshape(\cite[Conjecture 2.1]{Sai})\label{ConjF}
For a regular scheme $X$ of finite type over  
$\mathbb{F}_{p}$ or $\mathbb{Z}$,
$\HO_{\Zar}^{q}(X, \mathbb{Z}(r))$
is finitely generated.
\end{con}
\begin{con}\upshape(Lichtenbaum) (cf. \cite[p.1625, Conjecture 4.12]{Ge19})\label{ConjL}
If $X$ is regular and proper over the ring of integers of a number field, then the 
groups 
$\HO^{i}_{\et}(X, \mathbb{Z}(n))$
are finitely generated for $i\leq 2n$,
finite for $i=2n+1$,
and of cofinite type for $i\geq 2n+2$.
\end{con}

A consequence of Conjectures \ref{ConjF} and \ref{ConjL} will be
observed later.

\begin{lem}\upshape\label{Asvan}
Let $X$ be a regular scheme 
which is an essentially of finite type over
the spectrum of a Dedekind domain. If
$\mathcal{H}^{n+1}(\mathbb{Z}/m(n)_{\Nis}^{X})=0$
for any positive integer $m$, 
then the sequence
\begin{align*}
0\to   
R^{n+1}\alpha_{*}\mathbb{Z}/m(n)_{\et}^{X}
\to
R^{n+1}\alpha_{*}\mathbb{Q}/\mathbb{Z}(n)_{\et}^{X}
\xrightarrow{\times m}
R^{n+1}\alpha_{*}\mathbb{Q}/\mathbb{Z}(n)_{\et}^{X}
\end{align*}
is exact.
\end{lem}
\begin{proof}\upshape
Since
\begin{equation*}
\mathbb{Q}/\mathbb{Z}
=
\displaystyle\lim_{\xrightarrow[m]{}}
\frac{1}{m}\mathbb{Z}/\mathbb{Z}
=
\displaystyle\lim_{\xrightarrow[m]{}}
\mathbb{Z}/m\mathbb{Z}
\end{equation*}
where $m$ runs over all positive integers,
we have an exact sequence
\begin{equation*}
0\to 
\mathbb{Z}/m\mathbb{Z}
\to
\mathbb{Q}/\mathbb{Z}
\xrightarrow{\times m}
\mathbb{Q}/\mathbb{Z}
\to 
0
\end{equation*}
and a distinguished triangle
\begin{equation*}
\cdots\to 
\mathbb{Z}/m(n)^{X}_{\Nis}
\to
\mathbb{Q}/\mathbb{Z}(n)^{X}_{\Nis}
\xrightarrow{\times m}
\mathbb{Q}/\mathbb{Z}(n)^{X}_{\Nis}
\to 
\cdots
\end{equation*}
for a positive integer $m$. So we have an exact sequence
\begin{equation*}
\mathcal{H}^{n}(\mathbb{Z}/m(n)_{\Nis}^{X})
\to 
\mathcal{H}^{n}(\mathbb{Q}/\mathbb{Z}(n)_{\Nis}^{X})
\xrightarrow{\times m}
\mathcal{H}^{n}(\mathbb{Q}/\mathbb{Z}(n)_{\Nis}^{X})
\to 
0
\end{equation*}
by the assumption. Moreover, we have an isomorphism
\begin{equation*}
\mathcal{H}^{n}(
\mathbb{Q}/\mathbb{Z}(n)_{\Nis}^{X}
)
\simeq 
R^{n}\alpha_{*}
\mathbb{Q}/\mathbb{Z}(n)_{\Nis}^{X}
\end{equation*}
by Proposition \ref{SBL}. Hence the statement follows.
\end{proof}
\begin{rem}\upshape\label{ReSem}
In the case where $X$ is a semistable family over
the spectrum of a Dedekind domain, we have
\begin{equation*}
\mathcal{H}^{n+1}(
\mathbb{Z}/m(n)_{\Nis}^{X}
)    
=0
\end{equation*}
for any positive integer $m$ (cf. \cite[Proposition 4.5]{Sak4}, \cite[Remark 4.7]{Sak4}).
If $X$ is smooth over the spectrum of a Dedekind domain, we have
\begin{align*}
\mathcal{H}^{s}(
\mathbb{Z}(n)^{X}_{\Nis}
)
=
\mathcal{H}^{s}(
\mathbb{Z}/m(n)^{X}_{\Nis}
)
=0
\end{align*}
for $s\geq n+1$
and any positive integer $m$ by \cite[p.786, Corollary 4.4]{Ge}.
\end{rem}

We have the following as an application of Conjectures \ref{ConjF}
and \ref{ConjL}.

\begin{prop}\upshape\label{Appconj}
Let $\mathfrak{X}$ be a semistable family over $\spec(\mathbb{Z})$ and $n$ a positive integer.
Assume that Conjecture \ref{ConjF} and Conjecture \ref{ConjL} hold. Moreover, assume that 
\begin{equation*}
\mathcal{H}^{n+2}(\mathbb{Z}(n)_{\Nis}^{\mathfrak{X}})
=
\mathcal{H}^{n+2}(\mathbb{Z}/m(n)_{\Nis}^{\mathfrak{X}})
=0
\end{equation*}
for any positive integer $m$.
Then 
\begin{equation}\label{QZu}
\Gamma\left(
\mathfrak{X},
R\alpha_{*}^{n+1}
\mathbb{Q}/\mathbb{Z}(n)_{\et}
\right)
=
\Gamma\left(
\mathfrak{X},
R\alpha_{*}^{n+2}
\mathbb{Z}(n)_{\et}
\right)
\end{equation}
is finite and 
\begin{align}\label{glger}
\Gamma\left(
\mathfrak{X},
R\alpha_{*}^{n+1}
\mathbb{Q}/\mathbb{Z}(n)_{\et}
\right)
=
\operatorname{Ker}
\left(
\HO^{n+1}_{\et}(
k(\mathfrak{X}),
\mathbb{Q}/\mathbb{Z}(n)
)
\to
\bigoplus_{x\in \mathfrak{X}^{(1)}}
\HO^{n+2}_{x}(
\mathfrak{X}_{\et},
\mathbb{Q}/\mathbb{Z}(n)
)
\right).
\end{align}
\end{prop}
\begin{proof}
First we prove the isomorphism (\ref{QZu}). 
Let $i_{\eta}: \spec(\kappa(\eta))\to \mathfrak{X}$ be the generic point.  
Consider a commutative diagram
\begin{equation}\label{comNqz}
\xymatrix{
R^{n+1}\alpha_{*}
\mathbb{Q}/\mathbb{Z}(n)_{\et}
\ar[r]\ar[d]
&
R^{n+2}\alpha_{*}
\mathbb{Z}(n)_{\et}
\ar[d]
\\
R^{n+1}\alpha_{*}
R(i_{\eta})_{*}
\mathbb{Q}/\mathbb{Z}(n)_{\et}
\ar[r]
&
R^{n+2}\alpha_{*}
R(i_{\eta})_{*}
\mathbb{Z}(n)_{\et}.
}
\end{equation}
Then the stark of 
\begin{math}
R^{n+1}\alpha_{*}\mathbb{Q}/\mathbb{Z}(n)_{\et}    
\end{math}
at $x\in \mathfrak{X}$ is
\begin{math}
\HO^{n+1}_{\et}(\mathcal{O}_{\mathfrak{X}, x}^{h}, \mathbb{Q}/\mathbb{Z}(n))
\end{math}
and the stark of
\begin{math}
R^{n+1}\alpha_{*}
R(i_{\eta})_{*}\mathbb{Q}/\mathbb{Z}(n)_{\et}
\end{math}
at $x\in \mathfrak{X}$ is
\begin{math}
\HO^{n+1}_{\et}(
k(\mathcal{O}^{h}_{\mathfrak{X}, x}),
\mathbb{Z}(n)
).    
\end{math}
Here $\mathcal{O}_{\mathfrak{X}, x}^{h}$ is the henselization of
the local ring 
$\mathcal{O}_{\mathfrak{X}, x}$ 
of $\mathfrak{X}$ at $x\in \mathfrak{X}$
and $k(\mathcal{O}^{h}_{\mathfrak{X}, x})$ is the fraction field
of $\mathcal{O}_{\mathfrak{X}, x}^{h}$.
So the left map is injective by \cite[Proposition 4.2]{Sak4} and \cite[p.35, Proposition 5]{SM19}.
Since
\begin{equation*}
\HO^{s}_{\et}(
k(\mathcal{O}^{h}_{\mathfrak{X}, x}),
\mathbb{Q}(n)
)    
=
\HO^{s}_{\Zar}(
k(\mathcal{O}^{h}_{\mathfrak{X}, x}),
\mathbb{Q}(n)
) 
=0
\end{equation*}
for $s\geq n+1$
by \cite[p.781, Proposition 3.6]{Ge},
the bottom map is an isomorphism and so the upper map is injective.
Since we have
\begin{equation*}
\HO^{s}_{\Zar}(
\mathcal{O}^{h}_{\mathfrak{X}, x},
\mathbb{Q}/\mathbb{Z}(n)
)    
=
\lim_{\substack{
\to \\ m}}
\HO^{s}_{\Zar}(
\mathcal{O}^{h}_{\mathfrak{X}, x},
\mathbb{Z}/m(n)
)    
\end{equation*}
and the sequence
\begin{equation*}
\HO^{s}_{\Zar}(
\mathcal{O}^{h}_{\mathfrak{X}, x},
\mathbb{Z}(n)
)
\to
\HO^{s}_{\Zar}(
\mathcal{O}^{h}_{\mathfrak{X}, x},
\mathbb{Q}(n)
)
\to
\HO^{s}_{\Zar}(
\mathcal{O}^{h}_{\mathfrak{X}, x},
\mathbb{Q}/\mathbb{Z}(n)
)
\end{equation*}
is exact for any integer $s$, we have 
\begin{equation*}
R^{n+2}\alpha_{*}\mathbb{Q}(n)_{\et}
=
\mathcal{H}^{n+2}(
\mathbb{Q}(n)_{\Nis}
)
=0
\end{equation*}
by the assumption. Hence the upper map in the commutative diagram (\ref{comNqz})
is surjective and so is an isomorphism.
Therefore we have the isomorphism (\ref{QZu}).

Next we prove that 
\begin{math}
\Gamma\left(
\mathfrak{X},
R^{n+2}\alpha_{*}
\mathbb{Z}(n)_{\et}
\right)
\end{math}
is finite.
Since we have a quasi-isomorphism
\begin{equation}\label{BLN12}
\tau_{\leq n+1}R\alpha_{*}\mathbb{Z}(n)_{\et}
\simeq 
\tau_{\leq n+2}\mathbb{Z}(n)_{\Nis}
\end{equation}
by Proposition \ref{SBL} and the assumption, the homomorphism
\begin{equation*}
\HO^{n+3}_{\Nis}(
\mathfrak{X},
\tau_{\leq n+1}R\alpha_{*}
\mathbb{Z}(n)
)
\to
\HO^{n+3}_{\Nis}(
\mathfrak{X},
\mathbb{Z}(n)
)
\end{equation*}
is injective and we have
\begin{equation*}
\HO^{n+3}_{\Zar}(
\mathfrak{X},
\mathbb{Z}(n)
)
\simeq
\HO^{n+3}_{\Nis}(
\mathfrak{X},
\mathbb{Z}(n)
)
\end{equation*}
by \cite[p.781, Proposition 3.6]{Ge}.
Moreover, the sequence
\begin{align*}
\HO^{n+2}_{\et}(
\mathfrak{X},
\mathbb{Z}(n)
)    
\to
\Gamma\left(
\mathfrak{X},
R\alpha_{*}^{n+2}
\mathbb{Z}(n)_{\et}
\right)
\to
\HO^{n+3}_{\Nis}(
\mathfrak{X},
\tau_{\leq n+1}R\alpha_{*}
\mathbb{Z}(n)_{\et}
)
\end{align*}
is exact. So 
\begin{math}
\Gamma\left(\mathfrak{X}, R^{n+2}\alpha_{*}\mathbb{Z}(n)_{\et}\right)
\end{math}
is finitely generated 
over $\mathbb{Z}$ by Conjecture \ref{ConjF} and Conjecture \ref{ConjL}. 
Moreover, 
\begin{math}
\Gamma\left(\mathfrak{X}, R^{n+2}\alpha_{*}\mathbb{Z}(n)_{\et}\right)
\end{math}
is a torsion group by the isomorphism (\ref{QZu}). Hence 
\begin{math}
\Gamma\left(\mathfrak{X}, R^{n+2}\alpha_{*}\mathbb{Z}(n)_{\et}\right)
\end{math}
is finite.

Finally, we prove the isomorphism (\ref{glger}).
By Lemma \ref{Asvan} and Remark \ref{ReSem}, it suffices to show that the sequence
\begin{align}\label{lgerN}
0
&\to
\HO^{n+1}_{\et}(R, \mathbb{Z}/m(n))
\to
\displaystyle\bigoplus_{x\in \spec(R)^{(0)}}
\HO^{n+1}_{x}(
R_{\et},
\mathbb{Z}/m(n)
) \nonumber
\\
&\to
\displaystyle\bigoplus_{x\in \spec(R)^{(1)}}
\HO^{n+2}_{x}(
R_{\et},
\mathbb{Z}/m(n)
)
\end{align}
is exact 
for a prime power $m$
where $R$ is the henselization of the local ring 
$\mathcal{O}_{\mathfrak{X}, x}$ of $\mathfrak{X}$ at a point $x\in \mathfrak{X}$.
Moreover, it suffices to show the exactness of (\ref{lgerN}) in the case where $\operatorname{char}(R)=(0, p)$.
Note that we have a quasi-isomorphism
\begin{equation*}
\tau_{\leq n+2}\left(
\mathbb{Z}/m(n)_{\et}^{\spec(R)}
\right)
\simeq 
\begin{cases}
\mathfrak{T}_{r}(n)
&
\text{for $m=p^{r}$}
\\
\mu_{m}^{\otimes n}
&
\text{for $(m, p)=1$}
\end{cases}
\end{equation*}
by \cite[p.209, Remark 7.2]{SaR}, \cite[Remark 4.7]{Sak4} and the assumption.

We prove the exactness of (\ref{lgerN}) by induction on $\operatorname{dim}(R)$. 

Suppose that $\operatorname{dim}(R)=1$. Then the sequence (\ref{lgerN}) is exact by
\cite[p.30, Theorem 1.1]{SM}.

Assume that the sequence (\ref{lgerN}) is exact in the case where 
$\operatorname{dim}(R)\leq s$. Suppose that $\operatorname{dim}(R)=s+1$.
Let $Z$ be an irreducible component of the closed fiber of $\spec(R)$.
Then we have a spectral sequence
\begin{equation*}
E^{u, v}_{1}
=
\displaystyle
\bigoplus_{x\in \spec(R)^{(u)}\cap Z}
\HO^{u+v}_{x}(
R_{\et},
\mathbb{Z}/m(n)
)
\Rightarrow
E^{u+v}
=
\HO^{u+v}_{Z}(
R_{\et},
\mathbb{Z}/m(n)
)
\end{equation*}
(cf. \cite[The proof of Proposition 3.8]{Sak4}). Since we have
\begin{equation*}
\HO^{n+2}_{x}(
R_{\et}, 
\mathbb{Z}/m(n)
) 
\simeq 
\HO^{n+2-2u}_{\et}
(
\kappa(x),
\mathbb{Z}/m(n-u)
)
\end{equation*}
for 
$x\in\spec(R)^{(u)}\cap Z$ 
by \cite[p.540, Theorem 4.4.7]{SaP} and \cite{G},
we have
\begin{equation*}
E^{u, v}_{\infty}=E^{u, v}_{2}=0    
\end{equation*}
for $u+v=n+2$ and $u\geq 2$ by
\cite[p.600, Theorem 4.1]{Sh} and \cite[Corollary 2.2.2]{C-H-K}. 
So the homomorphism
\begin{equation*}
\HO^{n+2}_{Z}(
R_{\et},
\mathbb{Z}/m(n)
)
\to 
\displaystyle
\bigoplus_{x\in \spec(R)^{(1)}\cap Z}
\HO^{n+2}_{x}(
R_{\et},
\mathbb{Z}/m(n)
)
\end{equation*}
is injective.
Hence the sequence
\begin{align}\label{preex}
0\to  
\HO^{n+1}_{\et}(
R,
\mathbb{Z}/m(n)
)
\to
\HO^{n+1}_{\et}(
U,
\mathbb{Z}/m(n)
)
\to 
\displaystyle
\bigoplus_{x\in \spec(R)^{(1)}\cap Z}
\HO^{n+2}_{x}(
R_{\et},
\mathbb{Z}/m(n)
)
\end{align}
is exact 
by \cite[Proposition 4.2]{Sak4}
and \cite[p.35, Proposition 5]{SM19}
where $U=\spec(R)\setminus Z$. 
By \cite[Proposition 4.5]{Sak4}, 
\cite[p.37, Corollary 7]{SM19}
and the assumption, 
we have an isomorphism
\begin{equation*}
\HO^{s}_{\Zar}(R, \mathbb{Z}/m(n))=0    
\end{equation*}
for $n+1\leq s\leq n+2$.
Since $Z$ is the spectrum of a regular local ring of
positive characteristic, 
we have an  isomorphism
\begin{equation*}
\HO^{s}_{\Zar}(Z,
\mathbb{Z}/m(n)
)
=0
\end{equation*}
for $s \geq n+1$.
So we have an isomorphism
\begin{equation*}
\HO^{s}_{\Zar}(U, \mathbb{Z}/m(n))
=0
\end{equation*}
for $n+1\leq s\leq n+2$. Hence
we have an isomorphism
\begin{equation}\label{etagG}
\HO^{n+1}_{\et}(
U, 
\mathbb{Z}/m(n)
)    
=
\Gamma\left(
U,
R^{n+1}\alpha_{*}
\mathbb{Z}/m(n)_{\et}
\right)
\end{equation}
by the quasi-isomorphism (\ref{BLN12}). 
Since $\operatorname{dim}(U)=s$, we have
\begin{align}\label{GagK}
&\Gamma\left(
U,
R^{n+1}\alpha_{*}
\mathbb{Z}/m(n)_{\et}
\right)  \nonumber  \\
=
&\operatorname{Ker}\left(
\displaystyle\bigoplus_{x\in U^{(0)}}
\HO^{n+1}_{x}(
U_{\et},
\mathbb{Z}/m(n)
)
\to
\displaystyle\bigoplus_{x\in U^{(1)}}
\HO^{n+2}_{x}(
U_{\et},
\mathbb{Z}/m(n)
)
\right)
\end{align}
by the assumption of induction. Hence the sequence (\ref{lgerN}) is exact 
by (\ref{preex}), (\ref{etagG}) and (\ref{GagK}).
This completes the proof.

\end{proof}
\begin{prop}\upshape\label{hass1}
Let $X$ be a smooth scheme over $\spec(\mathbb{Q})$. Then the canonical map
\begin{equation}\label{lgf1}
\HO^{1}_{\et}(
k(X),
\mathbb{Q}/\mathbb{Z}
)    
\to 
\prod_{x\in X^{(1)}}
\HO^{1}_{\et}(
k(
\mathcal{O}_{X, x}^{h}),
\mathbb{Q}/\mathbb{Z}
) 
\end{equation}
is injective where $\mathcal{O}_{X, x}^{h}$
is the henselization of a local ring $\mathcal{O}_{X, x}$.
\end{prop}
\begin{proof}\upshape
Since we have an isomorphism
\begin{equation*}
\HO^{1}_{\et}(
\mathcal{O}^{h}_{X, x},
\mathbb{Q}/\mathbb{Z}
)
\simeq 
\HO^{1}_{\et}(
\kappa(x),
\mathbb{Q}/\mathbb{Z}
),
\end{equation*}
the kernel of the canonical map (\ref{lgf1}) agrees with the kernel of
\begin{equation*}
\HO^{1}_{\et}(X, \mathbb{Q}/\mathbb{Z})
\to
\prod_{x\in X^{(1)}}
\HO^{1}_{\et}(
\kappa(x),
\mathbb{Q}/\mathbb{Z}
)
\end{equation*}
by Lemma \ref{smdf}. Since $\mathbb{Q}$ is Hilbertian, 
the statement follows from \cite[p.53, Remark 3]{W}.
\end{proof}

By Proposition \ref{Appconj} and Proposition \ref{hass1}, 
we are able to raise the following questions:

\begin{Qu}\upshape\label{QG}
Let $\mathfrak{X}$ be a regular scheme which is proper and flat
over $\spec(\mathbb{Z})$. Then
\begin{description}
\item{(a)} Is the kernel of the canonical map
\begin{equation}\label{glgnh}
\HO^{n+1}_{\et}(
k(\mathfrak{X}),
\mathbb{Q}/\mathbb{Z}(n)
)    
\to
\prod_{x\in \mathfrak{X}^{(1)}}
\HO^{n+1}_{\et}(
k(\mathcal{O}_{\mathfrak{X}, x}^{h}),
\mathbb{Q}/\mathbb{Z}(n)
)
\end{equation}
finite ?
\item{(b)} When is the kernel of the canonical map (\ref{glgnh}) trivial ?
\end{description}
\end{Qu}

Question \ref{QG} relates to Kato conjecture as follows:

\begin{rem}\upshape
Let $\mathfrak{X}$ be a $d$-dimensional regular scheme which 
is proper and flat over 
the spectrum of the ring of integers of a number field $K$. 
Assume that $K$ is totally imaginary. 
Let $\mathcal{O}_{\mathfrak{X}, x}^{h}$ be the henselization of
the local ring $\mathcal{O}_{\mathfrak{X}, x}$ of $\mathfrak{X}$ at 
a point $x\in\mathfrak{X}^{(1)}$.
Then we have isomorphisms
\begin{align*}
\HO^{d+1}_{\et}(
k(\mathcal{O}_{\mathfrak{X}, x}^{h}),
\mathbb{Q}/\mathbb{Z}(d)
)    
\simeq 
\HO^{d+2}_{x}(
(\mathcal{O}_{\mathfrak{X}, x})_{\et},
\mathbb{Q}/\mathbb{Z}(d)
)
\simeq 
\HO^{d}_{\et}(
\kappa(x),
\mathbb{Q}/\mathbb{Z}(d-1)
)
\end{align*}
(cf.\cite[p.31, Theorem 1.2]{SM}, \cite[Proposition 4.8]{Sak5}). 
So, if Kato conjecture (cf.\cite[p.125, Conjecture 0.3]{K-S}) holds, 
then the canonical map
\begin{equation*}
\HO^{d+1}_{\et}(
k(\mathfrak{X}),
\mathbb{Q}/\mathbb{Z}(d)
)    
\to
\displaystyle\bigoplus_{x\in \mathfrak{X}^{(1)}}
\HO^{d+1}_{\et}(
k(\mathcal{O}_{\mathfrak{X}, x}^{h}),
\mathbb{Q}/\mathbb{Z}(d)
)
\end{equation*}
is injective. 
\end{rem}
\end{document}